\documentclass{article}
\usepackage{amsmath,amssymb,amsfonts,amsthm}
\newtheorem{theorem}{Theorem}[section]

\newtheorem{definition}[theorem]{Definition}
\newtheorem{remark}[theorem]{Remark}

\newtheorem{lemma}[theorem]{Lemma}

\newtheorem{corollary}[theorem]{Corollary}

\newcommand{\F}{{\ensuremath{\mathbb{F}}}}

\include{PDF}

\makeindex                        

\title{Proportions of Cyclic Matrices in Maximal Reducible Matrix Groups and Algebras}
\author{Scott Brown, Cheryl E. Praeger, Michael Giudici}

\begin{document}

\maketitle

\footnotetext[1]{
The first author acknowledges support from a University Postgraduate Award from the University of Western Australia,
a Research Travel Award from Convocation, The University of Western Australia Graduates Association,
and also the Australian Research Council.
The second and third authors are supported by an ARC Federation Fellowship and an Australian Research Fellowship, respectively.}

\def\poly{polynomial }
\def\polys{polynomials }
\def\min{minimal }
\def\char{characteristic }
\def\la{\langle}
\def\ra{\rangle_A}
\def\li{linearly independent }
\def\ld{linearly dependent }
\def\bb{\mathbb}
\def\m{\mathcal}
\def\det{{\rm det}}
\def\deg{{\rm deg}}
\def\nulls {{\rm null}}
\def\dim{{\rm dim}}
\def\GL{{\rm GL}}
\def\M{{\rm M}}
\def\lcm{{\rm lcm}}
\def\gcd{{\rm gcd}}

\begin{abstract}
A matrix is said to be {\it cyclic} if its characteristic polynomial
is equal to its minimal polynomial. Cyclic matrices play an important role in some algorithms for matrix group computation, such as the {\sc Cyclic Meataxe} developed by P. M. Neumann and C. E. Praeger in 1999.
In that year also, G. E. Wall and J. E. Fulman independently found the limiting proportion of cyclic matrices in general linear groups over
a finite field of fixed order $q$ as the dimension $n$ approaches infinity, namely
$(1-q^{-5}) \prod_{i=3}^\infty (1-q^{-i}) = 1 - q^{-3} + O(q^{-4}).$
We study cyclic matrices in a maximal reducible matrix group or algebra, that is, in
the largest subgroup or subalgebra that leaves invariant some proper nontrivial subspace.
We modify Wall's generating function approach
to determine the limiting proportions of cyclic matrices
in maximal reducible matrix groups and algebras  over a field of order $q$,
as the dimension of the underlying vector space increases while that of the
invariant subspace remains fixed.
The limiting proportion in a maximal reducible group is proved to be
$1 - q^{-2} + O(q^{-3})$; note
the change of the exponent of $q$ in the second term of the expansion.
Moreover, we exhibit  in each maximal reducible matrix group a family of noncyclic matrices whose proportion
is $q^{-2} + O(q^{-3})$.
\end{abstract}

\begin{center}
\section*{\small{Acknowledgements}}
\end{center}
A special thank you to Prof. Peter Neumann for advice on this research during his time as Professor-at-Large in the Institute of Advanced Studies, The University of Western Australia, and later in Oxford. We are grateful in particular for his suggestions regarding an approach for
finding the leading terms of the limiting proportion in Theorem \ref{1-q^-2ch1} and also for
his advice on locating a family of noncyclic matrices.  Thanks also to The Queen's College at Oxford University,
and Prof. Peter Neumann in particular, for their hospitality during a visit by the first two authors to the United Kingdom in 2005.

\tableofcontents

\section{Introduction} \label{Intro}

This paper studies reducible matrix groups and matrix algebras over
a finite field,  say $\mathbb{F}_q$ of order $q$, that is,
some proper non-trivial subspace
of the underlying vector space is fixed setwise by all the elements.
If no such subspace exists the
group or algebra is said to be {\it irreducible}.
The first algorithm developed to test the irreducibility of a matrix
algebra was the Parker {\sc Meataxe} based on the
Norton Irreducibility Test described in \cite{parker}.
Given generators for a matrix group or algebra the Parker Meataxe sought matrices of a
special type by means of which it either proved irreducibility or
constructed a non-trivial proper invariant subspace. The (non-zero)
probability of failing to find suitable elements was not estimated.

The first analysis of a modified version of Parker's {\sc Meataxe}
was due to Holt and Rees \cite{holtrees} in 1994, and an
alternative modification of the irreducibility test using cyclic
matrices was developed by Neumann and Praeger
in  \cite{PNmeataxe}. Their  Cyclic Irreducibility Test  is a Las
Vegas algorithm that verifies irreducibility of an
irreducible matrix algebra by producing a cyclic matrix and a
corresponding cyclic basis. Bounds on the error probability rely on estimates for the
proportion of cyclic matrices in finite irreducible matrix algebras. Namely,
in \cite{PNcyclic} they proved that the
proportion of cyclic matrices in a full matrix algebra over
$\bb{F}_q$ is $1 - q^{-3} + O(q^{-4})$ and obtained also, explicit lower bounds for this
proportion in arbitrary irreducible matrix algebras, in \cite[Theorem 5.5]{PNcyclic}.
A motivation for the work presented in this paper was the problem of
extending the scope of the Cyclic {\sc Meataxe} to constructing an
invariant subspace in the case of `large' reducible matrix groups and algebras.
Using the results of this paper such an extension of the Cyclic  {\sc Meataxe}
was developed and presented in \cite[Chap 7]{myPhD}.

A matrix is called {\it cyclic} if its characteristic polynomial equals
its minimal polynomial.
G.E. Wall \cite{wall} and Jason Fulman \cite{Ful} independently
calculated generating functions for the proportion of cyclic matrices
in general linear groups and full matrix algebras, and determined
exactly the limiting proportions as the dimension tends to infinity, namely for finite
general linear groups $\GL(n,q)$ (see \cite[Equation 6.24]{wall}
or \cite[Theorem 8]{Ful}) the limit as $n$ tends to infinity is
\begin{equation}\label{wall-gl-limit}
 \frac{1 - q^{-5}}{1 + q^{-3}} = 1 - q^{-3} + O(q^{-5})
\end{equation}
while for finite matrix algebras $\M(n,q)$ (see
\cite[Equation 6.23]{wall} and \cite[Theorem 6]{Ful}),  the limit as
$n$ tends to infinity is
\begin{equation}\label{wall-mx-limit}
 (1 - q^{-5}) \prod_{i=3}^\infty (1 - q^{-i}) = 1 - q^{-3} + O(q^{-4}).
\end{equation}
The leading $q$ term, $q^{-3}$ in this case, is of significance
as for large $q$ this term dominates the later terms.
Generalising these results, proportions of cyclic matrices have been computed
for finite classical groups using both geometric methods \cite{PNclassical} and
generating function methods \cite{Brit, BritSL,Ful,FNP}.

Some matrix algebras contain no
cyclic matrices, for example, the algebra of $n \times n$ scalar matrices
for any $n \ge 2$, while others contain a large proportion.
One of the few results in the literature on proportions of cyclic matrices in reducible
matrix groups and algebras is due to
Jason Fulman \cite{Ful2}. Using cycle index methods he obtains the
limiting proportion of cyclic matrices
in the largest parabolic subgroup of a general linear group, namely
the stabiliser of a $1$-dimensional subspace.
We reprove his result by a different method, and extend it for all
subspace stabilisers in general linear groups and matrix algebras, yielding
Theorem~\ref{1-q^-2ch1}, the main result of this paper.

\begin{theorem}\label{1-q^-2ch1}
Let $r,n \in \mathbb{Z}^+$ with $r < n$, let $q$ be a prime power and let $c_{\GL,r}(n)$ and $c_{\M,r}(n)$ denote the proportions
of cyclic matrices in the stabiliser in $\GL(n,q)$ and $\M(n,q)$, respectively, of an $r$-dimensional subspace
of $\bb{F}_q^n$. Then
$$\displaystyle c_{\GL,r}(\infty) := \lim_{n\rightarrow\infty}c_{\GL,r}(n) = 1 - q^{-2} + O(q^{-3})$$
and
$$\displaystyle c_{\M,r}(\infty) := \lim_{n\rightarrow\infty}c_{\M,r}(n) = 1 - q^{-2} + O(q^{-3}).$$
Moreover, for any $d$ such that $1 < d < q(q-1)$, we have that
$|c_{\GL,r}(n) - c_{\GL,r}(\infty)| = O(d^{-n})$ and
$|c_{\M,r}(n) - c_{\M,r}(\infty)| = O(d^{-n}).$
\end{theorem}

The assertions about the limit and rate of convergence of the quantities
$c_{\GL,r}(n)$ are proved in Theorems \ref{limcrn} and \ref{1-q^-2} in Section \ref{sect4.2},
while the assertions for $c_{\M,r}(n)$ are proved in Theorems \ref{limcrMn} and \ref{1-q^-2b} in Section \ref{M(V)_U}.

Our approach is similar to that of
Wall \cite{wall} in that we study the generating function for
the proportions  $c_{\GL,r}(n)$ in order to calculate their limit as
$n$ approaches infinity.
However, instead of producing an explicit expression for the generating function as an infinite product,
we specify a finite series of steps involving partial differentiation and evaluation that produces the result.
The number of steps is linear in $r$, and so for large (but fixed) $r$ we do not write down an explicit form for
the generating function. Nevertheless the information given is sufficient to determine the limiting proportion
and rate of convergence as $n$ tends to infinity. In Section \ref{4.4} we demonstrate, for $r=1,2$ how
the procedure can be applied to give an explicit form for the generating function - the case of $r=1$
retrieving the result of Fulman \cite{Ful2}.

Note that the exponent of the leading $q$-term $-q^{-2}$ has increased by one
over that for the general linear group case. Note also that
this leading $q$-term is independent of the dimension $r$ of the invariant subspace.
Although this may suggest that the limiting proportion may be independent of the
dimension of the invariant subspace, this is not the case. As we discuss
in Remark \ref{remark4.4}, it is only the coefficient of $q^{-2}$ that
is independent of the dimension of the invariant subspace; the
coefficient of $q^{-3}$ changes for different values of $r$.

To better understand Theorem \ref{1-q^-2ch1}, we exhibit in Theorem \ref{noncyc}, a family
of noncyclic matrices in maximal reducible matrix groups with
proportion $q^{-2} + O(q^{-3})$.

Results about the limits and rates of convergence of the proportion of cyclic matrices inside
maximal completely reducible subgroups of $\GL(n,q)$ and
maximal completely reducible subalgebras of $\M(n,q)$
are found in \cite[Chapter 5]{myPhD},
while similar results for separable matrices are found in \cite[Chapter 6]{myPhD}.

\section{Preliminary Results}

In this section we prove some preliminary results.
Throughout this paper $\mathbb{F}_q$ denotes a field of order $q$,
and $V=\mathbb{F}_q^n$ denotes the vector space of $n$-dimensional
row vectors. The algebra of all $n\times n$ matrices over $\mathbb{F}_q$
is denoted $\M(n,q)$, and the general linear group of all nonsingular $n\times n$
matrices over $\mathbb{F}_q$ by $\GL(n,q)$. We sometimes write $\M(V)=\M(n,q)$ or $\GL(V)=\GL(n,q)$.
Each $A \in \M(n,q)$ acts naturally on $V$ and for $w \in V$, $\langle w \rangle_A$ denotes the cyclic
$A$-module generated by $w$, that is, $\langle w \rangle_A = \langle w, wA, wA^2, \ldots, wA^n \rangle$.
In particular if $\langle w \rangle = V$ then $w$ is called a {\it cyclic vector} for $A$
and $(w,A)$ is called a {\it cyclic pair} on $V$.
It is well known that $A$ is cyclic if and only if it has a cyclic vector.
If $W$ is an $A$-invariant subspace of $V$ then $A$ induces a matrix $A|_W$ in $\M(W)$ and
a matrix $A|_{V/W}$ in $\M(V/W)$. In particular the {\it nullspace} null$(A)$ is an $A$-invariant subspace.

We denote the characteristic and minimal polynomial of $A$ on $V$ by $c_A(t)$ and $m_A(t)$ respectively.

\subsection{Cyclic Matrices on Subspaces}

\begin{lemma} \label{conj}
Let $X \in \M(n,q)$. Then $X$, $X^T$ and all conjugates of
$X$ have the same characteristic polynomial and the same minimal
polynomial. In particular, they are either all non-cyclic or all
cyclic.
\end{lemma}

\begin{proof}
This is well known, see for example \cite[Theorem 7.2 and Exercise 2 on p149]{Fink}
for the characteristic polynomial and \cite[Exercise 4 on p150]{Fink} for the minimal polynomial
of similar, this is, conjugate matrices. The assertion for the transpose holds since for any polynomial $f(t)$, we have
$f(X)^T=f(X^T)$.
\end{proof}

\begin{lemma} \label{W}
Let $A$ be a cyclic matrix on $V = \mathbb{F}_q^n$
with \min polynomial $m_A(t) = f(t)g(t)$ for monic
polynomials $f(t)$ and $g(t)$, and let $w \in V$ be a cyclic vector for
$A$. If $W = \langle w f(A) \rangle_A$ then

$(1)$ the minimal polynomial of $A|_W$ is $g(t)$;

$(2)$ $W$ has dimension $n - \deg(f) = \deg(g)$;

$(3)$ $W = {\rm null} (g(A))$.
\end{lemma}

\begin{proof}
$(1)$ Since $(wf(A))g(A) = 0$,
the minimal polynomial $h(t)$ of $A|_W$ divides $g(t)$.
Suppose that $\deg(h) < \deg(g)$.
Then $wf(A)A^ih(A) = 0$ for all $i$, and this implies that $vf(A)h(A)
= 0$ for all $v \in \langle w \rangle_A = V$. However $f(t)h(t)$ is a
polynomial of degree strictly less than $\deg(m_A)$, and we have a contradiction.
Hence $\deg(h) \ge \deg(g)$, and since $h(t)$ divides $g(t)$ it follows that $h(t)=g(t)$.

$(2)$ Since the \min \poly of $A|_W$ has degree $n - \deg(f)$ and the
space $W$ is cyclic as an $A$-module, it follows that $W$ has
dimension $n - \deg(f)$.

$(3)$ Let $u \in \langle wf(A) \rangle_A$. There exists a positive integer $k$ such
that the vectors $wf(A), wf(A)A, \ldots, wf(A)A^k$ form a basis for $\langle wf(A)
\rangle_A$ and so $u = \lambda_0wf(A) + \lambda_1wAf(A) + \cdots +
\lambda_kwA^kf(A)$ for some $\lambda_0, \ldots, \lambda_k \in \mathbb{F}_q$.
Then $ug(A) = \lambda_0wf(A)g(A) + \cdots
+ \lambda_kwf(A)A^kg(A)$ and this equals $0$ since every term is a multiple of
$wg(A)f(A) = wm_A(A) = 0$. Hence $u \in {\rm null} (g(A))$ and $\langle
wf(A) \rangle_A \subseteq {\rm null} (g(A))$.
In particular $\dim(\nulls(g(A))) \ge n - \deg(f) = \deg(g)$.

Let us now look at the rank of $g(A) = \sum_{i=0}^d a_iA^i$ where
$g(t) = \sum_{i=0}^d a_i t^i$ and $d$ is the degree of $g$.
Note in the following that $a_d = 1$ since $g(t)$ is monic.

$$ w   g(A)    = w A^d     + \sum_{i=0}^{d-1} a_i w A^i    $$
$$ wA  g(A)    = w A^{d+1} + \sum_{i=0}^{d-1} a_i w A^{i+1}$$
$$ \vdots $$
$$ wA^{n-{d+1}}g(A) = w A^{n-1} + \sum_{i=0}^{d-1} a_i w A^{i+n-d-1}$$

\noindent
Now since $wA^d, wA^{d+1}, \ldots, wA^{n-1}$ are linearly independent,
it follows from the equations above that
$wg(A), wA g(A), \ldots, wA^{n-d-1} g(A)$ are
linearly independent. So rank$(g(A)) \ge n-d$ which means that
$\dim({\rm null}(g(A))$) $\le d$.
By the previous paragraph equality holds and hence
$\langle wf(A) \rangle_A = {\rm null}(g(A))$.
\end{proof}

\begin{lemma} \label{cycA|U}
Let $(w,A)$ be a cyclic pair on $V=\mathbb{F}_q$,
let $U$ be an $A$-invariant $r$-dimensional subspace of $V$ and
let $c(t)$ be the characteristic polynomial of $A|_{V/U}$. Then $c(t)$ divides $c_A(t)$,

$(1)$ $U+w$ is a cyclic vector for $A|_{V / U}$;

$(2)$ $w c(A)$ is a cyclic vector for $A|_U$.
\end{lemma}

\begin{proof}
That $c(t)$ divides $c_A(t)$ is well known.
By assumption $\langle w \rangle_A = V$, and hence $\langle U+w \rangle_{A|_{V/U}}=V/U$, proving $(1)$.

By definition, $c(A|_{V/U})$ is the zero matrix of $V/U$, that is to say $c(t)$ is monic and $c(A)|_{V/U}=0$.
Thus $wc(A) \in U$. Since $\deg(c)=\dim(V/U)=n-r$, we have
$c(t) = t^{n-r} + \sum_{i=0}^{n-r-1} a_i t^i$ and
$wc(A) = wA^{n-r} + \sum_{i=0}^{n-r-1} a_i w A^i$.
Since $(w,A)$ is a cyclic pair, the vectors $wA, \ldots, wA^{n-r}$ are linearly independent.
Thus arguing as in the previous proof we see that the vectors $wc(A), \ldots, wc(A)A^{r-1}$ are also
linearly independent, and each lies in $U$ since $U$ is $A$-invariant. They therefore form a basis for $U$ since $\dim(U)=r$.
Hence $\langle wc(A) \rangle_A = U$ and $(2)$ is proved.
\end{proof}

The next lemma describes all the
invariant subspaces of a cyclic matrix whose characteristic
polynomial has only one irreducible factor.

\begin{lemma} \label{chain}
Let $(w,A)$ be a cyclic pair on $V=\mathbb{F}_q^n$.
Then the only $A$-invariant subspaces $U$ of $V$
are $U=\langle w f(A) \rangle_A$ for some monic polynomial $f(t)$ dividing $m_A(t)$
and the minimal polynomial of $A|_U$ on $U$ is $g(t)$ where $m_A(t) = f(t) g(t)$.
Moreover, there is a one-to-one correspondence between
$A$-invariant subspaces of $V$ and monic divisors of $m_A(t)$.
\end{lemma}

\begin{proof}
Let $U$ be an $A$-invariant subspace of $V$. By Lemma \ref{cycA|U},
$U+w$ is a cyclic vector for $A|_{V/U}$, and $wf(A)$ is a cyclic vector for $A|_U$, where $f(t)=c_{A|_{V/U}}(t)=m_{A|_{V/U}}(t)$
and $f(t)$ divides $m_A(t)$, say $m_A(t)=f(t)g(t)$.
Thus $\langle wf(A) \rangle_A = U$ and $\dim(U)=n-\deg(f)=\deg(g)$.
If also $U=\langle w f'(A) \rangle_A$ where $m_A(t)=f'(t)g'(t)$, then by Lemma \ref{W} $(3)$, $U= \nulls(g(A))= \nulls(g'(A))$ and $\deg(g)=\deg(g')$.
This implies that $wf(A)g'(A)=0$ whence $m_A(t)$ divides $f(t)g'(t)$, or equivalently $g(t)$ divides $g'(t)$. Thus $g(t)=g'(t)$ and hence for distinct divisors $f(t),f'(t)$ of
$m_A(t)$ the submodules $\langle wf(A) \rangle_A$ and $\langle wf'(A) \rangle_A$ are distinct. Also $g(t) = m_{A|_U}(t)$.
\end{proof}

\subsection{Decomposition} \label{2.2}

For $X \in \M(n,q)$, a \emph{cyclic decomposition} of $V=\mathbb{F}_q^n$ under $X$
is a direct sum $V = \oplus_{i=1}^\ell V_i$ where each
$V_i$ is $X$-cyclic and $X$-invariant,
the minimal polynomial of $X|_{V_i}$ divides the minimal polynomial of $X|_{V_{i-1}}$
for $1 < i \le \ell$, and the minimal polynomial of $X|_{V_1}$ is the minimal polynomial of $X$
(see \cite[Lemma 11.7]{H&H}). If $X$ is cyclic, the cyclic decomposition of
$V$ under $X$ is simply $V = V_1$.

Let $A \in \M(n,q)$ be cyclic with minimal polynomial
$m_A(t) = \prod_{i=1}^k f_i^{\alpha_i}(t)$ where the $f_i$ are distinct
monic irreducible polynomials and each $\alpha_i > 0$.
The \emph{primary decomposition} of $V$ under $A$ is
$V = \oplus_{i=1}^k V_i$ where for each $i$ we have
$V_i = \langle w g_i(A) \rangle_A$ for
$g_i(t) = \frac{m(t)}{f_i^{\alpha_i}(t)}$.
With respect to an appropriate basis, $A$ can be decomposed into blocks
or components as below:

\begin{displaymath}
\left(
\begin{array}{c c c c c}

A_1 &     &        &	 \\
    & A_2 &        &	 \\
    &     & \ddots &	 \\
    &     &        & A_k \\

\end{array}
\right)
\quad \mbox {which we write as} \quad
A_1 \oplus \cdots \oplus A_k
\end{displaymath}

\noindent
where each $A_i$ is cyclic on $V_i$ with
minimal polynomial $f_i^{\alpha_i}$ (see \cite[Lemma 11.8]{H&H}).

The proof of the following lemma is straight forward and is omitted.

\begin{lemma} \label{+ViB}
Let $A \in \M(n,q)$ be a cyclic matrix on $V=\mathbb{F}_q^n$ and
let $V = \oplus V_i$ be its cyclic (respectively primary) decomposition,
and let $B \in \GL(V)$. Then

$(a)$ The conjugate $B^{-1}AB$ has cyclic (respectively primary) decomposition
$V=\oplus (V_iB)$.

$(b)$ If $U_i$ is an $A$-invariant subspace of
$V_i$ then $U_iB$ is a $B^{-1}AB$-invariant subspace of $V_iB$ of the
same dimension and $m_{A|_{U_i}}(t)$ is the minimal polynomial of $(B^{-1}AB)|_{U_iB}$.

$(c)$ Moreover, $A$ and $A' \in \M(n,q)$ are conjugate under an element of $\GL(n,q)$
if and only if the sequences of minimal polynomials
induced on the spaces in their cyclic decompositions are the same.
\end{lemma}

Note that if $\oplus_{i=1}^s V_i$ is the cyclic decomposition for $A \in \M(n,q)$ and
$m_{A|_{V_i}}(t) = \sum_{j=0}^{d_i} a_{ij} t^j$ with $a_{id_i}=1$, then $A$ is conjugate to $A_1 \oplus \cdots \oplus A_s$
where for each $i$, $A_i$ is conjugate under $\GL(V_i)$ to the companion matrix

\begin{displaymath}
\left(
\begin{array}{c c c c c c}
0 & 1 & 0 & \ldots & 0 \\
0 & 0 & 1 & \ldots & 0 \\
\vdots &   &   & \ddots & \vdots \\
0 & 0 & 0 & \ldots & 1 \\
-a_{i0} & -a_{i1} & -a_{i2} & \ldots & -a_{i(d_i-1)} \\
\end{array}
\right)
\end{displaymath}

\begin{corollary} \label{min=conjclasses}
Cyclic matrices $X$ and $Y$ in $\M(n,q)$ have the same minimal polynomial if and
only if they are conjugate in $\GL(n,q)$.
\end{corollary}

This corollary follows from the fact that the cyclic decomposition
of a vector space under a cyclic matrix has just one component. It implies that there is exactly one
$\GL(n,q)$-conjugacy class of cyclic matrices with a given
characteristic polynomial.

\begin{lemma} \label{cyccoprime}
Let $A \in \M(n,q)$ acting on $V=\mathbb{F}_q^n$ and suppose that
$V = \oplus_{i\le k} V_i$ where each $V_i$ is
$A$-invariant. Let $c_i(t)$ and $m_i(t)$
be the characteristic and minimal polynomials respectively of $A$ on
$V_i$ for $i \le k$. Then $A$ is cyclic if and only if each $A|_{V_i}$ is
cyclic and $\gcd(c_i, c_{i'}) = 1$ for all $i \ne i'$.
\end{lemma}

\begin{proof}
Suppose that each $A|_{V_i}$ is cyclic and that $\gcd(c_i, c_{i'}) = 1$
for all $i \ne i'$. Then $m_i =
c_i$ for all $i$ so $\lcm(m_1, \ldots, m_k) = \lcm(c_1, \ldots,
c_k)=f$, say. Since $\gcd(c_i, c_{i'}) = 1$ for all $i \ne i'$ we have
that $f = c_1\ldots c_k = c$ and we know that $\lcm(m_1, \ldots, m_k) = m$. Hence $c = m$
and $A$ is cyclic.

Conversely suppose that $A$ is cyclic. By Lemma \ref{cycA|U}, each of the $A|_{V_i}$ is cyclic.
Suppose that $\gcd(c_i, c_{i'}) \ne
1$ for some $i \ne i'$. Then $f(t) := \lcm(c_1, \ldots, c_k)$
divides $\frac{c_1\ldots c_k}{\gcd (c_i, c_{i'})}$ so $\deg(f)$ $< n$ and each $m_i$ divides $c_i$ which, in turn, divides $f$.
Each vector $v \in V$ can be written as $v = v_1+\cdots +v_k$ for
some $v_i \in V_i$ for $i \le k$. Then $vf(A) = v_1f(A) + \cdots +
v_kf(A) = 0$ since each $m_i$ divides $f$. Hence $m$
divides $f$ which has degree less than $n$ so $m \ne c$ which
is a contradiction with $A$ being cyclic on $V$. Hence $\gcd(c_i,
c_{i'}) = 1$ for all $i \ne i'$.
\end{proof}

\begin{lemma} \label{UinV}
Let $A$ be a cyclic matrix on $V$ with minimal polynomial $\prod
f_i^{\alpha_i}$ and let $V = \oplus V_i$ be the primary decomposition
of $V$ under $A$. Let $U$ be an $A$-invariant subspace of $V$. Then
the primary decomposition for $U$ under $A|_U$ is
$\oplus U_i$, where for each $i$, $U_i=U\cap V_i$.
\end{lemma}

\begin{proof}
Let $v$ be a cyclic vector for $A$ in $V$.
Using the notation introduced at the beginning of Section \ref{2.2}, let
$c_A(t) = f_i^{\alpha_i}(t) g_i(t)$ for each $i$, so that by Lemma \ref{chain},
$V_i = \langle v_i \rangle_A$ where $v_i = v g_i(A)$.
Let $c_{A|_U}(t) = \prod f_i^{\beta_i}(t)$, and let $c_A(t) = c_{A|_U}(t) g(t)$ and
$c_{A|_U}(t) = f_i^{\beta_i}(t) g'_i(t)$ for each $i$, where $f_i$ does not divide $g'_i$.
Note that $g_i(t)$ divides $g'_i(t) g(t)$ for each $i$.

Again by Lemma \ref{chain}, $U = \langle u \rangle_A$ where $u=vg(A)$
and $U_i := \langle u g'_i(A) \rangle_A$ is $A$-invariant with $c_{A|_{U_i}}(t)=f_i^{\beta_i}(t)$.
Thus $U = \oplus_{i=1}^s U_i$ is the primary decomposition of $U$ under $A|_U$,
where $U_i$ may be $0$ for some $i$.
Also, since $g_i$ divides $g'_i g$ it follows that
$u g'_i(A) = v g(A) g'_i(A) \in \langle v g_i(A) \rangle_A = V_i$,
so $U_i \subseteq V_i \cap U$.
Then $\oplus_{j \ne i} U_j \subseteq \oplus_{j \ne i} V_j$
and so is disjoint from $V_i \cap U$.
This implies that $U_i = V_i \cap U$ for each $i$.
\end{proof}

We are working with matrices on $V$ which fix a proper subspace of $V$
so we need convenient notation for the
group and the algebra of such matrices.
For $V=\F_q^n$,
denote $\GL(V):=\GL(n,q)$ and $\M(V):=\M(n,q)$.
For a subspace $U \le V$, let $\GL(V)_U$ be the subgroup of $\GL(V)$ which consists of
all matrices that fix $U$ setwise.
Similarly let $\M(V)_U$ be the set of all matrices
in $\M(V)$ which leave the subspace $U$ invariant.

\begin{lemma} \label{m,m=conjclasses}
Let $A, A'$ be cyclic matrices in $\M(V)_U$. Then $A$ and $A'$
are conjugate by an element of $\GL(V)_U$ if and only if both $m_A=m_{A'}$ and $m_{A|_U}=m_{A'|_U}$.
\end{lemma}

\begin{proof}
If $A$ and $A'$ are $\GL(V)_U$-conjugate then also $A|_U$ and $A'|_U$ are $\GL(U)$-conjugate, and hence
by Lemma \ref{conj}, $c_A=c_{A'}$ and $c_{A|_U}=c_{A'|_U}$. By Lemma \ref{cycA|U}, $A|_U$ and $A'|_U$ are
cyclic also, and hence $m_A=c_A=c_{A'}=m_{A'}$ and similarly $m_{A|_U}=m_{A'|_U}$.

Conversely suppose that $m_A(t) = m_{A'}(t)$ and $m_{A|_U}(t) = m_{A'|_U}(t)$.
Then by Corollary \ref{min=conjclasses}, there exists $X \in \GL(V)$ such that $X^{-1} A X = A'$.
By assumption $U$ is an $A'$-invariant subspace and the minimal polynomial of
$A'|_U$ is $m_{A'|_U}=m_{A|_U}$. We also have that $UX$ is $A'$-invariant
(since $A'=X^{-1}AX$)
and the minimal polynomial of $A'$ on $UX$ is $m_{A|_U}$.
By Lemma \ref{chain}, it follows that $U=UX$ since the minimal polynomial of
$A'$ on each space is the same. Hence $X \in \GL(V)_U$.
\end{proof}

This lemma characterises the set of conjugacy classes of cyclic matrices
in $\M(V)_U$ under the action of
$\GL(V)_U$. There is one conjugacy class for each pair
$f,h$ where $h$ is a monic degree $n$ polynomial, $f$ is a monic
degree $r$ polynomial dividing $h$ and $r = \dim (U)$.
For matrices in $\GL(V)_U$, we in addition require $f$ and $h$ to have
a nonzero constant term.

\subsection{Centralisers} \label{BG-Cent}

This section describes the centralisers of cyclic
matrices in $\GL(V)_U$. For $A\in \M(V)$, let $C_{\GL(V)}(A)$ denote the centraliser of $A$ in $\GL(V)$.

\begin{lemma} \label{cent}
Let $V=\mathbb{F}_q^n$, $U$ be a subspace of $V$, and $A \in \M(V)_U$
be cyclic. Then $C_{\GL(V)}(A) \leq \GL(V)_U$.
\end{lemma}

\begin{proof}
Let $B \in C_{\GL(V)}(A)$, so $B^{-1} A B = A$.
Let $w$ be a cyclic vector for $A$.
By Lemma \ref{chain}, $U=\langle wf(A) \rangle_A$ for some monic $f(t)$ dividing $m_A(t)$
and the minimal polynomial of $A$ on $U$ is $g(t)$ where $m_A(t)=f(t)g(t)$.
The subspace $UB$ is also invariant under $B^{-1}AB = A$ and, since $Ug(A)=0$,
$UB$ is annihilated by $B^{-1} g(A) B = g(B^{-1}AB) = g(A)$.
Hence $UB \subseteq {\rm null}(g(A))$ and by Lemma \ref{W}(3),
${\rm null}(g(A)) = \langle wf(A) \rangle_A =U$.
Then since $\dim (UB) = \dim (U)$ we have that $UB=U$.
Hence $B$ fixes $U$ and $B \in \GL(V)_U$.
\end{proof}

A proof of the next lemma may be found in \cite[Corollary 2.3 and Remark 2.6]{PNcyclic}.

\begin{lemma} \label{same_cent}
Let $n=ij$ and suppose that $p$ is an irreducible polynomial of degree $i$ over
$\mathbb{F}_q$, and $A \in \M(n,q)$ is a
cyclic matrix with characteristic polynomial $p^j$. Then
$|C_{\GL(n,q)}(A)| = q^{ij}(1-q^{-i})$.
\end{lemma}

The statement of Lemma \ref{same_cent} is for an arbitrary irreducible
degree $i$ polynomial. Therefore the size of the centraliser of any matrix
with characteristic polynomial the $j$th power of irreducible degree $i$ polynomial,
depends only on $i$ and $j$, and we set

\begin{equation} \label{Centij}
Cent(i,j) := q^{ij} (1-q^{-i})
\end{equation}

\noindent
so that, by Lemma \ref{same_cent}, $Cent(i,j) = |C_{\GL(V)}(A)|$ for
any cyclic matrix $A$ with characteristic polynomial $p^j$ for
some monic degree $i$ irreducible polynomial $p$.
We will often write $Cent(p^j)$ in this instance also.

\begin{lemma}\label{Cent}
Let $f$ and $g$ be distinct monic irreducible polynomials and let
$a, b \in \mathbb{Z}^+$. Then $Cent(f^a) Cent(g^b) = Cent(f^ag^b).$
\end{lemma}

\begin{proof}
Let $A$ be a cyclic matrix over a vector space $V$ with minimal
polynomial $f^ag^b$. Then the primary decomposition of $V$ under $A$ is
$V = V_1 \oplus V_2$ where $A|_{V_1}$ is cyclic
with minimal polynomial $f^a$ and $A|_{V_2}$ is cyclic
with minimal polynomial $g^b$.

By Lemma \ref{cent}, $C_{\GL(V)}(A) \subseteq \GL(V)_{V_1} \cap \GL(V)_{V_2} \cong \GL(V_1) \times \GL(V_2)$
and hence $C_{\GL(V)}(A) \cong C_{\GL(V_1)}(A|_{V_1}) \times C_{\GL(V_2)}(A|_{V_2})$.
It follows that
$Cent(f^a g^b) = Cent(f^a) Cent(g^b)$.
\end{proof}

An analogous proof gives the following general result.

\begin{corollary} \label{Cent(p_i)}
Let $f = \displaystyle \prod_{i=1}^k f_i^{\alpha_i}$ where the $f_i$ are
pairwise distinct monic irreducible polynomials of degree $d_i$. Then
$$ Cent(f)= \prod_{i=1}^k Cent(f_i^{\alpha_i}) = \prod_{i=1}^k Cent(d_i,
\alpha_i).$$
\end{corollary}

\section{The work of Wall}

In this section we outline G.E. Wall's \cite{wall} generating function
method for determining the limiting proportion of cyclic matrices
in $\GL(V)$ and $\M(V)$. Our notation is based on that used in \cite{wall}.

\subsection{Cyclic Matrices in General Linear Groups} \label{GL(V)}

Let

$$C_{\GL}(t) = 1 + \sum_{n=1}^\infty c_{\GL}(n) t^n$$

\noindent
be the generating function for the proportion of cyclic matrices in
$\GL(V)$ where for each $n$, $V=\mathbb{F}_q^n$ and $c_{\GL}(n)$ denotes
the proportion of cyclic matrices in $\GL(V)$.

Let $\Gamma_{\GL}(n)$ be the set of all cyclic matrices in $\GL(V)$. Then by
Lemma \ref{conj} there is a one-to-one correspondence between the
set of orbits of $\GL(V)$ in its action on $\Gamma_{\GL}(n)$ by conjugation and
the set of monic degree $n$ polynomials $h$ over $\mathbb{F}_q$ such that
$h(0) \ne 0$. Denote the orbit on $\Gamma_{\GL}(n)$ consisting of all
matrices with minimal polynomial $h$ by $\Gamma_h$. By the Orbit-Stabiliser
Theorem

$$|\Gamma_h| = \frac{|\GL(V)|}{Cent(h)}$$
\noindent
where $Cent(h) = |C_{\GL(V)}(A)|$ for a matrix $A \in \Gamma_h$.
To enable us to focus on irreducible polynomials of a given degree, we make
the following definition.

\begin{definition} \label{P}
{\rm Let $\mathcal{P}^+$ be the set of all monic polynomials over
$\mathbb{F}_q$ with a nonzero constant term (including $1$) and let
$\mathcal{P}^+_i$ be the subset of $\mathcal{P}^+$ containing the
constant polynomial $1$ and those polynomials whose irreducible
factors all have degree $i$.}
\end{definition}

In light of this definition it follows that

\begin{equation} \label{cn}
c_{\GL}(n) =
\sum_{
\begin{array} {c}
\vspace{-0.2cm}
\textrm{\scriptsize $h$$\in$$\mathcal{P}^+$} \\
\vspace{-0.2cm}
\textrm{\scriptsize $\deg(h)$$=$$n$} \\
\end{array}
}
\frac{|\Gamma_h|}{|\GL(V)|}
= \sum_{
\begin{array} {c}
\vspace{-0.2cm}
\textrm{\scriptsize $h$$\in$$\mathcal{P}^+$} \\
\vspace{-0.2cm}
\textrm{\scriptsize $\deg(h)$$=$$n$} \\
\end{array}
}
\frac{1}{Cent(h)}.
\end{equation}

\noindent

For each monic irreducible polynomial $p$ of degree $i$,
define the formal power series

\begin{equation} \label{F_p}
F^p = F^p(t) := 1 + \sum_{j=1}^\infty \frac{t^{ij}}{Cent(p^j)}.
\end{equation}

\noindent
Using (\ref{cn}), the
product of the $F^p$ over all monic irreducible polynomials $p$ excluding $t$ is

\begin{equation} \label{bef312}
\prod_{p \ne t} F^p =  \sum_{h \in \mathcal{P}^+}
\frac{t^{\deg(h)}}{Cent(h)} = 1 + \sum_{n=1}^\infty c_{\GL}(n) t^{n} =
C_{\GL}(t).
\end{equation}

\begin{definition}\label{N}
{\rm Let $N(i,q)$ be the number of monic degree $i$
irreducible polynomials over $\mathbb{F}_q$ and let $N^+(i,q)$ be the
number of polynomials in $N(i,q)$ with nonzero constant term.}
\end{definition}

Thus $N^+(i,q) = N(i,q)$ if $i=1$ and $N(1,q) = N^+(1,q) + 1 = q$.
Note that $F^p(t)$ depends only on the degree and multiplicity of the irreducible factors of $p$.
Let $F_i(t)$ denote the following formal power series.

\begin{equation} \label{F_i}
F_i = F_i(t) := \prod_{
\begin{array}{c}
\vspace{-0.2cm}
\textrm{\scriptsize $deg(p)$$=$$i$} \\
\vspace{-0.2cm}
\textrm{\scriptsize $p$$\ne$$t$} \\
\end{array}
}
F^p(t)
= \left(1+ \sum_{j=1}^\infty \frac{t^{ij}}{Cent(i,j)}\right)^{N^+(i,q)}.
\end{equation}

\noindent
It follows that

\begin{equation} \label{C}
C_{\GL}(t) = \prod_{p\ne t} F^p = \prod_{i=1}^\infty F_i.
\end{equation}

The following lemma underpins the proof of Wall's convergence result in Theorem \ref{limpropGL(V)}
and we will also use it in Section \ref{chap4}.

\begin{lemma} \label{1-t}
Suppose that $g(u) = \sum_{n\ge 0} a_n u^n$ and $g(u) = f(u)/(1-u)$ for
$|u|<1$. If $f(u)$ is analytic in a disc of radius $R$, where $R>1$, then
$\lim_{n\rightarrow \infty} a_n = f(1)$ and $|a_n -f(1)| = O(d^{-n})$ for
any $d$ such that $1 < d < R$.
\end{lemma}

Wall proved that $C_{\GL}(t)$ is convergent for $|t| < 1$ and that
$(1-t) C_{\GL}(t)$ is convergent for $|t| < q^2$. Hence
by Lemma \ref{1-t} the limit of the coefficients of $C_{\GL}(t)$ is equal to
$(1-t) C_{\GL}(t)$ evaluated at $t=1$.
This, and the rate of convergence, was calculated by Wall and is given in Theorem \ref{limpropGL(V)}.

\begin{theorem} \label{limpropGL(V)}
$\lim_{n\to\infty} c_{\GL}(n) = c_{\GL}(\infty)$ exists and satisfies

$$ c_{\GL}(\infty) = \frac{1 - q^{-5}}{1 + q^{-3}} = 1 - q^{-3} + O(q^{-5}) $$

\noindent
and $|c_{\GL}(n) - c_{\GL}(\infty)| = O(d^{-n})$ for any $d$ with $1 < d < q^{-2}$.

\end{theorem}

\subsection{Cyclic Matrices in Full Matrix Algebras} \label{M(V)}

Let $\Gamma_\M(n)$ be the set of all cyclic matrices in $\M(V)$, where $V=\F_q^n$. Then by
Lemma \ref{conj} there is a one-to-one correspondence between the
set of orbits of $\GL(V)$ in its action on $\Gamma_\M(n)$ by conjugation and
the set of monic
degree $n$ polynomials over $\mathbb{F}_q$.
Denote the orbit on $\Gamma_\M(n)$ consisting of all matrices with minimal
polynomial $h$ by $\Gamma_h$. Note that $\Gamma_{\GL}(n) \subset \Gamma_\M(n)$ and that
when $h$ has nonzero constant term, the orbit $\Gamma_h$ defined in Section \ref{GL(V)} is the same as the orbit
$\Gamma_h$ we use here. By the Orbit-Stabiliser Theorem

$$|\Gamma_h| = \frac{|\GL(V)|}{Cent(h)}$$
\noindent

\begin{definition} \label{Q}
{\rm Let $\mathcal{P}$ be the set of all monic polynomials over
$\mathbb{F}_q$ (including $1$) and let
$\mathcal{P}_i$ be the subset of $\mathcal{P}$ consisting of the
constant polynomial $1$ and all polynomials whose irreducible
factors all have degree $i$.}
\end{definition}

Recall that $\mathcal{P}^+_i$ consists of $1$ and all monic polynomials
with nonzero constant term whose irreducible factors all have degree $i$.
Note that $\mathcal{P}^+_i = \mathcal{P}_i$ for $i \ge 2$
but we use the notation for convenience.

\begin{definition} \label{omega}
{\rm Let $\omega(n) = \frac{|\GL(n,q)|}{|\M(n,q)|}$.}
\end{definition}

It is easy to show that

\begin{equation} \label{omegaproof}
\omega(n) = \prod_{i=1}^n(1-q^{-i}).
\end{equation}

Denote by $c_\M(n)$ the proportion of cyclic matrices in $\M(V)$.
It follows that

\begin{equation} \label{cMn}
c_\M(n) = \sum_{
\begin{array} {c}
\vspace{-0.2cm}
\textrm{\scriptsize $h$$\in$$\mathcal{P}$} \\
\vspace{-0.2cm}
\textrm{\scriptsize $\deg(h)$$=$$n$} \\
\end{array}
}
\frac{|\Gamma_{h}|}{|\M(V)|}
= \sum_{
\begin{array} {c}
\vspace{-0.2cm}
\textrm{\scriptsize $h$$\in$$\mathcal{P}$} \\
\vspace{-0.2cm}
\textrm{\scriptsize $\deg(h)$$=$$n$} \\
\end{array}
}
\frac{|\Gamma_{h}||\GL(V)|}{|\M(V)||\GL(V)|}
= \sum_{
\begin{array} {c}
\vspace{-0.2cm}
\textrm{\scriptsize $h$$\in$$\mathcal{P}$} \\
\vspace{-0.2cm}
\textrm{\scriptsize $\deg(h)$$=$$n$} \\
\end{array}
}
\frac{\omega(n)}{Cent(h)}
\end{equation}

\noindent
and hence

$$ \frac{c_\M(n)}{\omega(n)} =
\sum_{
\begin{array} {c}
\vspace{-0.2cm}
\textrm{\scriptsize $h$$\in$$\mathcal{P}$} \\
\vspace{-0.2cm}
\textrm{\scriptsize $\deg(h)$$=$$n$} \\
\end{array}
}
\frac{1}{Cent(h)} $$

\noindent
which can be calculated by a similar method to that used for calculating $c_{\GL}(n)$.

The following is a `weighted' generating function for the proportion of cyclic matrices
in full matrix algebras over $\mathbb{F}_q$, and is considered by Wall:

\begin{equation} \label{CM}
C_\M(t) = 1 + \sum_{n=1}^\infty \left(\frac{c_\M(n)}{\omega(n)}\right) t^n.
\end{equation}

The product of the functions $F_i$ defined in (\ref{F_i}) gave

$$ \prod_i F_i(t) = \sum_{h \in \mathcal{P}^+}
\frac{t^{\deg(h)}}{Cent(h)}$$

\noindent
but in Section \ref{GL(V)}, $F_1(t)$ did not include the power series for the polynomial $t$.
Including the polynomial $t$ in Equation (\ref{bef312}) gives

$$ F^t(t) \prod_i F_i(t) = \sum_{h \in \mathcal{P}}
\frac{t^{\deg(h)}}{Cent(h)} = 1 + \sum_{i=1}^\infty \left(\frac{c_\M(n)}{\omega(n)}\right) t^n = C_\M(t). $$

\noindent

It follows that

\begin{displaymath}
\begin{array}{rclr}
C_\M(t) & = & \displaystyle \sum_{h \in \mathcal{P}}
\frac{t^{\deg(h)}}{Cent(h)} \\
& = & F^{t}(t) \prod_i F_i(t) \\
& = & \left( 1 + \sum_{j=1}^\infty \frac{t^{j}}{Cent(t^j)} \right) C_{\GL}(t)
& (\textrm{by (\ref{C})}) \\
& = & \left( 1 + \sum_{j=1}^\infty \frac{t^j}{q^j(1-q^{-1})}  \right) C_{\GL}(t) \\
& = & \left( 1 + \frac{t}{(q-1)(1-tq^{-1})} \right) C_{\GL}(t). \\
\end{array}
\end{displaymath}

To calculate the limiting proportion of the coefficients
of $C_\M(t)$ as $n$ tends to infinity, we use Lemma \ref{1-t} to get

\begin{displaymath}
\begin{array} {rcl}
\displaystyle
\lim_{n\to\infty} \frac{c_\M(n)}{\omega(n)} & = & \left( C_\M(t) (1-t) \right) |_{t=1} \\
& = & \left( C_{\GL}(t) (1-t) \left( 1 + \frac{t}{(q-1)(1-tq^{-1})} \right) \right) |_{t=1} \\
& = & \frac{1 - q^{-5}}{1 + q^{-3}}\times
\frac{1+q^{-3}}{(1-q^{-1})(1-q^{-2})} \\
& = & \frac{1-q^{-5}}{(1-q^{-1})(1-q^{-2})} \\
\end{array}
\end{displaymath}

\noindent
where Theorem \ref{limpropGL(V)} and
\cite[Lemma 2.5.2]{myPhD} are used to justify the second last line.
Since $\lim_{n\to\infty} \omega (n)$ exists, the limit of $c_\M(n)$ as $n$ tends to infinity is

\begin{displaymath}
\begin{array} {rcl}
\displaystyle
\lim_{n\to\infty} c_\M(n) & = & \displaystyle \lim_{n\to\infty} \omega(n)
\times \frac{1-q^{-5}}{(1-q^{-1})(1-q^{-2})} \\
& = & \displaystyle \prod_{i=1}^\infty (1-q^{-i}) \times
\frac{1-q^{-5}}{(1-q^{-1})(1-q^{-2})}\\
& = & (1-q^{-5}) \displaystyle \prod_{i=3}^\infty (1-q^{-i}). \\
\end{array}
\end{displaymath}

\noindent
Hence we have the following theorem, proved by Wall \cite[Equation 6.23]{wall}.
Wall also gave information about the rate of convergence of the $c_{\M}(n)$ to $c_{\M}(\infty)$.

\begin{theorem} \label{limpropM(V)}
$\lim_{n\to\infty} c_{\M}(n) = c_{\M}(\infty)$ exists and satisfies

$$c_\M(\infty) = (1-q^{-5}) \displaystyle \prod_{i=3}^\infty (1-q^{-i}) = 1 - q^{-3} + O(q^{-4})$$
and $|c_{\M}(n) - c_{\M}(\infty)| = O(d^{-n})$ for any $d$ with $1 < d < q^{-2}$.

\end{theorem}

\section{Maximal Reducible Groups and Algebras} \label{chap4}

This section is devoted to cyclic matrices in maximal reducible matrix
groups. We calculate the generating function for the proportions of cyclic
matrices inside such groups and we calculate the limiting proportion
of cyclic matrices as the size of the matrices tends
to infinity. Then we modify our procedures to
calculate the generating function and limiting proportion of cyclic
matrices in the corresponding maximal reducible matrix algebras.

\subsection{The Generating Function} \label{GL(V)_U}

Let $V = \F_q$ and let $U$ be an
$r$-dimensional subspace of $V$. We investigate cyclic matrices in
the setwise stabiliser, $\GL(V)_U$, of $U$ in $\GL(V)$, using similar
notation to (\ref{cn}). Let $\Gamma_{\GL,r}(n)$ be the set of all cyclic elements
of $\GL(V)_U$. Then $\Gamma_{\GL,r}(n)$ is invariant under conjugation by elements of
$\GL(V)_U$ and by Lemma \ref{m,m=conjclasses}, there is a one-to-one correspondence
between the set of orbits on $\Gamma_{\GL,r}(n)$ under this conjugation action of $\GL(V)_U$
and the set of pairs of
monic polynomials $(f,h)$ over $\F_q$ where \deg$(f)=r$, \deg$(h)=n$, $f$ divides
$h$ and $h(0)\ne 0$. Hence all cyclic matrices $A$ in $\GL(V)_U$, such
that $m_{A|_U} = f$ and $m_{A|_V} = h$, lie in the same
$\GL(V)_U$-orbit which we denote by $\Gamma_{f,h}$. Also, by Lemma
\ref{cent} the centraliser in $\GL(V)_U$ of a cyclic matrix $A$ in
$\Gamma_{f,h}$ is equal to the centraliser in $\GL(V)$
of $A$, and hence has order $Cent(h)$ (as defined in Section
\ref{BG-Cent}). We have that

\begin{equation} \label{gamma_fh}
|\Gamma_{f,h}| = \frac{|\GL(V)_U|}{Cent(h)}
\end{equation}
and hence $|\Gamma_{f,h}|$ is the same for any $f$ of degree $r$
dividing $h$.

Since any $r$-dimensional subspace of $V$ can be mapped to any
other $r$-dimensional subspace of $V$ by an element of $\GL(V)$,
the size $|\Gamma_{f,h}|$ is independent of the subspace $U$ of
$V$, depending only on the dimension $r$.
So without any loss of generality let $c_{\GL,r}(n)$ be the
proportion of cyclic matrices in $\GL(V)_U$ and let

$$\displaystyle{C_{\GL,r}(t) = \sum_{n=r}^\infty c_{\GL,r}(n) t^n}$$

\noindent
be the associated generating function. Note that the terms start from
$n=r$ because $c_{\GL,r}(n)$ is not defined
for $n < r$.

\begin{definition} \label{alpha}
{\rm Let $\alpha(h;r)$ denote the number of distinct monic degree $r$ factors of $h$.}
\end{definition}

\noindent Note that $\alpha(h;r)$ will be zero if the degree of $h$ is
less than $r$. Recall that $\mathcal{P}^+$ is the set of monic polynomials
over $\F_q$ with nonzero constant term. Since $| \Gamma_{f,h} |$ is the same
for any degree $r$ factor of $h$ we can write

\begin{equation} \label{crn}
c_{\GL,r}(n) = \sum_{
\begin{array}{c}
\vspace{-0.2cm}
\textrm{\scriptsize $h$$\in$$\mathcal{P}^+$} \\
\vspace{-0.2cm}
\textrm{\scriptsize \deg$(h)$$=$$n$} \\
\vspace{-0.2cm}
\textrm{\scriptsize \deg$(f)$$=$$r$} \\
\vspace{-0.2cm}
\textrm{\scriptsize $f|h$} \\
\end{array}
}
\frac{|\Gamma_{f,h}|}{|\GL(V)_U|} =
\sum_{
\begin{array}{c}
\vspace{-0.2cm}
\textrm{\scriptsize $h$$\in$$\mathcal{P}^+$} \\
\vspace{-0.2cm}
\textrm{\scriptsize \deg$(h)$$=$$n$} \\
\end{array}
}
\frac{\alpha(h;r) |\Gamma_{f,h}|}{|\GL(V)_U|} =
\sum_{
\begin{array}{c}
\vspace{-0.2cm}
\textrm{\scriptsize $h$$\in$$\mathcal{P}^+$} \\
\vspace{-0.2cm}
\textrm{\scriptsize \deg$(h)$$=$$n$} \\
\end{array}
}
\frac{\alpha(h;r)}{Cent(h)}
\end{equation}

\noindent
where in the summation $f$ is a monic degree $r$ divisor of $h$.
If there were an easy way to calculate
$\alpha(h;r)$ then calculating $c_{\GL,r}(n)$ would be simple but unfortunately there
is not.

For each monic irreducible polynomial $p$ of degree $i$ we will again
form the power series $F^p$ as defined in Equation (\ref{F_p}) and let
$F_i$ be as defined in Equation (\ref{F_i}). We saw in Section
\ref{GL(V)} that

$$\prod_{i=1}^\infty F_i = \sum_{h\in\mathcal{P}^+} \frac{t^{\deg(h)}}{Cent(h)}.$$
This is almost the generating function we require except that we need a
factor $\alpha(h;r)$ in the term corresponding to $h$, for each $h$.

We provide two lemmas before the methodology for calculating
a new expression for $C_{\GL,r}(t)$.

\begin{lemma} \label{te}
For $b \in \mathbb{Z}^+$,
$\displaystyle\sum_{m\ge 0} \frac{m^b t^{mk}}{k^m m!} =
\frac{t}{k} \frac{d}{d t} \left( \ldots \left(
\frac{t}{k} \frac{d}{d t}
\left(e^{t^k / k}\right)\right)\right)$, where the differentiation is performed $b$
times.\\
\end{lemma}

\begin{proof}
Since $e^{t^k / k} = \sum_{m \ge 0} \frac{t^{km}}{k^m m!}$,
we have $\frac{t}{k} \frac{d}{dt} (e^{t^k / k}) = \sum_{m \ge 0} \frac{m t^{km}}{k^m m!}$
and repeated differentiation gives the result.
\end{proof}

\begin{lemma} \label{diffchoose}
$$\frac{1}{m!} \frac{d^m}{dx^m} \sum_{n\ge 0} a_n x^n = \sum_{n\ge 0}
{{n}\choose{m}}a_n x^{n-m}.$$
\end{lemma}

\begin{proof}
Differentiating the power series gives $\sum_{n\ge 0} a_n n(n-1)\ldots (n-m+1)
x^{n-m}$. Now $\frac{n(n-1)\ldots (n-m+1)}{m!} = {{n}\choose{m}}$ so
on noting that ${n \choose m} = 0$ if $n<m$, we get the result.
\end{proof}

For a given $r$, the partitions of
$r$ correspond to all the ways we can obtain a monic degree $r$ polynomial
as a product of
irreducible polynomials. For example, the partitions of $2$ are
$\{2\}$ and $\{1,1\}$. These correspond to a monic irreducible degree $2$ polynomial
and a product of two monic irreducible degree $1$ polynomials respectively. The problem arises when
we have two irreducible \polys of the same degree dividing our degree $r$
polynomial, as is possible in the $\{1,1\}$ case. We can't just choose
two \polys out of all the degree $1$ factors because then we'll be
over-counting if we have polynomials with multiplicity greater than
one. So we need to break this case down into two sub-cases. In one sub-case
we choose two different degree $1$ \polys and in the other sub-case
we choose one degree $1$ \poly of multiplicity at least $2$. These choices
again correspond to the partitions of $2$ but this time the partitions
indicate the multiplicity of our factor not the degree.

Rather than taking partitions of the
components of the original partition, we create a two-dimensional array
whilst keeping in mind what we are counting.
We denote such an array by
$(m_{ij})$, where $m_{ij}$ denotes the number of monic irreducible degree $i$
factors of $h$ that divide our degree $r$ polynomial $f$ with multiplicity $j$.
The constraints on $(m_{ij})$ are that we need the
total degree of all the polynomials dividing $f$ to equal $r$ and we need,
for each $i$, the total number of distinct degree $i$ irreducible
polynomials dividing $f$ to be no more than $N(i,q)$ (see Definition \ref{N}).
We formulate these constraints as follows.

\begin{definition} \label{M}
{\rm Let $\mathcal{M}(r)$ be the set of all two-dimensional integer arrays,
$M = (m_{ij})$, such that
\begin{enumerate}
\item $\sum_{i,j}ijm_{ij} = r$;
\item $\sum_j m_{ij} \le N(i,q)$ for all $i$; and
\item $m_{ij}\ge0$ for all $i,j$.
\end{enumerate}
Let $\mathcal{M}_{part}(r)$ be the subset of $\mathcal{M}(r)$
consisting of all $(m_{ij})$ for which $m_{ij} = 0$ whenever $j \ge 2$.}
\end{definition}

Note that, by parts (1) and (3), $m_{ij}=0$ whenever $ij > r$ so these
arrays are finite and could have been defined as $r \times r$ arrays.

We will often work with just those arrays $(m_{ij})$ for which $m_{ij}=0$
whenever $j\ge 2$ so that $(m_{ij})$ is essentially a ``column
vector''. These correspond to polynomials $f$ having no
irreducible factors with multiplicity more than $1$, that is to say,
{\it separable polynomials}. Moreover such an $(m_{ij})$
corresponds to a partition of $r$ with $m_{i1}$ parts of size $i$, for
each $i$. Hence $\mathcal{M}_{part}(r)$ contains those $(m_{ij})$
which correspond to partitions of $r$.

Each degree $r$ factor $f$ of $h$ corresponds to a unique $M = M(f) =
(m_{ij}) \in \mathcal{M}(r)$ defined as follows. Let $m_{ij}$ be the number of distinct
monic irreducible degree $i$ factors of $f$ of multiplicity $j$.
Denote by $\alpha(h;r,M)$ the number of degree $r$ factors $f$ of $h$ such that $M(f) = M$.
If $M(f) = M$ we will say that {\emph f corresponds to M}.
We can break down the problem of computing
$\alpha(h;r)$ into computing $\alpha(h;r,M)$ for each $M \in \mathcal{M}(r)$ and summing over $M$, that is,

$$\alpha(h;r)=\sum_{M\in\mathcal{M}(r)} \alpha(h;r,M).$$

Now we concentrate on a fixed $M \in \mathcal{M}(r)$. We have to
make sure that, when selecting irreducible polynomials as factors
of $f$, we do not choose one we have already chosen to be in $f$ or
else that \poly will have the incorrect multiplicity. The choices of
irreducible degree $i$ polynomials of various multiplicities are
independent for each $i$ so we break down the problem even
further.

Let $h \in \mathcal{P}^+$ and for each $i$, let $h_i$ be the product of
all the monic irreducible degree $i$ factors of $h$, including
multiplicities. If $h$ has no degree $i$ factors for some $i$, then set $h_i := 1$.
Thus $h_i \in \mathcal{P}^+_i$ for all $i$, and $h =
h_1\ldots h_n$. Similarly, any degree $r$ factor $f$ of $h$ can be written as
$f = f_1 \ldots f_r$ where for each $i$,
$f_i$ is the product of all the monic irreducible degree $i$ factors of $f$, including multiplicities,
$f_i \in \mathcal{P}^+_i$ and $f_i$ divides $h_i$.

Since each degree $r$ factor $f$ of $h$ corresponds to a unique $M = M(f) = (m_{ij}) \in \mathcal{M}(r)$,
it follows that $r = \sum_i i \left( \sum_j j m_{ij}\right)$.
Since all the factors of $f_i$ are degree $i$ polynomials, $f_i$ `corresponds' to row $i$ of the matrix $M = M(f) = (m_{ij})$.
We define the matrix to which $f_i$ corresponds as

\begin{equation} \label{Mi}
M_i := M_i(M) = (m'_{kj}) = \left\{
\begin{array}{l l}
m_{ij} & \textrm{if } k  =  i \\
0      & \textrm{if } k \ne i. \\
\end{array}
\right.
\end{equation}

Let $r_i := r_i(M_i) := i \left( \sum_j j m_{ij}\right)$. Then it follows that
$M_i \in \mathcal{M}(r_i)$, and
$r_i = \deg(f_i)$. Also, $M = \sum_{i=1}^r M_i$.

The problem of calculating the number of degree $r$
factors of $h$ corresponding to $M$, can be solved by calculating for
each $i$, the number of monic degree $r_i$ polynomials dividing $h_i$
corresponding to $M_i$. We have

\begin{equation} \label{alphas}
\alpha(h;r) = \sum_{M\in\mathcal{M}(r)} \alpha(h;r,M) =
\sum_{M\in\mathcal{M}(r)} \left( \prod_{i=1}^r
\alpha(h_i;r_i,M_i)\right).
\end{equation}

We should note that $\alpha(h_i; 0, M_i) = 1$ for all $M_i \in \mathcal{M}(0)$,
all $h_i \in \mathcal{P}^+_i$ and all $i$ since the
constant polynomial $1$ divides $h_i$.

Also note that if the degree of $h$ is less than $r$ then $h$
clearly has no degree $r$ factors. Hence in this case we have,
for all $M \in \mathcal{M}(r)$, $\alpha(h;r,M) = 0$ since for at least
one $i$, we must have $\alpha(h_i;r_i,M_i) = 0$.

Below we define $\tau$-parameters and then determine $\alpha(h_i;r_i,M_i)$
in terms of these parameters.

\begin{definition} \label{tau}
{\rm Let $\tau(h;i,j)$ denote the number of distinct monic irreducible degree $i$
factors of $h$ that have multiplicity exactly $j$. Let $\tau(h;i,j,+)$
be the number of distinct monic irreducible degree $i$ factors of $h$ with
multiplicity $j$ or greater.}
\end{definition}

\begin{lemma} \label{alpha^i}
For $M = (m_{ij}) \in \mathcal{M}(r)$,
$$\displaystyle\alpha(h_i;r_i,M_i) = \prod_{j=1}^{\lfloor r/i
\rfloor} {{\tau(h_i;i,j,+)-\sum_{k>j} m_{ik}}\choose{m_{ij}}}.$$
\end{lemma}

\begin{proof}
Set $m_i :=  \sum_j m_{ij}$.
We describe the process of selecting the $m_i$
pairwise distinct monic irreducible degree $i$ factors of $f$ for a
fixed value of $i$. If $m_i = 0$ there is nothing to do so assume
that $m_i > 0$ and let $j(i,M)$ be the largest integer $j$ such that
$m_{ij} > 0$. Note that $j(i,M) \le r/i$ by Definition \ref{M} $(1)$.
The $m_{ij(i,M)}$ monic irreducible degree $i$ factors of $f$
with multiplicity $j(i,M)$ can be chosen in ${\tau(h;i,j(i,M),+) \choose
m_{ij(i,M)}}$ ways. Once these are chosen we must not choose them again
as factors to ensure they have the correct multiplicity in $f$. Thus for the
next largest $j$ such that $m_{ij} > 0$, we choose the $m_{ij}$
monic irreducible degree $i$ factors having multiplicity $j$ from the
remaining available $\tau(h;i,j,+) - m_{ij(i,M)}$ irreducible
factors of $h_i$ having multiplicity at least $j$. In general, if $j$ is
such that we have already chosen the $m_{ik}$ monic irreducible degree $i$
factors of $f$ of multiplicity $k$ for all $k > j$, then we may choose
the $m_{ij}$ such factors of $f$ of multiplicity $j$ in exactly
${ \tau(h;i,j,+) - \sum_{k>j} m_{ik} \choose m_{ij} }$ ways. Hence
the result

$$\displaystyle\alpha(h_i;r_i,M_i) = \prod_{j=1}^{j(i,M)}
{{\tau(h;i,j,+)-\sum_{k>j} m_{ik}}\choose{m_{ij}}}.$$

\end{proof}

We now define the function $\Phi^p$ for a monic irreducible degree $i$
polynomial $p$ as

\begin{equation} \label{Phi_p}
\Phi^p = \Phi^p(t, (s_{ij})_{j\ge 1} ) := 1 + \sum_{j=1}^\infty
\frac{s_{ij} t^{ij}}{Cent(i,j)}.
\end{equation}

As with $F^p$, $\Phi^{p_1} = \Phi^{p_2}$ if $p_1$ and
$p_2$ are monic, irreducible polynomials of the same degree. So for
all $i\in\mathbb{Z}^+$ we will define $\Phi^+_i$ as the product of
the $\Phi^p$ for all monic irreducible degree $i$ polynomials $p$ with a
nonzero constant term. Thus

\begin{equation}\label{Phi_i}
\Phi^+_i = \Phi_i^+(t, (s_{ij})_{j\ge 1} ) := \left(1 + \sum_{j=1}^\infty
\frac{s_{ij} t^{ij}}{Cent(i,j)}\right)^{N^+(i,q)}.
\end{equation}

\begin{lemma} \label{Phipower}
As a power series, $\Phi^+_i$ satisfies:

$(1).$ $\displaystyle \Phi^+_i = \sum_{h_i \in \mathcal{P}^+_i}
\frac{(\prod_j s_{ij}^{\tau{(h_i;i,j)}}) t^{\deg(h_i)}}{Cent(h_i)}, \quad and$

$(2).$ $\displaystyle\prod_{i=1}^\infty \Phi^+_i = \sum_{h\in
\mathcal{P}^+} \frac{(\prod_{i,j} s_{ij}^{\tau{(h;i,j)}}) t^{\deg(h)}}{Cent(h)},$

\noindent where $\mathcal{P}^+$ and $\mathcal{P}^+_i$ are as in Definition
\ref{P}.
\end{lemma}

\begin{proof}
$(1).$
Every polynomial in $\mathcal{P}^+_i$ corresponds to a summand of
$\Phi^+_i$, namely each $h_i \in
\mathcal{P}^+_i$ is a product of $p^{j(p)}$ (where $j(p)\ge 0$) over all
monic degree $i$ irreducible polynomials $p$, and \deg$(h_i) = \sum_p
i j(p)$. The corresponding summand of $\Phi^+_i$ is $\frac{ (\prod_p
s_{ij(p)}) t^{\deg(h_i)}}{Cent(h_i)}$ and this is obtained by
choosing the term corresponding to $p^{j(p)}$, which is
$\frac{s_{ij(p)}t^{ij(p)}}{Cent(i,j(p))}$ if $j(p)\ge 1$ and $1$ if
$j(p)=0$. The denominator of the summand corresponding to $h_i$ is
$Cent(h_i)$ since $\prod_p Cent(i,j(p)) = Cent(\prod_p p^{j(p)}) =
Cent(h_i)$ by Lemma \ref{Cent}. The exponent of $t$ is $\sum_p ij(p) = \deg(h_i)$.
For each monic irreducible degree $i$
polynomial with multiplicity $j$, the corresponding $\Phi^p$
contains an $s_{ij}$ term, so
the exponent of $s_{ij}$ in the summand corresponding to $h_i$ is the
number of monic irreducible degree $i$ factors $p$ of $h_i$ such that
$j(p) = j$, that is $\tau(h_i;i,j)$. Hence

$$\displaystyle \Phi^+_i = \sum_{h_i\in \mathcal{P}^+_i}
\frac{(\prod_j s_{ij}^{\tau{(h_i;i,j)}}) t^{\deg(h_i)}}{Cent(h_i)}.$$

$(2).$
$\prod_i \Phi^+_i$ is simply the product over all monic irreducible
polynomials with nonzero constant terms as opposed to the computation
in case $(1)$ where we just chose those irreducible polynomials with
degree $i$. Hence every polynomial in $\mathcal{P}^+$ will correspond to
a certain term in $\prod_i \Phi^+_i$. By part $(1)$ the term
corresponding to a polynomial $h$ is

$$\frac{(\prod_{i,j} s_{ij}^{\tau{(h;i,j)}}) t^{\deg(h)}}{Cent(h)}$$

\noindent
where now we have $s_{ij}$ present for each $i$ such that $\tau(h;i,j)
> 0$. Hence we get our result

$$\displaystyle\prod_{i=1}^\infty \Phi^+_i = \sum_{h\in
\mathcal{P}^+} \frac{(\prod_{i,j} s_{ij}^{\tau{(h;i,j)}}) t^{\deg(h)}}{Cent(h)}.$$

\end{proof}

Note again that for each $i \in \mathbb{Z}^+$ the only occurrences of
$s_{ij}$ for any $j$ are in $\Phi^+_{i}$. So if we partially
differentiate $\Phi^+$ with respect to $s_{ij}$ for some $j$ we get the
same answer as if we partially differentiated $\Phi^+_i$ with respect to
$s_{ij}$ and multiplied by the product of the remaining $\Phi^+_{i'}$,
$i'\ne i$.

We will perform a series of operations on each $\Phi^+_i$ in turn so
that after this procedure for $\Phi^+_i$ the coefficient of the term
corresponding to a polynomial $h_i \in \mathcal{P}^+_i$ and an $M \in
\mathcal{M}(r)$ will be $\alpha(h_i;r_i,M_i) / Cent(h_i)$.

\begin{definition}\label{Phialphadef}
{\rm For $i\ge 1$ and $M \in \mathcal{M}(r)$, let

$$\displaystyle \Phi^+_{i,M,\alpha} = \Phi^+_{i,M,\alpha}(t) := \sum_{h_i \in
\mathcal{P}^+_i} \alpha(h_i;r_i,M_i) \frac{t^{\deg(h_i)}}{Cent(h_i)},$$

\noindent where $r_i = i \sum_j jm_{ij}$.}
\end{definition}

Note that, if $r_i = 0$ (which holds in particular if $i>r$) then
$\alpha(h_i, r_i, M_i) = 1$ for each $h_i \in \mathcal{P}^+_i$, and
hence in this case $\Phi^+_{i,M,\alpha} = F_i$ as defined in (\ref{F_i}). In
this case the leading term of $\Phi^+_{i,M,\alpha}$ is $1$ and
corresponds to the constant polynomial $1 \in \mathcal{P}^+_i$.

The product of
$\Phi^+_{i,M,\alpha}$ over all $i$ is a power series such that the
coefficient of the term corresponding to a polynomial $h \in
\mathcal{P}^+$ is $\alpha(h;r,M) / Cent(h)$. Then summing over all $M
\in \mathcal{M}(r)$ will produce a series such that the coefficient of
the term corresponding to $h$ is $\alpha(h;r) / Cent(h)$, that is
$C_{\GL,r}(t)$. We prove this now.

\begin{lemma} \label{Crfinal}
$\displaystyle C_{\GL,r}(t) = \sum_{M \in \mathcal{M}(r)} \prod_{i=1}^\infty
\Phi^+_{i,M,\alpha}(t)$.
\end{lemma}

\begin{proof}
From Definition \ref{Phialphadef} we have that

$$\Phi^+_{i,M,\alpha} = \sum_{h_i\in\mathcal{P}^+_i} \alpha(h_i;r_i,M_i)
\frac{t^{\deg(h_i)}}{Cent(h_i)}.$$

\noindent Recall from Definition \ref{P}, that $\mathcal{P}^+_i$ is the
set containing the constant polynomial $1$ along with all monic polynomials
that have a nonzero constant term and whose only irreducible factors
have degree $i$. Each $h \in \mathcal{P}^+$, that is, each monic
polynomial with nonzero constant term, has a unique factorisation, $h
= \prod_i h_i$, where $h_i \in \mathcal{P}^+_i$ for each $i$. Using this
notation we have

$$ \prod_{i=1}^\infty \Phi^+_{i,M,\alpha} = \sum_{h\in\mathcal{P}^+} \prod_{i=1}^\infty
\frac{\alpha(h_i;r_i,M_i) t^{\deg(h_i)}} {Cent(h_i)}.$$

\noindent By Equation \ref{alphas} we know that $\prod_i
\alpha(h_i;r_i,M_i) = \alpha(h;r,M)$. By Lemma \ref{Cent} we know that
$\prod_i Cent(h_i) = Cent(h)$ since the $h_i$ are pairwise coprime and
we know that each polynomial $h_i$ corresponds to the $t$-power
$t^{\deg(h_i)}$ so $\prod_i t^{\deg(h_i)} = t^{\deg(h)}$. Thus

$$\prod_{i=1}^\infty \Phi^+_{i,M,\alpha} =  \sum_{h\in\mathcal{P}^+}
\frac{\alpha(h;r,M) t^{\deg(h)}}{Cent(h)}.$$

Now we want to sum over all $M \in \mathcal{M}(r)$. The sum over
$h\in\mathcal{P}^+$ and the sum over $M\in\mathcal{M}(r)$ can be
interchanged so that

\begin{displaymath}
\begin{array} {rcl}
\displaystyle
\sum_{M\in\mathcal{M}(r)} \prod_{i=1}^\infty \Phi^+_{i,M,\alpha} &
= & \displaystyle \sum_{M\in\mathcal{M}(r)} \sum_{h\in\mathcal{P}^+} \frac{\alpha(h;r,M)
t^{\deg(h)}}{Cent(h)} \\
& = & \displaystyle \sum_{h\in\mathcal{P}^+}
\frac{\left(\sum_{M\in\mathcal{M}(r)} \alpha(h;r,M)\right)
t^{\deg(h)}}{Cent(h)}. \\
\end{array}
\end{displaymath}

\noindent By Lemma \ref{alphas}, we know that
$\sum_{M\in\mathcal{M}(r)} \alpha(h;r,M) = \alpha(h;r)$ and by using
Equation (\ref{crn}) we have

\begin{displaymath}
\begin{array} {rcl}
\displaystyle
\sum_{M\in\mathcal{M}(r)} \prod_{i=1}^\infty \Phi^+_{i,M,\alpha} & = &
\displaystyle
\sum_{h\in\mathcal{P}^+} \frac{\alpha(h;r) t^{\deg(h)}}{Cent(h)} \\
& = & \displaystyle \sum_{n=0}^\infty c_{\GL,r}(n)t^n \\
& = & \displaystyle \sum_{n=r}^\infty c_{\GL,r}(n)t^n \\
& = & C_{\GL,r}(t) \\
\end{array}
\end{displaymath}

\noindent since $c_{\GL,r}(n)$ is trivially evaluated to $0$ for $n < r$.
\end{proof}

We now describe the procedure to produce $\Phi^+_{i,M,\alpha}$ for a fixed $r
\in \mathbb{Z}^+$, a fixed $M = (m_{ij}) \in \mathcal{M}(r)$ and a
fixed $i < r$.

\bigskip
\noindent \textbf{Procedure:} \underline{\sc{PhiAlpha $(r, M, i)$}}

\bigskip
\noindent \underline{Input:} \quad $r = \dim(U)$, $M \in \mathcal{M}(r)$, $i \in \mathbb{Z}^+$.

\bigskip
\noindent \underline{Output} \quad $\Phi^+_{i,M,\alpha}(t)$

\bigskip
\noindent \underline{Operations}
\vspace{0.1cm}

Set $\Psi:=\Phi^+_i$, $s_{i0}:=1$ and
$k:=\lfloor \frac{r}{i}\rfloor$.

\bigskip
For $j > k$ assign $s_{ij} := s_{ik}$ in $\Psi$.

\bigskip
While $k > 0$

\quad Set $\Psi:= \displaystyle \frac{1}{m_{ik}!}
\frac{\partial^{m_{ik}}}{\partial s_{ik}^{m_{ik}}} \Psi.$

\quad Set $s_{ik} := s_{i(k-1)}$ and $k := k-1$.

\bigskip
Return $\Psi$
\bigskip
\bigskip

\begin{lemma} \label{phialpha}
(1) The procedure {\sc PhiAlpha}$(r,M,i)$ correctly returns
$\Phi^+_{i,M,\alpha}$;

(2) The smallest $j$ such that $t^j$ has nonzero coefficient in the power series expansion of
$\Phi^+_{i,M,\alpha}(t)$ is $j = r_i$.

(3) Moreover, if $i \sum_j j m_{ij} = r_i = 0$ then
$\Phi^+_{i,M,\alpha}(t) = F_i(t)$.
\end{lemma}

\begin{proof}
The initial value of $k$ is $0$ if and only if $i > r$
and in this case the procedure returns $\Phi^+_i|_{(s_{ij}=1 \textrm{ for all } j)}.$
This power series is equal to $F_i$, as defined in Equation (\ref{F_i}), and is correct
since we do not wish to choose any polynomials of this degree
(since they cannot divide $f$). As noted after Definition \ref{Phialphadef}, in this case $r_i =
0$, $F_i(t) = \Phi^+_{i,M,\alpha}(t)$ and the leading term is $1$. Thus
the assertions $(1) - (3)$ follow in this case.

Assume now that $i \le r$. We first prove part $(1)$. At the start of
the procedure $\Psi := \Phi^+_i$. The value $k = \lfloor r/i
\rfloor \ge 1$ is the highest possible multiplicity of a degree $i$
factor we could choose to be a factor of $f$.

Let $j_0 = \lfloor r/i \rfloor$. Start by making $s_{ij}$ equal
$s_{ij_0}$ in $\Psi$ for all $j > j_0$. After this the exponent of
$s_{ij_0}$ is the number of irreducible degree $i$ factors of $h$ that
have multiplicity at least $j_0$, that is $\tau(h;i,j_0,+)$. At the
beginning of the while loop we have

$$ \Psi = \sum_{h_i\in\mathcal{P}^+_i} \left( \left(\prod_{j<j_0}
s_{ij}^{\tau(h_i;i,j)}\right) s_{ij_0}^{\tau(h_i;i,j_0,+)} \frac{t^{\deg(h_i)}}{Cent(h_i)} \right).$$

We repeat the steps in the while loop $j_0$ times. After performing
the partial differentiation step the first time (for $k=j_0$),
Lemma \ref{diffchoose} implies we get

$$\Psi = \sum_{h_i\in\mathcal{P}^+_i} \left(
\left(\prod_{j<j_0} s_{ij}^{\tau(h_i;i,j)}\right)
s_{ij_0}^{\tau(h_i;i,j_0,+)-m_{ij_0}} {{\tau(h_i;i,j_0,+)}\choose{m_{ij_0}}}
\frac{t^{\deg(h_i)}}{Cent(h_i)}\right).$$

\noindent The $s_{ij}$ for $j < j_0$ remain, while the power of
$s_{ij_0}$ is reduced by $m_{ij_0}$. The
${{\tau(h_i;i,j_0,+)}\choose{m_{ij_0}}}$ appears as shown in Lemma
\ref{diffchoose}. If $j_0 = 1$ then we assign $s_{i1}:=1$ and we
return

$$\Psi =  \sum_{h_i\in\mathcal{P}^+_i}
{{\tau(h_i;i,j_0,+)}\choose{m_{ij_0}}} \frac{t^{\deg(h_i)}}{Cent(h_i)}.$$
By Lemma \ref{alpha^i}, if $j_0 = 1$ then
${{\tau(h_i;i,j_0,+)}\choose{m_{ij_0}}} = \alpha(h_i;r_i,M_i)$
and hence $\Psi = \Phi^+_{i,M,\alpha}$ is correctly returned.

Now assume that $j_0 > 1$ so that we assign $s_{ij_0}$ to equal
$s_{i(j_0-1)}$. The new exponent on $s_{i(j_0-1)}$ will be
$\tau(h_i; i,(j_0-1),+) - \sum_{j>j_0} m_{ij}$. For a general $k$ the while
loop proceeds as follows. After differentiating, by Lemma
\ref{diffchoose} the term corresponding to $h_i \in \mathcal{P}^+_i$ is
multiplied by ${{\tau(h_i;i,k,+)-\sum_{j>k}m_{ij}}\choose{m_{ik}}}$, and the $s_{ij}$ are
modified. Hence after running the while loop for $k = j_0, j_0-1, \ldots,
1$ we will have produced the coefficient

$$\prod_{j=1}^{j_0} {{\tau(h_i;i,j,+)-\sum_{k>j}
m_{ik}}\choose{m_{ij}}}$$
and the function returned is

$$ \Psi = \sum_{h_i \in \mathcal{P}^+_i} \prod_{j=1}^{j_0}
{{\tau(h_i;i,j,+)-\sum_{k>j} m_{ik}}\choose{m_{ij}}} \frac{t^{\deg(h_i)}}{Cent(h_i)}.$$
By Lemma \ref{alpha^i} we see that we have correctly returned

$$ \Psi = \Psi(t) = \sum_{h_i \in \mathcal{P}^+_i} \alpha(h_i;r_i,M_i)
\frac{t^{\deg(h_i)}}{Cent(h_i)} = \Phi^+_{i,M,\alpha}(t) .$$

\noindent
Thus part $(1)$ is proved. In particular, if $r_i = 0$ then
$\alpha(h_i; r_i, M_i) = 1$ so that $\Phi^+_{i,M,\alpha} = F_i$, so part
$(3)$ holds.

Let $l < r_i = i \sum_j jm_{ij}$. If $h_i \in \mathcal{P}^+_i$ and
\deg$(h_i) = l$, then $h_i$ has no factors of degree $r_i$, so
$\alpha(h_i, r_i, M_i) = 0$. Thus the coefficient of $t^l$ in
$\Phi^+_{i,M,\alpha}(t)$ is zero. On the other hand if $l = r_i$ then
there exists at least one $h_i \in \mathcal{P}^+_i$ of degree $r_i$
(even if $r_i=0$) such that $h_i$ has $m_{ij}$
monic irreducible degree
$i$ factors of multiplicity $j$ for each $j$. For each such $h_i$,
$\alpha(h_i; r_i, M_i) = 1$ and so the coefficient of $t^{r_i}$ is
nonzero. Thus part $(2)$ is proved.
\end{proof}

We now give an alternative expression for determining the generating
function $C_{\GL,r}(t)$ that involves Wall's function $C_{\GL}(t)$ defined in
Equation (\ref{C}). This expression will be used later to study
the asymptotic properties of $c_{\GL,r}(n)$.

\begin{theorem} \label{C_r=Cblah}
For each $i$, let $m_i := \sum_j m_{ij}$. Then

$$\displaystyle C_{\GL,r}(t) = C_{\GL}(t) \sum_{M\in\mathcal{M}(r)} \prod_{i=1}^r
\phi^+_{i,M}(t)$$

\noindent where $\displaystyle \phi^+_{i,M}(t) = m_i! {N^+(i,q) \choose
m_i} \prod_{j=1}^{\lfloor r/i \rfloor}
\frac{(tq^{-1})^{ijm_{ij}}}
{m_{ij}!(1-q^{-i}+t^iq^{-2i})^{m_{ij}}}$ and $C_{\GL}(t) = \prod_i F_i$ as
in Equation (\ref{C}). Moreover, if $i > r$ then $\phi^+_{i,M}(t) = 1$ for all $M \in \mathcal{M}(r)$.
\end{theorem}

\begin{proof}
By Lemma \ref{Crfinal},

$$ C_{\GL,r}(t) = \sum_{M \in \mathcal{M}(r)} \prod_{i=1}^\infty
\Phi^+_{i,M,\alpha}(t) $$

\noindent
and by (\ref{C})

$$ C_{\GL}(t) = \prod_{i=1}^\infty F_i(t). $$

\noindent
Also in Lemma \ref{phialpha} (2), for $i > r$ we have $\Phi^+_{i,M,\alpha}(t)
= F_i(t)$. Then it follows that

$$ C_{\GL,r}(t) = C_{\GL}(t) \sum_{M \in \mathcal{M}(r)} \prod_{i=1}^r
\frac{\Phi^+_{i,M,\alpha}(t)} {F_i(t)} $$

\noindent
where by (\ref{F_i}), $F_i(t) = \left( 1 + \sum_{j=1}^\infty
\frac{t^{ij}}{Cent(i,j)} \right)^{N^+(i,q)}$ and by Lemma
\ref{same_cent}, \newline $Cent(i,j) = q^{ij}(1-q^{-i}).$

We consider the details of the procedure {\sc PhiAlpha}$(r,M,i)$ that
constructs $\Phi^+_{i,M,\alpha}(t)$. Firstly the procedure sets $\Psi := \Phi^+_i$ and
assign $k := \lfloor r/i \rfloor$. Then the $s_{ij}$ are assigned to
$s_{ik}$ for all $j > k$ so that $\Psi$ becomes

$$ \left(1 + \sum_{j=1}^{k-1}
\frac{s_{ij}t^{ij}}{q^{ij}(1-q^{-i})} + \sum_{j=k}^\infty
\frac{s_{ik}t^{ij}}{q^{ij}(1-q^{-i})}\right)^{N^+(i,q)}
= \Psi_0^{N^+(i,q)}$$

\noindent
where $\Psi_0$ is given by $1 + \sum_{j=1}^{k-1}
\frac{s_{ij}t^{ij}}{q^{ij}(1-q^{-i})} + \sum_{j=k}^\infty
\frac{s_{ik}t^{ij}}{q^{ij}(1-q^{-i})}$.

In the first run of the while loop, we
partially differentiate $\Psi$ with respect to $s_{ik}$,
$m_{ik}$ times, and divide by $m_{ik}!$.
Suppose first that $m_{ik} > 0$. Then, by the chain rule, this will
produce

$$ N^+(i,q)(N^+(i,q)-1) \ldots (N^+(i,q)-m_{ik}+1) = \prod_{0 \le u < m_{ik}} (N^+(i,q)-u). $$

\noindent
The first partial derivative of $\Psi_0$ is
$\sum_{j=k}^\infty \frac{t^{ij}}{q^{ij}(1-q^{-i})}$ which is a
geometric progression equal to $\frac{t^{ik}}{q^{ik}(1-q^{-i})(1-t^i/q^i)}$.
Thus after performing the partial differentiation and division by
$m_{ik}!$ we obtain

$$ \left( \prod_{0\le u < m_{ik}} (N^+(i,q) - u) \right)
\frac{1}{m_{ik}!}
\left( \frac{t^{ik}}{q^{ik}(1-q^{-i})(1-t^iq^{-i})}\right)^{m_{ik}}
\Psi_0^{N^+(i,q)-m_{ik}}.$$

The final step in the while loop is to change $s_{ik}$ into
$s_{i(k-1)}$ and hence $\Psi_0$ becomes

$$ 1 + \sum_{j=1}^{k-2} \frac{s_{ij}t^{ij}}{q^{ij}(1-q^{-i})} +
\sum_{j=k-1}^\infty
\frac{s_{i(k-1)}t^{ij}}{q^{ij}(1-q^{-i})}.$$

As the procedure prescribes, we now repeat the while loop using $k-1$
in place of $k$. After performing the while loop $\lfloor r/i \rfloor$
times we will have partially differentiated $m_i$ times. After all
runs of the while loop, $\Psi$ will equal

$$ \left( \prod_{0\le u < m_{i}} (N^+(i,q) - u) \right)
\prod_{j=1}^{\lfloor r/i \rfloor}
\frac{1}{m_{ij}!} \left( \frac{t^{ij}}
{q^{ij}(1-q^{-i})(1-t^iq^{-i})}\right)^{m_{ij}} \hat{\Psi}_0^{N^+(i,q)-m_i} $$

\noindent where $\hat{\Psi}_0$ is equal to the function $\Psi_0$ with all the
$s_{ij}$ set to $1$. We make note of the fact that we only need to
take the product over $j$ from $1$ to $\lfloor r/i \rfloor$
because for $j > \lfloor r/i \rfloor$ we have that $m_{ij} = 0$ and
hence $\frac{1}{m_{ij}!} \left( \frac{t^{ij}}
{q^{ij}(1-q^{-i})(1-t^iq^{-i})}\right)^{m_{ij}} = 1$. This final
expression for $\Psi$ equals $\Phi^+_{i,M,\alpha}(t)$, by Lemma
\ref{phialpha}.

Finally observe that

\begin{displaymath}
\begin{array}{rcl}

\hat{\Psi}_0(t) & = & 1 + \displaystyle \sum_{j=1}^\infty
\frac{t^{ij}}{q^{ij}(1-q^{-i})} \\

& = & 1 + \displaystyle\frac{t^i}{q^i(1-q^{-i})(1-t^iq^{-i})} \\

\end{array}
\end{displaymath}

\noindent
and hence, by Equation \ref{F_i}, substituting in the values of
$Cent(i,j)$, we have $F_i(t) = \hat{\Psi}_0(t)^{N^+(i,q)}$. Since $m_i$ equals
$\sum_j m_{ij}$ and $m_i! {N^+(i,q) \choose m_i}$ is equal to $\prod_{0\le u < m_{i}}
(N^+(i,q) - u)$, it follows that $\frac{\Phi^+_{i,M,\alpha}(t)}{F_i(t)}$ is equal to

\begin{displaymath}
\begin{array}{l}

 \displaystyle m_i! {N^+(i,q) \choose m_i}
\frac{1}{\hat{\Psi}_0(t)^{m_i}}
\prod_{j=1}^{\lfloor r/i \rfloor} \frac{1}{m_{ij}!}
\left( \frac{t^{ij}}{q^{ij}(1-q^{-i})(1-t^iq^{-i})} \right)^{m_{ij}} = \\

\displaystyle m_i! {N^+(i,q) \choose m_i} \prod_{j=1}^{\lfloor r/i \rfloor}
\frac{1}{m_{ij}!} \left( \frac{1}{(1 +
\frac{t^i}{q^i(1-q^{-i})(1-t^i/q^i)})}
\frac{t^{ij}}{q^{ij}(1-q^{-i})(1-t^i/q^i)} \right)^{m_{ij}}. \\

\end{array}
\end{displaymath}

\noindent
Rearranging gives

\begin{displaymath}
\begin{array}{r c l}

\displaystyle \frac{\Phi^+_{i,M,\alpha}(t)}{F_i(t)}  & = & \displaystyle
m_i! {N^+(i,q) \choose m_i} \prod_{j=1}^{\lfloor r/i \rfloor}
\frac{1}{m_{ij}!} \left(
\frac{t^{ij} q^{-ij}} {(1-q^{-i})(1-t^iq^{-i}) + t^i q^{-i}}
\right)^{m_{ij}} \\

& = & \displaystyle  m_i! {N^+(i,q) \choose m_i} \prod_{j=1}^{\lfloor r/i \rfloor}
\frac{1}{m_{ij}!} \left( \frac{t^{ij} q^{-ij}}
{1 - q^{-i} + t^i q^{-2i}} \right)^{m_{ij}} \\

& = & \displaystyle  m_i! {N^+(i,q) \choose m_i} \prod_{j=1}^{\lfloor r/i \rfloor}
\frac{(tq^{-1})^{ijm_{ij}}}
{m_{ij}! \left(1 - q^{-i} + t^i q^{-2i}\right)^{m_{ij}}}. \\

& = & \phi^+_{i,M}(t). \\

\end{array}
\end{displaymath}

\end{proof}

\subsection{The Limiting Proportion} \label{sect4.2}

Now that we have a formula for the generating function, $C_{\GL,r}(t)$,
for any $r$, we will prove that its coefficients converge and find their
limiting value. We use the expression for
$\phi^+_{i,M}(t)$ from Theorem \ref{C_r=Cblah}.

\begin{lemma} \label{phiconv}
Let $M = (m_{ij}) \in \mathcal{M}(r)$. Then for any $i \le r$, the power
series expansion of

$$\phi^+_{i,M}(t) = m_i! {N^+(i,q) \choose m_i} \prod_{j=1}^{\lfloor r/i
\rfloor} \frac{(tq^{-1})^{ijm_{ij}}}
{m_{ij}!(1-q^{-i}+t^iq^{-2i})^{m_{ij}}},$$

\noindent is convergent for $|t| < q (q^i-1)^{1/i}$.
\end{lemma}

\begin{proof}
We first express $\phi^+_{i,M}(t)$ as follows:

$$ \phi^+_{i,M}(t) = m_i! {N^+(i,q) \choose m_i} t^{r_i}
q^{-r_i} (1-q^{-i}+t^iq^{-2i})^{-m_i}
\prod_{j=1}^{\lfloor r/i \rfloor} (\frac{1}{m_{ij}!}).$$

Clearly the polynomial $m_i! {N^+(i,q) \choose m_i} t^{r_i}
q^{-r_i} \prod_{j=1}^{\lfloor r/i \rfloor} (\frac{1}{m_{ij}!})$ is
convergent for all $t$, but $(1-q^{-i}+t^iq^{-2i})^{-m_i}$ needs further
inspection. We need to find the radius of convergence of the power
series for $(1-q^{-i}+ t^i q^{-2i})^{-1}$. Factorisation gives

\begin{displaymath}
\begin{array}{r c l}

\displaystyle \frac{1}{1-q^{-i}+t^iq^{-2i}} & = & \displaystyle
\frac{1}{1-q^{-i}} \left( \frac{1}{1+\frac{t^iq^{-2i}}{1-q^{-i}}}
\right). \\

\end{array}
\end{displaymath}

Now $\frac{1}{1-x} = \sum_{i=0}^\infty x^i$ is convergent if and only
if $|x| < 1$ and hence the power series for the displayed function is
convergent if $|\frac{t^iq^{-2i}}{1-q^{-i}}| < 1$, that is, $|t^i| <
q^i(q^i-1)$. This is equivalent to $|t| < q (q^i-1)^{1/i}$.
\end{proof}

\begin{corollary}\label{q(q-1)}
The power series expansion of $\displaystyle \sum_{M\in\mathcal{M}(r)}
\prod_{i=1}^r \phi^+_{i,M}(t)$ is convergent for $|t| < q(q-1)$.
\end{corollary}

\begin{proof}
By Lemma \ref{phiconv}, the power series expansion of $\phi^+_{i,M}(t)$
is convergent for $|t| < q(q^i-1)^{1/i}$ for all $i \le r$ and any
$M \in \mathcal{M}(r)$. Hence $\prod_i \phi^+_{i,M}(t)$ is convergent
for $|t|$ less than the minimum of $q(q^i-1)^{1/i}$ over $i=1,\dots,
r$. That is to say, for $|t| < q(q-1)$. Hence we have our result since
$\mathcal{M}(r)$ is finite.
\end{proof}

Despite the above corollary, $C_{\GL,r}(t)$ is only convergent for $|t| < 1$
because $C_{\GL}(t)$ is only convergent for $|t| < 1$ (see our comment
before Theorem \ref{limpropGL(V)}.

The next theorem gives us a
definite formula for the limit of $c_{\GL,r}(n)$ as $n$ tends to infinity.

\begin{theorem} \label{limcrn}
For $r \in \mathbb{Z}^+$, $c_{\GL,r}(\infty) := \lim_{n\to\infty} c_{\GL,r}(n)$ exists and satisfies

$$\displaystyle c_{\GL,r}(\infty) =
\frac{1-q^{-5}}{1+q^{-3}} \sum_{M\in\mathcal{M}(r)} \prod_{i=1}^r
\phi^+_{i,M}(1),$$
where $\phi^+_{i,M}(t)$ is defined as in Theorem \ref{C_r=Cblah}. Moreover, for any $d$ such that $1 < d < q(q-1)$,
$| c_{\GL,r}(n) - c_{\GL,r}(\infty) | = O(d^{-n})$.
\end{theorem}

\begin{proof}
From \cite{wall}, $(1-t) C_{\GL}(t)$ is convergent for $|t| < q^2$ and by
Corollary \ref{q(q-1)}, $\sum_{M\in\mathcal{M}(r)} \prod_i
\phi^+_{i,M}(t)$ is convergent for $|t| < q(q-1)$. Hence $(1-t)C_{\GL,r}(t)$
is convergent for $|t| < q(q-1)$.

By Lemma \ref{1-t}, $(1-t) C_{\GL,r}(t)$ evaluated at $t=1$ will give us
the limit of $c_{\GL,r}(n)$ as $n \rightarrow \infty$. By
\cite[Equation 6.23]{wall} we know that $(1-t) C_{\GL}(t)$ evaluated at $t=1$ is

$$\frac{1 - q^{-5}}{1 + q^{-3}}.$$
Thus

$$\displaystyle \lim_{n\rightarrow\infty} c_{\GL,r}(n) =
\frac{1-q^{-5}}{1+q^{-3}} \sum_{M\in\mathcal{M}(r)} \prod_{i=1}^r
\phi^+_{i,M}(1).$$

We showed above that $(1-t) C_{\GL,r}(t)$ is convergent for $|t| < q(q-1)$,
so the final assertion stating the rate of convergence follows by Lemma \ref{1-t}.
\end{proof}

We note that the existance of the limit $c_{\GL,r}(\infty)$ and the convergence rate of
$c_{\GL,r}(n)$ asserted in Theorem \ref{1-q^-2ch1} follow from Theorem \ref{limcrn}.

We now evaluate $\phi^+_{i,M}(1)$ for all $i$ and
determine the first few terms of the power series expansion. This will
aid us in proving Theorem \ref{1-q^-2} which determines
the first nontrivial term in the power series expansion for
$c_{\GL,r}(\infty)$.

\begin{lemma}\label{phiigen}
Let $M=(m_{ij}) \in \mathcal{M}(r)$ and set $m_i = \sum_j m_{ij}$ for
each $i$. Then with the notation of Theorem \ref{C_r=Cblah} we get
\begin{eqnarray*}
\phi^+_{1,M}(1) & = & \frac{q^{m_1-\sum jm_{1j}}}{\prod_j m_{1j}!} \left( 1 +
\left( -\frac{m_{1}^2}{2} + \frac{m_{1}}{2} \right) q^{-1} \right.\\
& & + \left. \left(\frac{m_{1}^4}{8} - \frac{5m_{1}^3}{12} - \frac{m_{1}^2}{8} -
\frac{7m_{1}}{12}\right)q^{-2} + O(q^{-3}) \right), \\
\phi^+_{2,M}(1) & = & \frac{q^{2m_2-2\sum jm_{2j}}}{2^{m_2} \prod_j m_{2j}!}
\left(1 - m_2q^{-1} + \left(-\frac{m_2^2}{2}+\frac{3m_2}{2}\right)
q^{-2} + O(q^{-3})\right), \\
\phi^+_{3,M}(1) & = & \frac{q^{3m_3-3\sum jm_{3j}}}{3^{m_3} \prod_j m_{3j}!}
\left(1 - m_3q^{-2} + O(q^{-3})\right), \\
\phi^+_{4,M}(1) & = & \frac{q^{4m_4-4\sum jm_{4j}}}{4^{m_4} \prod_j m_{4j}!}
\left(1 - m_4q^{-2} + O(q^{-3})\right), \\
\phi^+_{i,M}(1) & = & \frac{q^{im_i-i\sum jm_{ij}}}{i^{m_i} \prod_j m_{ij}!}
\left(1 + O(q^{-3})\right) \quad \textrm{for } i \ge 5. \\
\end{eqnarray*}
\end{lemma}

\begin{proof}
We consider $\phi^+_{i,M}(1)$ separately for $i=1,2,3,4$ before
looking at the general case $i \ge 5$.

\bigskip
\noindent \textbf{Case: } {\boldmath $i=1$}

From the definition of $\phi^+_{i,M}(t)$ in Theorem \ref{C_r=Cblah},

$$\displaystyle \phi^+_{1,M}(1) = m_1! {N^+(1,q) \choose m_1}
\prod_{j=1}^{\lfloor r/i \rfloor}
\frac{q^{-jm_{1j}}}
{m_{1j}!(1-q^{-1}+q^{-2})^{m_{1j}}}.$$
We consider $\phi^+_{1,M}(1)$ in three parts. First of all, noting
that $m_1 \le N^+(1,q) = q-1$ since $m_1$ is the number of distinct
linear factors of a polynomial $f$, we have $m_1! {N^+(1,q) \choose m_1} = (q-1)(q-2)\ldots(q-m_1)$,
which equals

\begin{equation} \label{N+(1,q)}
\begin{array}{c l}
& q^{m_1} - \frac{m_1(m_1+1)}{2} q^{m_1-1} +  \left(\frac{m_1^4}{8} +
\frac{m_1^3}{12} - \frac{m_1^2}{8} - \frac{m_1}{12}\right) q^{m_1-2} +
O(q^{m_1-3}) \\

= & q^{m_1} \left( 1 - \left(\frac{m_1^2}{2} + \frac{m_1}{2}\right)
q^{-1} + \left(\frac{m_1^4}{8} + \frac{m_1^3}{12} - \frac{m_1^2}{8} -
\frac{m_1}{12} \right)q^{-2} + O(q^{-3}) \right).\\

\end{array}
\end{equation}

\noindent
where the coefficient of $q^{m_1-2}$ was computed in detail in Lemma 2.5.1 of \cite{myPhD}.

Secondly,

$$\prod_j \frac{q^{-jm_{1j}}}{m_{1j}!} = \frac{q^{-\sum
jm_{1j}}}{\prod_j m_{1j}!}. $$

Finally, since $\frac{1}{(1-x)^m} = \sum_{k \ge 0} {m+k-1 \choose
k}x^k$, making a Taylor expansion we get

\begin{displaymath}
\begin{array}{rcl}

\prod_j \frac{1}{(1-q^{-1}+q^{-2})^{m_{1j}}} & = &
\frac{1}{(1-q^{-1}+q^{-2})^{m_1}} \\

& = & 1 + m_1(q^{-1}-q^{-2}) +
{m_1+1 \choose 2} (q^{-1}-q^{-2})^2 + \ldots  \\

& = & 1 + m_1 q^{-1} + \frac{m_1(m_1-1)}{2}q^{-2} + O(q^{-3}).  \\
\end{array}
\end{displaymath}

\noindent Multiplying together the three parts and collecting
terms gives us $\phi^+_{1,M}(1)$ equal to

\begin{displaymath}
\begin{array}{l}
\displaystyle \frac{q^{m_1-\sum jm_{1j}}}{\prod_j m_{1j}!} \times \\
\displaystyle \left(1 + \left( - \frac{m_1^2}{2} + \frac{m_{1}}{2} \right) q^{-1}
+ \left(\frac{m_{1}^4}{8} -
\frac{5m_{1}^3}{12} - \frac{m_{1}^2}{8} -
\frac{7m_{1}}{12}\right)q^{-2}
+ O(q^{-3})\right). \\
\end{array}
\end{displaymath}

\bigskip
\noindent \textbf{Case: } {\boldmath $i=2$}

From the definition of $\phi^+_{i,M}(t)$ in Theorem \ref{C_r=Cblah},

$$\displaystyle \phi^+_{2,M}(1) = m_2! {N(2,q) \choose m_2} \prod_{j=1}^{\lfloor r/i \rfloor}
\frac{q^{-2jm_{2j}}}
{m_{2j}!(1-q^{-2}+q^{-4})^{m_{2j}}}.$$
Note that $N(2,q) = N^+(2,q) = \frac{q^2 - q}{2}$.
We consider $\phi^+_{2,M}(1)$ in three parts. First of all,

\begin{equation} \label{N(2,q)}
\begin{array}{r c l}
m_2! {N(2,q) \choose m_2} & = & \left(\frac{1}{2}(q^2-q)\right)\left(
\frac{1}{2}(q^2-q)-1\right)\ldots\left(\frac{1}{2}(q^2-q)-m_2+1\right)
\\

& = & \frac{1}{2^{m_2}} (q^2-q)(q^2-q-2) \ldots (q^2-q-2m_2+2) \\

& = & \frac{1}{2^{m_2}} \left( q^{2m_2} - m_2q^{2m_2-1} -
\frac{m_2(m_2-1)}{2}q^{2m_2-2} + O(q^{2m_2-3}) \right) \\

& = & \frac{q^{2m_2}}{2^{m_2}} \left( 1 - m_2q^{-1} +
\left(-\frac{m_2^2}{2} + \frac{m_2}{2}\right)q^{-2} + O(q^{-3}) \right). \\

\end{array}
\end{equation}

The coefficient of $q^{2m_2-2}$ arises from summing together $\frac{m_2(m_2-1)}{2}$,
which is the number of ways to choose $q^2$ from $m_2-2$ terms
and $-q$ from two terms above, with $-m_2(m_2-1)$, which is the
coefficient obtained by choosing $q^2$ from $m_2-1$ terms and a nonzero
constant term from one term above.

The second part is

$$\prod_j \frac{q^{-2jm_{2j}}}{m_{2j}!} = \frac{q^{-2\sum
jm_{2j}}}{\prod_j m_{2j}!}. $$

Finally the Taylor expansion of

\begin{displaymath}
\begin{array}{rcl}

\prod_j \frac{1}{(1-q^{-2}+q^{-4})^{m_{2j}}} & = &
\frac{1}{(1-q^{-2}+q^{-4})^{m_2}} \\

& = & 1 + m_2(q^{-2}-q^{-4}) +
{m_2+1 \choose 2} (q^{-2}-q^{-4})^2 + \ldots  \\

& = & 1 + m_2 q^{-2} + O(q^{-4}).  \\
\end{array}
\end{displaymath}

Multiplying together the three parts and collecting terms gives us
$\phi^+_{2,M}(1)$ equal to

$$\frac{q^{2m_2-2\sum jm_{2j}}}{2^{m_2} \prod_j m_{2j}!}
\left(1 - m_2q^{-1} + \left(-\frac{m_2^2}{2}+\frac{3m_2}{2}\right)
q^{-2} + O(q^{-3})\right). $$

\bigskip
\noindent \textbf{Case: } {\boldmath $i=3$}

From the definition of $\phi^+_{i,M}(t)$ in Theorem \ref{C_r=Cblah},

$$\displaystyle \phi^+_{3,M}(1) = m_3! {N(3,q) \choose m_3} \prod_{j=1}^{\lfloor r/i \rfloor}
\frac{q^{-3jm_{3j}}}
{m_{3j}!(1-q^{-3}+q^{-6})^{m_{3j}}}.$$

\noindent
Note that $N(3,q) = N^+(3,q) = \frac{q^3-q}{3}$.
We consider $\phi^+_{3,M}(1)$ in three parts. First of all,

\begin{equation}\label{N(3,q)}
\begin{array}{r c l}
m_3! {N(3,q) \choose m_3} & = & \left(\frac{1}{3}(q^3-q)\right)\left(
\frac{1}{3}(q^3-q)-1\right)\ldots\left(\frac{1}{3}(q^3-q)-m_3+1\right)
\\

& = & \frac{1}{3^{m_3}} \left( q^{3m_3} - m_3q^{3m_3-2} +
O(q^{3m_3-3}) \right) \\

& = & \frac{q^{3m_3}}{3^{m_3}} \left( 1 - m_3q^{-2} +
O(q^{-3}) \right). \\

\end{array}
\end{equation}

Secondly,

$$\prod_j \frac{q^{-3jm_{3j}}}{m_{3j}!} = \frac{q^{-3\sum
jm_{3j}}}{\prod_j m_{3j}!}. $$

Finally $\prod_j \frac{1}{(1-q^{-3}+q^{-6})^{m_{3j}}} = \frac{1}{(1-q^{-3}+q^{-6})^{m_3}} = 1 + O(q^{-3}).$
Multiplying together the three parts and collecting terms gives us
$\phi^+_{3,M}(1)$ equal to

$$\frac{q^{3m_3-3\sum jm_{3j}}}{3^{m_3} \prod_j m_{3j}!}
\left(1 - m_3q^{-2} + O(q^{-3})\right). $$

\bigskip
\noindent \textbf{Case: } {\boldmath $i=4$}

From the definition of $\phi^+_{i,M}(t)$ in Theorem \ref{C_r=Cblah},

$$\displaystyle \phi^+_{4,M}(1) = m_4! {N(4,q) \choose m_4} \prod_{j=1}^{\lfloor r/i \rfloor}
\frac{q^{-4jm_{4j}}}
{m_{4j}!(1-q^{-4}+q^{-8})^{m_{4j}}}.$$
Note that $N(4,q) = N^+(4,q) = \frac{q^4-q^2+q}{4}$.
We consider $\phi^+_{4,M}(1)$ in three parts. First of all,

\begin{equation}\label{N(4,q)}
\begin{array}{r c l}
m_4! {N(4,q) \choose m_4} & = & \left(\frac{1}{4}(q^4-q^2+q)\right)
\ldots\left(\frac{1}{4}(q^4-q^2+q)-m_4+1\right) \\

& = & \frac{1}{4^{m_4}} \left( q^{4m_4} - m_4q^{4m_4-2} +
O(q^{4m_4-3}) \right) \\

& = & \frac{q^{4m_4}}{4^{m_4}} \left( 1 - m_4q^{-2} +
O(q^{-3}) \right). \\

\end{array}
\end{equation}

Secondly,

$$\prod_j \frac{q^{-4jm_{4j}}}{m_{4j}!} = \frac{q^{-4\sum
jm_{4j}}}{\prod_j m_{4j}!}. $$

Finally $\prod_j \frac{1}{(1-q^{-4}+q^{-8})^{m_{4j}}} =
\frac{1}{(1-q^{-4}+q^{-8})^{m_4}} = 1 + O(q^{-4}).$
Multiplying together the three parts and collecting terms gives us
$\phi^+_{4,M}(1)$ equal to

$$\frac{q^{4m_4-4\sum jm_{4j}}}{4^{m_4} \prod_j m_{4j}!}
\left(1 - m_4q^{-2} + O(q^{-3})\right). $$

\bigskip
\noindent \textbf{Case: } {\boldmath $i\ge 5$}

From the definition of $\phi^+_{i,M}(t)$ in Theorem \ref{C_r=Cblah},

$$\displaystyle \phi^+_{i,M}(1) = m_i! {N(i,q) \choose m_i} \prod_{j=1}^{\lfloor r/i \rfloor}
\frac{q^{-ijm_{ij}}}
{m_{ij}!(1-q^{-i}+q^{-2i})^{m_{ij}}}.$$
Note that $N(i,q) = N^+(i,q)$ for $i \ge 2$.
We consider $\phi^+_{i,M}(1)$ in three parts. First of all, since
the exponent of the second term in $N(i,q)$ is at least $3$ less than
the leading term, we have

\begin{equation}\label{N(i,q)}
\begin{array}{r c l}
m_i! {N(i,q) \choose m_i} & = & \frac{q^{im_i}}{i^{m_i}} \left( 1 +
O(q^{-3}) \right). \\

\end{array}
\end{equation}

Secondly,

$$\prod_j \frac{q^{-ijm_{ij}}}{m_{ij}!} = \frac{q^{-i\sum
jm_{ij}}}{\prod_j m_{ij}!}. $$

Finally $\prod_j \frac{1}{(1-q^{-i}+q^{-2i})^{m_{ij}}} =
\frac{1}{(1-q^{-i}+q^{-2i})^{m_i}} = 1 + O(q^{-i}).$
Multiplying together the three parts and collecting terms gives us
$\phi^+_{i,M}(1)$ equal to

$$\frac{q^{im_i-i\sum jm_{ij}}}{i^{m_i} \prod_j m_{ij}!}
\left(1 + O(q^{-3})\right). $$
\end{proof}

\begin{corollary} \label{phiipart}
For $M=(m_{ij}) \in \mathcal{M}_{part}(r)$, we get that
\begin{eqnarray*}
\phi^+_{1,M}(1) & = &
\frac{1}{m_{11}!} \left( 1 +
\left(-\frac{m_{11}^2}{2} +\frac{m_{11}}{2}\right)q^{-1} \right.\\
& & \left.+ \left(\frac{m_{11}^4}{8} - \frac{5m_{11}^3}{12} - \frac{m_{11}^2}{8} -
\frac{7m_{11}}{12}\right)q^{-2} + O(q^{-3}) \right), \\
\phi^+_{2,M}(1) & = & \frac{1}{2^{m_{21}}m_{21}!} \left(1 - m_{21}q^{-1} +
\left(-\frac{m_{21}^2}{2} + \frac{3m_{21}}{2}\right)q^{-2} + O(q^{-3})\right),\\
\phi^+_{3,M}(1) & = & \frac{1}{3^{m_{31}}m_{31}!} \left(1 - m_{31}q^{-2} +
+ O(q^{-3})\right),\\
\phi^+_{4,M}(1) & = & \frac{1}{4^{m_{41}}m_{41}!} \left(1 - m_{41}q^{-2} +
O(q^{-3})\right), \\
\phi^+_{i,M}(1) & = & \frac{1}{i^{m_{i1}}m_{i1}!} \left(1 +
O(q^{-3})\right) \textrm{for } i\ge 5. \\
\end{eqnarray*}
\end{corollary}

\begin{proof}
Let $(m_{ij}) \in \mathcal{M}_{part}(r)$, that is $(m_{ij})$
corresponds to a partition of $r$. Then for all $i$ we have $m_{ij}=0$
for $j\ge 2$. Hence $im_i - i\sum_j j m_{ij} = 0$. Thus the first term of
$\phi^+_{i,M}(1)$
is a constant term for all $i$ and the
remaining terms are as in Lemma \ref{phiigen}.
\end{proof}

We have expressions for the $\phi^+_{i,M}(1)$ for different $M \in \mathcal{M}(r)$.
The following lemma and corollary will help us understand more about them.

\begin{lemma}\label{kindapartiff}
For a given $i$, the leading term of $\phi^+_{i,M}(1)$ is a constant if
and only if $m_{ij}=0$ for all $j\ge 2$.
\end{lemma}

\begin{proof}
Let $M=(m_{ij}) \in \mathcal{M}(r)$ such that for a given $i$, $m_{ij}
= 0$ for all $j\ge 2$. By Lemma \ref{phiigen} the leading
term of $\phi^+_{i,M}(1)$ is $q^{i m_i - i \sum_j jm_{ij}}$. Since
$m_{ij}=0$ for $j\ge 2$, we get that $im_i - i\sum_j jm_{ij} = 0$ and
hence the leading term of $\phi^+_{i,M}(1)$ is a constant.

Conversely suppose that the leading term of $\phi^+_{i,M}(1)$ is a
constant. Then by Lemma \ref{phiigen}, $i \sum_j m_{ij} - i \sum_j
jm_{ij} = 0$, that is, $i \sum_j (1-j) m_{ij} = 0$. Hence $m_{ij} = 0$
for all $j\ge 2$.
\end{proof}

\begin{corollary}\label{partiff}
The leading term of $\phi^+_{i,M}(1)$ is a constant for all $i$
if and only if $(m_{ij}) \in \mathcal{M}_{part}(r)$.
\end{corollary}

\begin{proof}
If the leading term of every $\phi^+_{i,M}(1)$ is a constant then by
Lemma \ref{kindapartiff}, $m_{ij}=0$ for all $j\ge 2$ and for all
$i$. Hence $(m_{ij})$ is a partition of $r$, that is
$(m_{ij}) \in \mathcal{M}_{part}(r)$.
Conversely, if $(m_{ij})$ is a partition of $r$, then $m_{ij} = 0$ for all $i$ and
$j \ge 2$. Then by Lemma \ref{kindapartiff} the leading term of
$\phi^+_{i,M}(1)$ is a constant for all $i$.
\end{proof}

The next two lemmas provide us with the tools which will enable us
to sum the $\phi_{i,M}^+(1)$ over all $M \in \mathcal{M}(r)$.

\begin{lemma} \label{1/1-t}
Suppose that for some $k \in \mathbb{Z}^+$ and $f_k: \mathbb{Z}^+ \to
\mathbb{Q}$ we have, for each $r$,

$$a_r = \displaystyle\sum_{M\in\mathcal{M}_{part}(r)} f_k(m_{k1})
\prod_i \frac{1}{i^{m_{i1}} m_{i1}!}.$$

\noindent Then $$\displaystyle 1 + \sum_{r=1}^{\infty} a_r t^r = \left(
\sum_{m\ge 0} \frac{f_k(m) t^{mk}}{k^m m!} \right)
\frac{e^{\frac{-t^k}{k}}}{1-t}.$$
\end{lemma}

\begin{proof}
For $i \in \mathbb{Z}^+$, consider the series

$$ \frac{1}{1-t^i} =  1 + t^i + t^{2i} + \ldots $$

\noindent The product of $\frac{1}{1-t^i}$ taken over all $i$ is equal
to the generating function for the number of partitions (see
\cite{wilf}). That is, it is the series $1 + \sum_{r=1}^\infty p_r t^r$
where $p_r$ is the number of partitions of $r$. This is because each
partition of $r$ corresponds to a unique selection of terms $t^{m_{i1} i}$
from the series above, one term for each $i$, such that $\sum_i m_{i1}
i = r$. The corresponding partition has $m_{i1}$ parts of size $i$ for each
$i$.

In the above scenario, each partition of $r$ contributes $t^r$ to the
generating function. For the lemma we want the partition of $r$ having
$m_{i1}$ parts of size $i$, for each $i$, to contribute
$f_k(m_{k1})t^r \prod_i \frac{1}{i^{m_{i1}}m_{i1}!}$ to our
generating function. Hence each time we choose $m_{i1}$ parts of size $i$
for our partition of $r$ we want it to contribute
$\frac{t^{im_{i1}}}{i^{m_{i1}}m_{i1}!}$ for $i \ne k$ and
$\frac{f_k(m_{i1})t^{im_{i1}}}{i^{m_{i1}}m_{i1}!}$ when $i=k$.

For all $i \ne k$, consider the series

$$ e^{\frac{t^i}{i}} = 1 + \frac{t^i}{i^1 1!} + \frac{t^{2i}}{i^2 2!}
+ \ldots  $$

\noindent and for $i=k$ consider the series

$$ \sum_{m \ge 0} \frac{f_k(m) t^{mk}}{k^m m!} = f_k(0) + \frac{
f_k(1)t^k}{k^1 1!} + \frac{f_k(2)t^{2k}}{k^2 2!} + \ldots. $$

Then the partition $(m_{ij}) \in \mathcal{M}_{part}(r)$ corresponds to
selecting, for each $i$, the term involving $t^{im_{i1}}$ from the
$i$th series above, and its contribution to the generating function is
$f_k(m_{k1}) t^r \prod_i \frac{1}{i^{m_{i1}}m_{i1}!}$. Hence

$$\left( \sum_{m \ge 0} \frac{f_k(m)t^{mk}}{k^m m!} \right) \prod_{i\ne
k} e^{\frac{t^i}{i}} = 1 + \sum_{r=1}^\infty a_r t^r$$

\noindent and this is the generating function we seek. Now

$$\prod_{i\ne k} e^{\frac{t^i}{i}} =  e^{\frac{-t^k}{k}} \prod_{i}
e^{\frac{t^i}{i}} = e^{\frac{-t^k}{k}} e^{\sum \frac{t^i}{i}} =
e^{\frac{-t^k}{k}} e^{-\log(1-t)} =
\frac{e^{\frac{-t^k}{k}}}{1-t}$$

\noindent and hence

$$ 1 + \sum_{r=1}^\infty a_r t^r = \left( \sum_{m \ge 0} \frac{
f_k(m)t^{mk}}{i^m m!} \right) \frac{e^{\frac{-t^k}{k}}}{1-t}.$$
\end{proof}

\begin{corollary} \label{sumparts=1}
For each $r > 0$, let

$$\displaystyle a_r =
\sum_{M\in\mathcal{M}_{part}(r)} \prod_i
\frac{1}{i^{m_{i1}} m_{i1}!}.$$

\noindent
Then $1 + \sum_{r\ge 0} a_r t^r$ is
equal to $\frac{1}{1-t}$. That is, $a_r = 1$ for all $r \ge 0$.
\end{corollary}

\begin{proof}
Letting $f_k(m) = 1$ for all $m$ in Lemma \ref{1/1-t}, we obtain

$$ a_r = \displaystyle\sum_{M\in\mathcal{M}_{part}(r)} \prod_i
\frac{1}{i^{m_{i1}} m_{i1}!} $$

\noindent and

\begin{displaymath}
\begin{array}{rcl}

\displaystyle 1 + \sum_{r=1}^\infty a_r t^r & = & \left( \displaystyle
\sum_{m \ge 0} \frac{t^{mk} }{k^m m!} \right)
\frac{e^{\frac{-t^k}{k}}}{1-t} \\

& = & \left( e^{\frac{t^k}{k}} \right)
\frac{e^{\frac{-t^k}{k}}}{1-t}\\

& = & \frac{1}{1-t}. \\

\end{array}
\end{displaymath}

Clearly all the coefficients in the expansion of $\frac{1}{1-t}$ are
$1$. Hence $a_r$ is equal to $1$ for all $r$.
\end{proof}

\begin{lemma} \label{kl1/1-t}
Suppose that, for some $k,\ell \in \mathbb{Z}^+$ and functions $f_k, f_\ell:
\mathbb{Z}^+ \to \mathbb{Q}$ and for each $r$,

$$a_r = \displaystyle\sum_{M\in\mathcal{M}_{part}(r)}
f_k(m_{k1})f_\ell (m_{\ell 1}) \prod_i \frac{1}{i^{m_{i1}} m_{i1}!}.$$

\noindent Then $$\displaystyle 1 + \sum_{r=1}^{\infty} a_r t^r =
\left( \sum_{m\ge 0} \frac{f_k(m)t^{mk}}{k^m m!} \right)
\left( \sum_{m\ge 0} \frac{f_\ell(m)t^{m\ell}}{\ell^m m!} \right)
\frac{e^{\frac{-t^k}{k}} e^{\frac{-t^\ell}{\ell}}}{1-t}.$$
\end{lemma}

\begin{proof}
As with the proof of Lemma \ref{1/1-t}, for all $i \ne k$ and $i \ne
\ell$, consider the series

$$ e^{\frac{t^i}{i}} = 1 + \frac{t^i}{i^1 1!} + \frac{t^{2i}}{i^2 2!}
+ \ldots, $$

\noindent for $i=k$ consider the series

$$ \sum_{m \ge 0} \frac{f_k(m) t^{mk}}{k^m m!} = f_k(0) +
\frac{f_k(1)t^k}{k^1 1!} + \frac{f_k(2)t^{2k}}{k^2 2!} + \ldots $$

\noindent and for $i=\ell$ consider the series

$$ \sum_{m \ge 0} \frac{f_\ell(m)t^{m\ell}}{\ell^m m!} = f_\ell(0) +
\frac{f_\ell(1)t^\ell}{\ell^1 1!} + \frac{f_\ell(2)t^{2\ell}}{\ell^2
2!} + \ldots . $$

Each partition $(m_{ij}) \in \mathcal{M}_{part}(r)$, corresponds
to a product of a unique selection of terms from the series above, one
for each $i$, and the contribution to the generating function is
$f_k(m_{k1}) f_\ell(m_{\ell 1})t^r \prod_i \frac{1}{i^{m_{i1}}m_{i1}!}$. Hence

$$ 1 + \sum_{r=1}^\infty a_r t^r = \left( \sum_{m \ge 0} \frac{t^{mk}
f_k(m)}{k^m m!} \right)  \left( \sum_{m\ge 0} \frac{t^{m\ell}
f_\ell(m)}{\ell^m m!} \right) \prod_{i\ne k,\ell} e^{\frac{t^i}{i}} $$

\noindent and hence is the generating function we require. We see now
that

$$\prod_{i\ne k,\ell} e^{\frac{t^i}{i}}
= e^{\frac{-t^k}{k}} e^{\frac{-t^\ell}{\ell}} \prod_{i} e^{\frac{t^i}{i}}
= e^{\frac{-t^k}{k}} e^{\frac{-t^\ell}{\ell}} e^{\sum \frac{t^i}{i}}
= e^{\frac{-t^k}{k}} e^{\frac{-t^\ell}{\ell}} e^{-\log (1-t)} =
\frac{e^{\frac{-t^k}{k}} e^{\frac{-t^\ell}{\ell}} }{1-t}.$$

\noindent Hence the generating function we require is

$$\left( \sum_{m \ge 0} \frac{f_k(m) t^{mk}}{k^m m!} \right)
\left( \sum_{m \ge 0} \frac{f_\ell(m) t^{m\ell}}{\ell^m m!} \right)
\frac{ e^{\frac{-t^k}{k}} e^{\frac{-t^\ell}{\ell}} }{1-t}.$$
\end{proof}

We now have the techniques required to determine $\lim_{n\rightarrow\infty}c_{\GL,r}(n)$.
The following theorem does this.

\begin{theorem} \label{1-q^-2}
For $r \in \mathbb{Z}^+$, $\displaystyle
\lim_{n\rightarrow\infty}c_{\GL,r}(n) = 1 - q^{-2} + O(q^{-3})$.
\end{theorem}

\begin{proof}
The proof will consist of three parts. The first part will calculate
the constant term in the expansion of $\lim_{n \to \infty}
c_{\GL,r}(n)$, the second part will calculate the coefficient of
$q^{-1}$ and the third part will calculate the coefficient of
$q^{-2}$.

Before starting the first part, we note that the expansion of
$\frac{1-q^{-5}}{1+q^{-3}}$ is

$$ 1 - q^{-3} - q^{-5} + O(q^{-6}) $$

\noindent and by Theorem \ref{limcrn}
we multiply this with $\sum_{M \in\mathcal{M}(r)}
\prod_i \phi^+_{i,M}(1)$ to produce lim$_{n \to\infty} c_{\GL,r}(n)$.
Thus the constant term, the $q^{-1}$ term and the $q^{-2}$ term in
$\lim_{n \to \infty} c_{\GL,r}(n)$ are the same as those in the expansion of
$\sum_{M \in\mathcal{M}(r)} \prod_i \phi^+_{i,M}(1)$.

\bigskip
\noindent \textbf{Constant Term:}

By Corollary \ref{partiff}, the leading term of $\phi^+_{i,M}(1)$ is a
constant if and only if $M \in \mathcal{M}_{part}(r)$. From Corollary
\ref{phiipart}, we can see that this leading term is
$\frac{1}{i^{m_{i1}} m_{i1}!}$ if $M \in \mathcal{M}_{part}(r)$ so the
constant term in $\lim_{n \to \infty} c_{\GL,r}(n)$ is

$$ \sum_{M \in \mathcal{M}_{part}(r)} \prod_i \frac{1}{i^{m_{i1}}
m_{i1}!}.$$

\noindent By Corollary \ref{sumparts=1} this equals $1$ for all
$r$. Hence $1$ is the constant term of $\lim_{n \to \infty} c_{\GL,r}(n)$
for all $r$.

\noindent \textbf{Coefficient of $q^{-1}$}

For the coefficient of $q^{-1}$, we will first sum over all
$M \in \mathcal{M}_{part}(r)$ and then look at those $M \notin
\mathcal{M}_{part}(r)$ that also contribute to the $q^{-1}$ term.

Let $M = (m_{ij}) \in \mathcal{M}_{part}(r)$. By Corollary
\ref{phiipart} there are exactly two ways to produce a
$q^{-1}$ term. The first is to multiply the $q^{-1}$ term in
$\phi^+_{1,M}(1)$ with the constant term from each of the remaining
$\phi^+_{i,M}(1)$. The second way is to multiply the $q^{-1}$ term in
$\phi^+_{2,M}(1)$ with the constant term in each of the remaining
$\phi^+_{i,M}(1)$.

The first way produces $\left( -\frac{m_{11}^2}{2} + \frac{m_{11}}{2}
\right) \prod_i \frac{1}{i^{m_{i1}}m_{i1}!} q^{-1}$. Summing this over
all $M \in \mathcal{M}_{part}(r)$, gives the generating function
 as

$$ -\frac{1}{2} \left( \sum_{m\ge 0} \frac{m^2 t^m}{m!} \right)
\frac{e^{-t}}{1-t}
+ \frac{1}{2} \left( \sum_{m\ge 0} \frac{m t^m}{m!} \right)
\frac{e^{-t}}{1-t} $$

\noindent by Lemma \ref{1/1-t}, using $f_1(m)=m^2$ in the first
case and $f_1(m)=m$ in the second case. By
Lemma \ref{te}, we get the generating function for the
coefficient of $q^{-1}$ arising in this way to be

$$ \left( -\frac{1}{2} (t^2+t) e^t + \frac{1}{2}te^t \right)
\frac{e^{-t}}{1-t} = \left(\frac{-t^2}{2}\right)\frac{1}{1-t}. $$

The second way an $M \in \mathcal{M}_{part}(r)$ can produce a $q^{-1}$
term, produces $-m_{21} \prod_i \frac{1}{i^{m_{i1}}m_{i1}!}
q^{-1}$. Summing this over all $M \in \mathcal{M}_{part}(r)$, gives
the generating function for the contribution to the
coefficient of $q^{-1}$ as

$$ \left( - \sum_{m\ge 0} \frac{m t^{2m}}{2^m m!} \right)
\frac{e^{\frac{-t^2}{2}}}{1-t} $$

\noindent by Lemma \ref{1/1-t}, using $f_2(m)=-m$. By
Lemma \ref{te} with $b=1$ and $k=2$, we get $\sum_{m \ge 0}
\frac{m t^{2m}}{2^m m!} = \frac{t^2}{2} e^{\frac{t^2}{2}}$, so the
generating function for the coefficient of $q^{-1}$ arising in
this way is

$$ \left( \frac{-t^2}{2}e^{\frac{t^2}{2}} \right)
\frac{e^{\frac{-t^2}{2}}}{1-t} = \left(\frac{-t^2}{2}\right)
\frac{1}{1-t}.$$

Summing together these two functions gives the generating
function for the contribution to the coefficient of $q^{-1}$ from $M \in
\mathcal{M}_{part}(r)$, namely $\frac{-t^2}{1-t}$.

We need to look at $M = (m_{ij}) \notin
\mathcal{M}_{part}(r)$ that also contribute to the $q^{-1}$
term. By Lemma \ref{kindapartiff} there exists an $i$ such that the
leading term of $\phi^+_{i,M}(1)$ is not a constant. Hence for $M$ to
contribute to the $q^{-1}$ term there is a unique $i$ such that the
leading term of $\phi^+_{i,M}(1)$ is not a constant and this leading
term is $q^{-1}$. Thus $i\sum_j m_{ij} - i\sum_j jm_{ij} = -1$.
Rewriting the equation we see that we need $\sum_j (1-j)m_{ij} =
\frac{-1}{i}$. Since the left hand side is an integer, this equality
can only hold for $i=1$. Thus we need
$\phi^+_{i,M}(1)$ for $i \ge 2$ to have constant leading term and
$\sum_j (1-j)m_{1j}=-1$. By Lemma \ref{kindapartiff}
The first condition implies that $m_{ij} = 0$
for $i\ge 2$ and $j\ge 2$ while the second condition implies that
$m_{12}=1$ and $m_{1j}=0$ for $j\ge 3$. Hence if $(m_{ij}) \notin
\mathcal{M}_{part}(r)$ and contributes to the $q^{-1}$ term then

\begin{equation} \label{m12=1a}
(m_{ij}) = \left\{
\begin{array}{l}
m_{i1} = \textrm{anything } (i \ge 1) \\
m_{12} = 1 \\
m_{1j} = 0 \textrm{ } (j \ge 3) \\
m_{ij} = 0 \textrm{ } (i \ge 2, j \ge 2). \\
\end{array}
\right.
\end{equation}

\noindent Note that there are no such $(m_{ij})$ when $r=1$.
If we take such an $(m_{ij}) \in \mathcal{M}(r)$ and form a new
$M'=(m_{ij}')$ with $m_{ij}' = m_{ij}$ for all $(i,j) \ne (1,2)$ and
$m_{12}' = 0$ then $M' \in \mathcal{M}_{part}(r-2)$.
Moreover, each element of $\mathcal{M}_{part}(r-2)$ occurs in this way
so we think
of an $(m_{ij})$ as in (\ref{m12=1a}) as a partition of $r-2$ with $m_{12}=1$
appended to it.

We can see from Lemma \ref{phiigen}, that for each $M=(m_{ij})$ as in
(\ref{m12=1a}), for $i \ge 2$, $\phi^+_{i,M}(1)$ has constant term
$\frac{1}{i^{m_{i1}} m_{i1}!}$ and for $i=1$ the coefficient of
$q^{-1}$ is $\frac{1}{m_{11}!}$. Hence each of these $(m_{ij})$
produces $\prod_i \frac{1}{i^{m_{i1}} m_{i1}!}$ as the coefficient of
$q^{-1}$. Summing this over all $M \notin \mathcal{M}_{part}(r)$ that
contribute to the $q^{-1}$ term is then the same as summing over all
$M' \in \mathcal{M}_{part}(r-2)$. By Lemma \ref{1/1-t} the generating
function for the sum of $\prod_i \frac{1}{i^{m_{i1}} m_{i1}!}$ over
all partitions of $r$ is $\frac{1}{1-t}$. The generating function for
the sum of $\prod_i \frac{1}{i^{m_{i1}} m_{i1}!}$ over all partitions
of $r-2$ has no terms of degree less than $2$ and so is
$\frac{t^2}{1-t}$. Hence the generating function for the contribution to the coefficient
of $q^{-1}$ from $M \notin \mathcal{M}_{part}(r)$ is
$\frac{t^2}{1-t}$.

Adding the generating functions for contributions due to all
$M \in \mathcal{M}_{part}(r)$ and $M \notin \mathcal{M}_{part}(r)$
to the coefficient of $q^{-1}$ gives us $0$.
Hence for all $r$ there is no $q^{-1}$ term in the expansion
of $\lim_{n\rightarrow\infty}c_{\GL,r}(n)$.

\bigskip
\noindent \textbf{Coefficient of $q^{-2}$}

We take the same approach in calculating the coefficient of
$q^{-2}$. We first sum over all $M \in
\mathcal{M}_{part}(r)$. Let $M \in \mathcal{M}_{part}(r)$. From the
expansion of the $\phi^+_{i,M}(1)$ in Corollary \ref{phiipart} we
see that there are several ways to form a $q^{-2}$ term. Each of these
ways corresponds to a line of Table \ref{q^-2parttable}.
For each line in Table \ref{q^-2parttable}, the contribution from $\phi^+_{i,M}$
for $i \ge 5$ is 1.
If a line contains
exactly one non-$1$ entry then that entry is a certain function of
$m_{k1}$ for some $k$, say $f_k(m_{k1})$ and this line corresponds to
a summand $a_r$ of the coefficient of $q^{-2}$, where $a_r$ is as given in Lemma
\ref{1/1-t}. The sum $1 + \sum_{r=1}^\infty a_r t^r$ corresponding to
these values of $a_r$ is evaluated using Lemma \ref{1/1-t} and Lemma \ref{te}
and recorded in the last entry for this line. There is one further
line that contains two non-$1$ entries which are certain functions
$f_k(m_{k1})$ with $k=1$ and $f_\ell(m_{\ell 1})$ with $\ell = 2$ and
this line corresponds to a summand $a_r$ of the coefficient of
$q^{-1}$, where $a_r$ is as given in Lemma \ref{kl1/1-t}. The sum $1 +
\sum_{r=1}^\infty a_r t^r$ corresponding to these values of $a_r$ is
evaluated using Lemmas \ref{kl1/1-t} and Lemma \ref{te} and recorded in the
last entry for this line.

\begin{table}
\begin{center}
\begin{tabular} {c c c c | c}
$\phi^+_{1,M}(1)$ & $\phi^+_{2,M}(1)$ & $\phi^+_{3,M}(1)$ & $\phi^+_{4,M}(1)$ & Generating Function \\
\hline
$1$ & $1$ & $1$ & $-m_{41}q^{-2}$ &
$\left(-\frac{t^4}{4}\right)\frac{1}{1-t}$ \\
$1$ & $1$ & $-m_{31}q^{-2}$ & $1$ &
$\left(-\frac{t^3}{3}\right)\frac{1}{1-t}$ \\
$1$ & $\frac{3m_{21}}{2}q^{-2}$ & $1$ & $1$ &
$\left(\frac{3t^2}{4}\right)\frac{1}{1-t}$ \\
$1$ & $-\frac{m_{21}^2}{2}q^{-2}$ & $1$ & $1$ &
$\left(-\frac{t^4}{8}-\frac{t^2}{4}\right)\frac{1}{1-t}$ \\
$-\frac{m_{11}(m_{11}-1)}{2}q^{-1}$ & $-m_{21}q^{-1}$ & $1$ &
$1$ &  $\left(\frac{t^4}{4}\right)\frac{1}{1-t}$ \\
$-\frac{7m_{11}}{12}q^{-2}$ & $1$ & $1$ & $1$ &
$\left(-\frac{7t}{12}\right)\frac{1}{1-t}$ \\
$-\frac{m_{11}^2}{8}q^{-2}$ & $1$ & $1$ & $1$ &
$\left(\frac{-t^2-t}{8}\right)\frac{1}{1-t}$ \\
$-\frac{5m_{11}^3}{12}q^{-2}$ & $1$ & $1$ & $1$ &
$\left(\frac{-5(t^3+3t^2+t)}{12}\right)\frac{1}{1-t}$ \\
$\frac{m_{11}^4}{8}q^{-2}$ & $1$ & $1$ & $1$ &
$\left(\frac{t^4+6t^3+7t^2+t}{8}\right)\frac{1}{1-t}$ \\
\hline
& & & Total & $\left(-t\right)\frac{1}{1-t}$ \\
\hline
\end{tabular}
\end{center}
\caption{Producing the generating function for the coefficient
of $q^{-2}$ in the expansion of $\lim_{n\to\infty} c_{\GL,r}(n)$
due to $M \in \mathcal{M}_{part}(r)$.}
\label{q^-2parttable}
\end{table}

\noindent
The sum of all the generating functions is the generating function
for the coefficient of $q^{-2}$ due to $M \in
\mathcal{M}_{part}(r)$. It is $\frac{-t}{1-t}$.

We now look for the contribution to the $q^{-2}$ term in $\prod_i
\phi^+_{i,M}(1)$ from $M \notin \mathcal{M}_{part}(r)$. For each $i$
we need to choose a term involving $q^{a_i}$ from $\phi^+_{i,M}(1)$ such
that $\sum_i a_i = -2$. In particular each $a_i$ satisfies $0 \ge a_i
\ge -2$. We will use Lemma \ref{phiigen}.

For $i \ge 3$ the exponent of $q$ either is $i \sum_j (1-j)m_{ij}$ or is at most
$i \sum_j (1-j) m_{ij} -3$. The only possible value out of $0, -1, -2$
is $0$ and hence for each such $M$ we must have $m_{ij}=0$ for all $i
\ge 3$, $j \ge 2$ and $a_i = 0$ for all $i \ge 3$.

Since we require $M \notin \mathcal{M}_{part}(r)$ there must exist an
$i_0$ equal to $1$ or $2$ and $j_0 \ge 2$ such that $m_{i_0 j_0} > 0$.
By Lemma \ref{phiigen},

\begin{equation} \label{a_i}
-2 \le a_{i_0} \le i_0 \sum_j (1-j)m_{i_0j} \le i_0 (1-j_0)
m_{i_0j_0}.
\end{equation}

\noindent
Hence either

\begin{itemize}
\item
$i_0=2$ and $(j_0, m_{2j_0}) = (2,1)$ and $m_{2j}=0$ for all $j\ge 3$; or

\item
$i_0=1$ and $(j_0, m_{1j_0}) = (2,1), (2,2)$ or $(3,1)$
with $m_{1j}=0$ for all $j \ne {1,j_0}$.
\end{itemize}

\noindent We consider each of these four cases in turn.

Let us consider the first scenario where $i_0=2$, $j_0=2$, $m_{22}=1$
and $m_{2j}=0$ for $j \ge 3$. It follows from (\ref{a_i}) that $a_2 =
-2$ and hence $a_i=0$ for $i \ne 2$. Thus the first type of
$(m_{ij}) \notin \mathcal{M}_{part}(r)$ that contributes to the
$q^{-2}$ term is as follows

\begin{equation}\label{nonpart1}
(m_{ij}) = \left\{
\begin{array}{l}
m_{i1} = \textrm{anything } (i \ge 1) \\
m_{22} = 1 \\
m_{2j} = 0 \textrm{ } (j \ge 3) \\
m_{ij} = 0 \textrm{ } (i \ne 2, j \ge 2). \\
\end{array}
\right.
\end{equation}

\noindent Note that there are no such $(m_{ij})$ for $r<4$.
If we take such an $(m_{ij}) \in \mathcal{M}(r)$ and form a new
$M'=(m_{ij}')$ with $m_{ij}' = m_{ij}$ for all $(i,j) \ne (2,2)$ and
$m_{22}' = 0$ then $M' \in \mathcal{M}_{part}(r-4)$.
So we think of an $(m_{ij})$ as in (\ref{nonpart1})
as a partition of $r-4$ with $m_{22}=1$ appended to it.

For $(m_{ij})$ as in (\ref{nonpart1}), we have $m_2 = m_{21}+1$
and $m_i = m_{i1}$ for $i\ne 2$. By Lemma \ref{phiigen} we obtain the
$\phi^+_{i,M}(1)$ as follows:

$$\phi^+_{1,M}(1) = \frac{1}{ m_{11}!} (1 + O(q^{-1}))$$

$$\phi^+_{2,M}(1) = \frac{q^{-2}}{2^{m_{21}} m_{21}!} (\frac{1}{2})
(1 + O(q^{-1}))$$

\noindent
and for $i\ge 3$

$$\phi^+_{i,M}(1) = \frac{1}{i^{m_{i1}} m_{i1}!} (1 + O(q^{-1})).$$

\noindent
The `$\frac{1}{2}$' inside the $\phi^+_{2,M}(1)$ occurs because $2^{m_2} =
2^{m_{21}}2^1$.

There is only one way to make up a $q^{-2}$ term and the first line in
Table \ref{q^-2nonparttable} corresponds to forming a $q^{-2}$
term in this way.

From now on we will let $i_0=1$ so $m_{2j}=0$ for all $i \ge 2, j\ge
2$. If $(j_0, m_{1j_0}) = (3,1)$ or $(2,2)$ then the leading term of
$\phi^+_{1,M}(1)$ involves $q^{-2}$ so $a_1 = -2$ and $a_i = 0$ for all
$i \ge 2$. Hence we get two more types of nonpartitions which
contribute to the $q^{-2}$ term, namely

\begin{equation}\label{nonpart3}
(m_{ij}) = \left\{
\begin{array}{l}
m_{i1} = \textrm{anything } (i \ge 1) \\
m_{13} = 1 \\
m_{1j} = 0 \textrm{ } (j \ne 1,3) \\
m_{ij} = 0 \textrm{ } (i \ge 2, j \ge 2) \\
\end{array}
\right.
\end{equation}

\noindent and

\begin{equation}\label{nonpart4}
(m_{ij}) = \left\{
\begin{array}{l}
m_{i1} = \textrm{anything } (i \ge 1) \\
m_{12} = 2 \\
m_{1j} = 0 \textrm{ } (j \ge 3) \\
m_{ij} = 0 \textrm{ } (i \ge 2, j \ge 2). \\
\end{array}
\right.
\end{equation}

\noindent Note that there are no such $(m_{ij})$ as in (\ref{nonpart3}) for $r<3$.
If we take such an $(m_{ij}) \in \mathcal{M}(r)$ and form a new
$M'=(m_{ij}')$ with $m_{ij}' = m_{ij}$ for all $(i,j) \ne (1,3)$ and
$m_{13}' = 0$ then $M' \in \mathcal{M}_{part}(r-3)$.
So we think
of an $(m_{ij})$ as in (\ref{nonpart3}) as a partition of $r-3$ with
$m_{13}=1$ appended to it.

Also note that there are no such $(m_{ij})$ as in (\ref{nonpart4}) for $r<4$.
If we take such an $(m_{ij}) \in \mathcal{M}(r)$ and form a new
$M'=(m_{ij}')$ with $m_{ij}' = m_{ij}$ for all $(i,j) \ne (1,2)$ and
$m_{12}' = 0$ then $M' \in \mathcal{M}_{part}(r-4)$.
So we think
of an $(m_{ij})$ as in (\ref{nonpart4}) as a partition of $r-4$ with
$m_{12}=2$ appended to it

For $(m_{ij})$ as in (\ref{nonpart3}), we have $m_1 = m_{11} + 1$
and $m_i = m_{i1}$ for $i \ne 1$, and by Lemma \ref{phiigen} we
produce $\phi^+_{i,M}(1)$ as

$$ \phi^+_{1,M}(1) = \frac{q^{-2}}{m_{11}!} (1 + O(q^{-1}))$$

\noindent
and for $i\ge 2$

$$\phi^+_{i,M}(1) = \frac{1}{i^{m_{i1}} m_{i1}!} (1 + O(q^{-1})).$$

\noindent
There is only one way to make up a $q^{-2}$ term and the second line
in Table \ref{q^-2nonparttable} corresponds to forming a
$q^{-2}$ term in this way.

For $(m_{ij})$ as in $($\ref{nonpart4}$)$, we have $m_1= m_{11} + 2$
and $m_i = m_{i1}$ for $i \ne 1$, and by Lemma \ref{phiigen} we
obtain the $\phi^+_{i,M}(1)$ as follows:

$$ \phi^+_{1,M}(1) = \frac{q^{-2}}{m_{11}! 2!} (1 + O(q^{-1}))$$

\noindent
and for $i\ge 2$

$$\phi^+_{i,M}(1) = \frac{1}{i^{m_{i1}} m_{i1}!} (1 + O(q^{-1})).$$

\noindent
There is again only one way to make a $q^{-2}$ term here and the third
line in Table \ref{q^-2nonparttable} corresponds to forming a
$q^{-2}$ term in this way.
Note that $\prod_j m_{1j}! = m_{1j}!2!$, so we can take the half out as a
factor.

Finally $i_0 = 1$, $j_0 = 2$ and $m_{12}=1$ with $m_{1j}=0$ for $j\ge
3$ provides us with our last case. Here we have

\begin{equation}\label{nonpart2}
(m_{ij}) = \left\{
\begin{array}{l}
m_{i1} = \textrm{anything } (i \ge 1) \\
m_{12} = 1 \\
m_{1j} = 0 \textrm{ } (j \ge 3) \\
m_{ij} = 0 \textrm{ } (i \ge 2, j \ge 2) \\
\end{array}
\right.
\end{equation}

\noindent Note that there are no such $(m_{ij})$ for $r<2$.
If we take such an $(m_{ij}) \in \mathcal{M}(r)$ and form a new
$M'=(m_{ij}')$ with $m_{ij}' = m_{ij}$ for all $(i,j) \ne (1,2)$ and
$m_{12}' = 0$ then $M' \in \mathcal{M}_{part}(r-2)$. So we think
of an $(m_{ij})$ as in \ref{nonpart2} as a partition of $r-2$ with
$m_{12}=1$ appended to it.

For $(m_{ij})$ as in $($\ref{nonpart2}$)$, we have $m_1 = m_{11} + 1$
and $m_i = m_{i1}$ for $i \ne 1$, and by Lemma \ref{phiigen} we
obtain the $\phi^+_{i,M}(1)$ as follows

$$\phi^+_{1,M}(1) = \frac{q^{-1}}{ m_{11}!} \left(1 -
\left(\frac{m_{11}^2}{2} + \frac{m_{11}}{2}\right) q^{-1} +
O(q^{-2}) \right)$$

$$\phi^+_{2,M}(1) = \frac{1}{2^{m_{21}} m_{21}!} \left(1 -m_{21}q^{-1} +
O(q^{-2}) \right)$$

\noindent
and for $i\ge 3$

$$\phi^+_{i,M}(1) = \frac{1}{i^{m_{i1}} m_{i1}!} \left(1 + O(q^{-1}) \right).$$

We have now that $\phi^+_{1,M}(1)$ has leading term involving $q^{-1}$
so we can choose $a_1 = -1$ and $a_2 = -1$, or $a_1 = -2$ and $a_2 =
0$. The former corresponds to line four of Table
\ref{q^-2nonparttable} while the latter corresponds to lines five and six
of that table.

The first column
of Table \ref{q^-2nonparttable}
indicates the type of array $M$ used to produce the
$\phi^+_{i,M}(1)$. If a line contains
exactly one nonconstant entry then that entry is a certain function of
$m_{k1}$ for some $k$, say $f_k(m_{k1})$ and this line corresponds to
a summand $a_r$ of the coefficient of $q^{-2}$, where $a_r$ is as given in Lemma
\ref{1/1-t}. The sum $1 + \sum_{r=1}^\infty a_r t^r$ corresponding to
these values of $a_r$ is evaluated using Lemmas \ref{1/1-t} and Lemma \ref{te}
and recorded in the last entry for this line. There are lines in which
all the entries are constant and we produce the generating functions by
Corollary \ref{sumparts=1} in these cases.
Each generating function has been multiplied by $t^b$ for the positive
integer $b$ such that the $(m_{ij})$'s in question correspond to
partitions of $r-b$.

\begin{table}
\begin{center}
\begin{tabular}{c | c c c | c}
$M$ as in: & $\phi^+_{1,M}(1)$ & $\phi^+_{2,M}(1)$ & $\phi^+_{M,3+}(1)$ & Generating Function \\
\hline
(\ref{nonpart1}) & $1$ & $\frac{1}{2}$ & $1$ &
$\left(\frac{t^4}{2}\right) \frac{1}{1-t}$ \\
(\ref{nonpart3}) & $1$ & $1$ & $1$ & $\left({t^3}\right) \frac{1}{1-t}$ \\
(\ref{nonpart4}) & $\frac{1}{2}$ & $1$ & $1$ &
$\left(\frac{t^4}{2}\right)\frac{1}{1-t}$ \\
(\ref{nonpart2}) & $1$ & $-m_{21}$ & $1$ &
$\left(\frac{-t^4}{2}\right)\frac{1}{1-t}$ \\
(\ref{nonpart2}) & $-\frac{m_{11}^2}{2}$ & $1$ & $1$ &
$\left(-\frac{t^3}{2} - \frac{t^4}{2}\right) \frac{1}{1-t}$  \\

(\ref{nonpart2}) & $-\frac{m_{11}}{2}$ & $1$ & $1$ &
$\left(-\frac{t^3}{2}\right) \frac{1}{1-t}$  \\

\hline
& & & Total & $0$ \\
\hline
\end{tabular}
\end{center}
\caption{Producing the generating function for the coefficient
of $q^{-2}$ in the expansion of $\lim_{n\to\infty} c_{\GL,r}(n)$
due to $M \in \mathcal{M}(r)$ corresponding to nonpartitions.}
\label{q^-2nonparttable}
\end{table}

Adding all the generating functions together for those
$M \notin \mathcal{M}_{part}(r)$ gives $0$. So the generating function
for the coefficient of $q^{-2}$ is just what is produced by the
$M \in \mathcal{M}_{part}(r)$ which is $\frac{-t}{1-t}$. In the
expansion of this, the coefficient of $t^r$ is $-1$ for all $r \ge
1$. Hence for all $r$, the coefficient of $q^{-2}$ is $-1$. This
completes the proof that for all $r \in \mathbb{Z}^+$, the
$\lim_{n\rightarrow\infty}c_{\GL,r}(n) = 1 - q^{-2} + O(q^{-3})$.
\end{proof}

Theorem \ref{1-q^-2ch1} for general linear groups now follows from
Theorem \ref{limcrn} and Theorem \ref{1-q^-2}.

It is a basic assumption of Section \ref{chap4} that $r$, the dimension
of the invariant subspace, is greater than or equal to $1$.
However, we could have relaxed that assumption, allowed $r$ to equal zero
and worked through the section in the same manner.
If we had allowed $r$ to equal zero then
Theorem \ref{C_r=Cblah} would state that $C_{\GL,r}(t) = C_{\GL}(t)$
when $r=0$, since the only matrix in $\mathcal{M}(0)$ is the matrix
containing only zeroes.
In the proof of Theorem \ref{1-q^-2}, the
generating function for the coefficient of
$q^{-2}$ in $\lim_{n\to\infty} c_{\GL,r}(n)$ is $\frac{-t}{1-t}$.
So for $r = 0$, the resulting answer for the
limit of $c_{\GL,0}(n)$ would be $1 - q^{-3} + O(q^{-4})$, since
the coefficient of $t^0$ in $\frac{-t}{1-t}$ (which gives the
coefficient of $q^{-2}$ when $r=0$) is $0$.

These are exactly the answers that should arise, because when $r=0$
the invariant subspace has dimension $0$ and hence our matrix group
$\GL(V)_U$ is equal to $\GL(V)$. By Theorem \ref{limpropGL(V)}
the limiting proportion of cyclic matrices in $\GL(V)$ is indeed
$1 - q^{-3} + O(q^{-4})$.

\subsection {Inside the Matrix Algebras} \label{M(V)_U}

We move to looking for cyclic matrices inside the matrix algebra
of all matrices which
fix a subspace $U$ of the vector space $V$.
We now consider the subalgebra $\M(V)_U$, that is, we
include matrices whose characteristic polynomial has
zero constant term, into our calculations.

Let $\Gamma_{\M,r}(n)$ be the set of all cyclic matrices in $\M(V)_U$ where $n$,
the dimension of the vector space $V$, will vary and $r$, the dimension
of the subspace $U$, will remain fixed. Then by
Lemma \ref{m,m=conjclasses} there is a one-to-one correspondence between the
set of orbits of $\GL(V)_U$ in its action on $\Gamma_{\M,r}(n)$ by conjugation and the set of
pairs of
monic polynomials $(f,h)$ over $\mathbb{F}_q$ where $f$ is of degree $r$,
$h$ is of degree $n$ and $f$ divides $h$.
We denote by $\Gamma_{f,h}$ the orbit of $\Gamma_\M(n)$
containing those matrices with minimal polynomial $f$ on $U$ and minimal
polynomial $h$ on $V$. Note that when $h$ has nonzero constant term, this $\Gamma_{f,h}$ is the same
as the set $\Gamma_{f,h}$ defined in Section \ref{GL(V)_U}. Thus, by Equation (\ref{gamma_fh}) we have

$$|\Gamma_{f,h}| = \frac{|\GL(V)_U|}{Cent(h)}.$$

\noindent
We digress briefly to define $\omega_r(n)$ and prove a result about it.

\begin{definition} \label{omegar}
{\rm Let $\omega_r(n) = \frac{|\GL(V)_U|}{|\M(V)_U|}$ where $V$ has dimension $n$
and $U$ has dimension $r$.}
\end{definition}

\begin{lemma} \label{omegarproof}
For some $(n-r)$-dimensional subspace $W$,

$$ \omega_r(n) = \frac{|\GL(V)_U|}{|\M(V)_U|} =
\frac{|\GL(V)_{U\oplus W}|}{|\M(V)_{U \oplus W}|}
= \displaystyle \prod_{i=1}^r(1-q^{-i}) \prod_{i=1}^{n-r}(1-q^{-i}). $$

\noindent Also
\begin{eqnarray*}
\lim_{n\to\infty}
\omega_r(n) & =
\left\{
\begin{array}{l @{+} l l}
1 - 2q^{-1}          & \hspace{0.2cm} q^{-3} \hspace{1.13cm}+ O(q^{-5}), & \textrm{if } r=1 \\
1 - 2q^{-1} - q^{-2} & 3q^{-3} \hspace{1.15cm} + O(q^{-5}), & \textrm{if } r= 2 \\
1 - 2q^{-1} - q^{-2} & 2q^{-3} + 2q^{-4}       + O(q^{-5}), & \textrm{if } r= 3 \\
1 - 2q^{-1} - q^{-2} & 2q^{-3} + \hspace{0.18cm} q^{-4}  + O(q^{-5}), & \textrm{if } r\ge 4 \\
\end{array}
\right. \\
\end{eqnarray*}
\end{lemma}

\begin{proof}
This is proved, for example, in \cite[Lemma 4.3.2]{myPhD}.
\end{proof}

Recall that $\mathcal{P}$ is the set of all monic polynomials over
$\mathbb{F}_q$ and that $\mathcal{P}_i$ is the subset of $\mathcal{P}$
that contains the constant polynomial $1$ and those polynomials whose
irreducible factors all have degree $i$.
Also recall that $\alpha(h;r)$
denotes the number of distinct degree $r$ factors of the polynomial $h$.

Now denote by $c_{\M,r}(n)$ the proportion of cyclic matrices in $\M(V)_U$.
Then it follows that

\begin{equation} \label{cMrn}
\begin{array} {rcl}
c_{\M,r}(n) & = &
\displaystyle \sum_{
\begin{array}{c}
\vspace{-0.2cm}
\textrm{\scriptsize $h$$\in$$\mathcal{P}$} \\
\vspace{-0.2cm}
\textrm{\scriptsize \deg$(h)$$=$$n$}\\
\vspace{-0.2cm}
\textrm{\scriptsize \deg$(f)$$=$$r$}\\
\vspace{-0.2cm}
\textrm{\scriptsize $f|h$} \\
\end{array}
}
\frac{|\Gamma_{f,h}|}{|\M(V)_U|}
\\ \vspace{0.3cm}
& = &
\displaystyle \sum_{
\begin{array}{c}
\vspace{-0.2cm}
\textrm{\scriptsize $h$$\in$$\mathcal{P}$} \\
\vspace{-0.2cm}
\textrm{\scriptsize \deg$(h)$$=$$n$}\\
\end{array}
}
\frac{\alpha(h;r)|\Gamma_{f,h}||\GL(V)_U|}{|\M(V)_U||\GL(V)_U|}
\\ \vspace{0.3cm}
& = &
\displaystyle \sum_{
\begin{array}{c}
\vspace{-0.2cm}
\textrm{\scriptsize $h$$\in$$\mathcal{P}$} \\
\vspace{-0.2cm}
\textrm{\scriptsize \deg$(h)$$=$$n$}\\
\end{array}
}
\frac{\alpha(h;r) \omega_r(n)}{Cent(h)}.
\end{array}
\end{equation}

\noindent
We have that

$$ \frac{c_{\M,r}(n)}{\omega_r(n)} =
\sum_{
\begin{array}{c}
\vspace{-0.2cm}
\textrm{\scriptsize $h$$\in$$\mathcal{P}$} \\
\vspace{-0.2cm}
\textrm{\scriptsize \deg$(h)$$=$$n$}\\
\end{array}
}
\frac{\alpha(h;r)}{Cent(h)} $$

\noindent
which we will calculate by a similar method to that used for calculating $c_{\GL,r}(n)$ in
Section \ref{GL(V)_U}. Let

$$ C_{\M,r}(t) = \sum_{n=r}^\infty
\left(\frac{c_{\M,r}(n)}{\omega_r(n)}\right) t^n$$

\noindent
be the 'weighted' generating function for the proportion of cyclic matrices in $\M(V)_U$.
Note that the coefficients of $C_{\M,r}(t)$ are not the proportions we desire, as they are
'weighted' by $\omega_r(n)$.
Note also that the sum starts from $n=r$ because $r \le$ dim$(V)$.

Recall that Equation \ref{Phi_p} gave us a formal power series for each
irreducible polynomial and that we collected all irreducible polynomials of
degree $i$ to form

$$ \Phi^+_i = \left( 1 + \sum_{j=1}^\infty \frac{s_{ij}t^{ij}}{Cent(i,j)}
\right)^{N^+(i,q)}. $$

\noindent
We will slightly modify these $\Phi^+_i$ by making just one small change and
we will call the new power series $\Phi_i$.
Because we want to include the irreducible polynomial $t$ in our calculations
now, we make the exponent equal to $N(i,q)$ instead of $N^+(i,q)$.
Since $N(i,q) = N^+(i,q)$ for $i \ge 2$ it follows that
$\Phi_i = \Phi^+_i$ for $i\ge 2$. So we define

$$ \Phi_i = \left( 1 + \sum_{j=1}^\infty \frac{s_{ij}t^{ij}}{Cent(i,j)}
\right)^{N(i,q)}. $$

Now we give the first of several corollaries. The first follows immediately
from Lemma \ref{Phipower}, the only difference being that the summands now include
polynomials with zero constant term.

\begin{corollary}\label{Phihatpower}
As a power series, $\Phi_i$ satisfies

$(1).$ $\displaystyle \Phi_i = \sum_{h_i \in \mathcal{P}_i}
\frac{\left(\prod_j s_{ij}^{\tau(h_i;i,j)}\right) t^{\deg(h_i)}}
{Cent(h_i)},$

$(2).$ $\displaystyle \prod_{i=1}^\infty \Phi_i =
\sum_{h \in \mathcal{P}}
\frac{\left( \prod_j s_{ij}^{\tau(h;i,j)} \right) t^{\deg(h)}}
{Cent(h)}.$
\end{corollary}

Recall that in Section \ref{GL(V)} we defined $\alpha(h;r)$ by

$$ \alpha(h;r) = \sum_{M \in \mathcal{M}(r)}
\prod_{i=1}^r \alpha(h_i; r_i, M_i) $$

\noindent
where $h_i$ was the product of all the irreducible degree $i$
factors of $h$, we assigned $r_i = i \sum_j jm_{ij}$, $M_i = M_i(M)$ is as in (\ref{Mi})
and the overall term
$\alpha(h_i; r_i, M_i)$ referred to the number of degree $r_i$ factors of
$h_i$ that corresponded with the array $M$.

Similarly to Definition \ref{Phialphadef} we define
$\Phi_{i,M,\alpha}$, for each $i$, by

\begin{equation} \label{Phialphanorm}
\Phi_{i,M,\alpha} = \Phi_{i,M,\alpha}(t) :=
\sum_{h_i \in \mathcal{P}_i} \alpha(h_i; r_i, M_i)
\frac{t^{\deg(h_i)}}{Cent(h_i)}.
\end{equation}

The next corollary, which follows immediately from Lemma \ref{Crfinal}, tells us
how to obtain $C_{\M,r}(t)$ using the $\Phi_{i,M,\alpha}$.

\begin{corollary} \label{CMrfinal}
$$C_{\M,r}(t) = \sum_{M \in \mathcal{M}(r)} \prod_{i=1}^\infty
\Phi_{i,M,\alpha} (t)$$
\end{corollary}

When given $\Phi_i$ as input, the procedure {\sc PhiAlpha}
from Section \ref{GL(V)_U} produces $\Phi_{i,M,\alpha} (t)$. So we have
the following theorem which is similar to Theorem \ref{C_r=Cblah}.
The proof is the same as that of Theorem \ref{C_r=Cblah} with $N(i,q)$
substituted for $N^+(i,q)$.

\begin{theorem}\label{CMr=CMblah}
Let $\phi_{i,M}(t) = m_i! {N(i,q) \choose m_i} \prod_j
\frac{(tq^{-1})^{ijm_{ij}}}{m_{ij}!(1-q^{-i}+t^iq^{-2i})^{m_{ij}}}$. Then

$$ C_{\M,r}(t) = C_\M(t) \sum_{M \in \mathcal{M}(r)}
\prod_{i=1}^r \phi_{i,M}(t) $$

\noindent
Moreover, $\phi_{i,M}(t) = 1$ for all $M \in \mathcal{M}(r)$ when $i>r$.
\end{theorem}

Theorem \ref{CMr=CMblah} gives us a direct formula for calculating the generating
function $C_{\M,r}(t)$ for any $r$. Now we want to know the limit of the
coefficients as $n$ tends to infinity. We first need to know the convergence
properties of $C_{\M,r}(t)$. The following corollary follows immediately from Lemma \ref{phiconv}.

\begin{corollary}
Let $M \in \mathcal{M}(r)$. Then for any $i \le r$ the power series
expansion of

$$ \phi_{i,M}(t) =  m_i! {N(i,q) \choose m_i} \prod_j
\frac{(tq^{-1})^{ijm_{ij}}}{m_{ij}!(1-q^{-i}+t^iq^{-2i})^{m_{ij}}} $$

\noindent
is convergent for $|t| < q(q^i-1)^{1/i}$.
\end{corollary}

This again follows since the only difference between $\phi_{i,M}$ and
$\phi^+_{i,M}$ is when $i=1$ where we have $N(1,q) = N^+(1,q)+1$.
Similarly we have the next corollary.

\begin{corollary} \label{q(q-1)b}
The power series expansion of $\sum_{M \in \mathcal{M}(r)} \prod_{i=1}^r
\phi_{i,M}(t)$ is convergent for $|t| < q(q-1)$.
\end{corollary}

Now that the convergence properties are established we produce the first main
result of the section - a theorem giving a formula for the limit of the
coefficients of $t^n$ in $C_{\M,r}(t)$ as $n$ tends to infinity.

\begin{theorem} \label{limcrMn}
For $r \in \mathbb{Z}^+$

$$ \lim_{n \to \infty} \left( \frac{c_{\M,r}(n)}{\omega_r(n)} \right) =
\frac{1-q^{-5}}{(1-q^{-1})(1-q^{-2})} \sum_{M \in \mathcal{M}(r)}
\prod_{i=1}^r \phi_{i,M}(1).$$

\noindent
Moreover, $|\frac{c_{\M,r}(n)}{\omega_r(n)} - \lim_{n\to\infty} \frac{c_{\M,r}(n)}{\omega_r(n)}| = O(d^{-n})$
for any $d$ such that $1 < d < q(q-1)$.
\end{theorem}

\begin{proof}
By \cite{wall}, $(1-t)C_\M(t)$ is convergent for $|t|<q^2$, and by
Corollary \ref{q(q-1)b}, $\sum_{M \in \mathcal{M}(r)}
\prod_{i=1}^r \phi_{i,M}(t)$ is convergent for $|t| < q(q-1)$. Hence
$C_{\M,r}(t)$ is convergent for $|t|< q(q-1)$.

By Lemma \ref{1-t}, $(1-t) C_{\M,r}(t)$ evaluated at $t=1$ gives the limit
of the coefficients of $C_{\M,r}(t)$, that is
$\lim_{n \to\infty} \left( \frac{c_{\M,r}(n)}{\omega_r(n)} \right)$. But

$$C_{\M,r}(t) = C_\M(t) \sum_{M \in \mathcal{M}(r)}
\prod_{i=1}^r \phi_{i,M}(t)$$

\noindent
by Theorem \ref{CMr=CMblah} and we know
already that $(1-t) C_\M(t)$ evaluated at $t=1$ is
$\frac{1-q^{-5}}{(1-q^{-1})(1-q^{-2})}$ by Theorem \ref{limpropM(V)}.
Hence

$$ \lim_{n\to\infty} \left( \frac{c_{\M,r}(n)}{\omega_r(n)} \right) =
\frac{1-q^{-5}}{(1-q^{-1})(1-q^{-2})} \sum_{M \in \mathcal{M}(r)}
\prod_{i=1}^r \phi_{i,M}(1). $$

We showed above that $(1-t)C_{\M,r}(t)$ is convergent for $|t| < q(q-1)$,
so the final assertion about the rate of convergence follows by Lemma \ref{1-t}.
\end{proof}

We now look at the power series expansion of $\phi_{i,M}(1)$
for all $i$. This will help us obtain a power series expansion of
$\lim_{n\to\infty} c_{\M,r}(t)$.

\begin{lemma} \label{phiihatgen}
Let $M \in \mathcal{M}(r)$ and set $m_i = \sum_j m_{ij}$. Then with the
notation of Theorem \ref{CMr=CMblah}, we have
\begin{eqnarray*}
\phi_{1,M}(1) & = & \frac{q^{m_1-\sum jm_{1j}}}{\prod_j m_{1j}!}
\left( 1 + \left(\frac{-m_1^2}{2} + \frac{3m_1}{2} \right)q^{-1} \right. \\
& & +
\left. \left( \frac{m_1^4}{8} - \frac{11m_1^3}{12} + \frac{11m_1^2}{8}
- \frac{7m_1}{12} \right)q^{-2} + O(q^{-3}) \right), \\
\phi_{2,M}(1) & = & \frac{q^{2m_2-2\sum jm_{2j}}}{2^{m_2} \prod_j m_{2j}!}
\left(1 - m_2q^{-1} + \left(-\frac{m_2^2}{2}+\frac{3m_2}{2}\right)
q^{-2} + O(q^{-3})\right), \\
\phi_{3,M}(1) & = & \frac{q^{3m_3-3\sum jm_{3j}}}{3^{m_3} \prod_j m_{3j}!}
\left(1 - m_3q^{-2} + O(q^{-3})\right),\\
\phi_{4,M}(1) & = & \frac{q^{4m_4-4\sum jm_{4j}}}{4^{m_4} \prod_j m_{4j}!}
\left(1 - m_4q^{-2} + O(q^{-3})\right), \\
\phi_{i,M}(1) & = & \frac{q^{im_i-i\sum jm_{ij}}}{i^{m_i} \prod_j m_{ij}!}
\left(1 + O(q^{-3})\right) \quad \textrm{for } i \ge 5. \\
\end{eqnarray*}
\end{lemma}

\begin{proof}
Since $N(i,q) = N^+(i,q)$ for $i\ge 2$ it follows that
$\phi_{i,M}(1) = \phi^+_{i,M}(1)$ for $i\ge 2$ so the expansions are as in
Lemma \ref{phiigen}. Thus we only need to consider $i=1$ for this proof.

From the definition of $\phi_{i,M}(t)$ in Theorem \ref{CMr=CMblah},

$$\displaystyle \phi_{1,M}(1) = m_1! {N(1,q) \choose m_1} \prod_j
\frac{q^{-jm_{1j}}}
{m_{1j}!(1-q^{-1}+q^{-2})^{m_{1j}}}
= \frac{q}{q-m_1} \times \phi^+_{1,M}(1). $$

Firstly

$$ \frac{q}{q-m_1} = \frac{1}{1-m_1q^{-1}} = 1 + m_1 q^{-1} + m_1^2 q^{-2} + O(q^{-3}). $$

\noindent
Then from Lemma \ref{phiigen} we have that

\begin{displaymath}
\begin{array} {rcl}
\displaystyle\phi^+_{1,M}(1) & = & \displaystyle \frac{q^{m_1-\sum jm_{1j}}}{\prod_j m_{1j}!} \left( 1 +
\left( -\frac{m_{1}^2}{2} + \frac{m_{1}}{2} \right) q^{-1} \right. \\
& & \left. + \displaystyle \left(\frac{m_{1}^4}{8} -
\frac{5m_{1}^3}{12} - \frac{m_{1}^2}{8} -
\frac{7m_{1}}{12}\right)q^{-2} + O(q^{-3}) \right). \\
\end{array}
\end{displaymath}

\noindent Multiplying together these two parts and collecting
terms gives us $\phi_{1,M}(1)$ equal to

$$\frac{q^{m_1-\sum jm_{1j}}}{\prod_j m_{1j}!} \times$$
$$\left(1 + \left( - \frac{m_1^2}{2} + \frac{3m_{1}}{2} \right) q^{-1}
+ \left(\frac{m_{1}^4}{8} -
\frac{11m_{1}^3}{12} + \frac{11m_{1}^2}{8} -
\frac{7m_{1}}{12}\right)q^{-2} + O(q^{-3})\right). $$
\end{proof}

\begin{corollary} \label{phiihatpart}
For $M \in \mathcal{M}_{part}(r)$ we have
\begin{eqnarray*}
\phi_{1,M}(1) & = & \frac{1}{m_{11}!}
\left( 1 + \left(\frac{-m_{11}^2}{2} + \frac{3m_{11}}{2} \right)q^{-1} + \right.\\
& & \left. \left( \frac{m_{11}^4}{8} - \frac{11m_{11}^3}{12} + \frac{11m_{11}^2}{8}
- \frac{7m_{11}}{12} \right)q^{-2} + O(q^{-3}) \right), \\
\phi_{2,M}(1) & = & \frac{1}{2^{m_{21}} m_{21}!}
\left(1 - m_{21}q^{-1} + \left(-\frac{m_{21}^2}{2}+\frac{3m_{21}}{2}\right)
q^{-2} + O(q^{-3})\right), \\
\phi_{3,M}(1) & = & \frac{1}{3^{m_{31}} m_{31}!}
\left(1 - m_{31}q^{-2} + O(q^{-3})\right),\\
\phi_{4,M}(1) & = & \frac{1}{4^{m_{41}} m_{41}!}
\left(1 - m_{41}q^{-2} + O(q^{-3})\right), \\
\phi_{i,M}(1) & = & \frac{1}{i^{m_{i1}} m_{i1}!}
\left(1 + O(q^{-3})\right) \textrm{for } i \ge 5.\\
\end{eqnarray*}
\end{corollary}

\begin{proof}
For $M \in \mathcal{M}_{part}(r)$ we have $m_{ij} = 0$ for all $i$ and $j \ge 2$.
Hence $im_i - i\sum_j jm_{ij} = 0$. Thus the first term of $\phi_{i,M}(1)$
is a constant for all $i$ and the remaining terms are as in Lemma \ref{phiihatgen}
\end{proof}

The next corollary follows immediately from Corollary \ref{partiff}.

\begin{corollary} \label{phihatpartiff}
The leading term of $\phi_{i,M}(1)$ is a constant for all $i$
if and only if $(m_{ij}) \in \mathcal{M}_{part}(r)$.
\end{corollary}

Now we move onto the main theorem of the section which gives the
limiting proportion of cyclic matrices in $\M(V)_U$.
Theorem \ref{1-q^-2ch1} for matrix algebras follows from Theorem \ref{limcrMn} and this result Theorem \ref{1-q^-2b}.

\begin{theorem} \label{1-q^-2b}
For $r \in \mathbb{Z}^+$, $\displaystyle \lim_{n\to\infty} c_{\M,r}(n)
= 1 - q^{-2} + O(q^{-3}).$
\end{theorem}

\begin{proof}
By Theorem \ref{limcrMn},

\begin{equation} \label{eqcrMn}
\lim_{n\to\infty} c_{\M,r}(n) = \lim_{n\to\infty} \omega_r(n) \times
\frac{1-q^{-5}}{(1-q^{-1})(1-q^{-2})} \times \sum_{M \in \mathcal{M}(r)}
\prod_{i=1}^r \phi_{i,M}(1).
\end{equation}

We calculate the expansions of each of the three parts above. From
Lemma \ref{omegarproof}, we have that
$\lim_{n\to\infty} \omega_r(n) =
\prod_{i=1}^r (1-q^{-i}) \prod_{i=1}^\infty (1-q^{-i})$ and its
expansion is given by

\begin{displaymath}
\lim_{n\to\infty} \omega_r(n) = \left\{
\begin{array} {l l}
1 - 2q^{-1} + O(q^{-3}) & \textrm{if } r=1 \\
1 - 2q^{-1} - q^{-2} + O(q^{-3}) & \textrm{if } r\ge 2 \\
\end{array}
\right.
\end{displaymath}

For the second part, a simple expansion gives

$$ \frac{1-q^{-5}}{(1-q^{-1})(1-q^{-2})} = 1 + q^{-1} + 2q^{-2} + O(q^{-3}). $$

Calculating the expansion of the third part,
$\sum_{M \in \mathcal{M}(r)} \prod_{i=1}^r \phi_{i,M}(1)$,
requires more work. We will calculate the constant term in the expansion
followed by the coefficient of $q^{-1}$ and then the coefficient of $q^{-2}$.

\bigskip
\noindent \textbf{Constant Term:}

By Corollary \ref{phihatpartiff}, the leading term of $\phi_{i,M}(1)$
is a constant if and only if $M \in \mathcal{M}_{part}(r)$. From Corollary
\ref{phiihatpart}, we can see that this leading term is
$\frac{1}{i^{m_{i1}} m_{i1}!}$ if $M \in \mathcal{M}_{part}(r)$ so the
constant term in $\lim_{n \to \infty} c_{\GL,r}(n)$ is

$$ \sum_{M \in \mathcal{M}_{part}(r)} \prod_i \frac{1}{i^{m_{i1}}
m_{i1}!}.$$

\noindent By Corollary \ref{sumparts=1} this equals $1$ for all
$r$. Hence $1$ is the constant term for all $r$.

\bigskip
\noindent \textbf{Coefficient of $q^{-1}$}

For the coefficient of $q^{-1}$, we first sum over all
$M \in \mathcal{M}_{part}(r)$ and then look at those $M \notin
\mathcal{M}_{part}(r)$ that also contribute to the $q^{-1}$ term.

Let $M = (m_{ij}) \in \mathcal{M}_{part}(r)$. By Corollary
\ref{phiihatpart} there are exactly two ways to produce a
$q^{-1}$ term. The first is to multiply the $q^{-1}$ term in
$\phi_{1,M}(1)$ with the constant term from each of the remaining
$\phi_{i,M}(1)$. The second way is to multiply the $q^{-1}$ term in
$\phi_{2,M}(1)$ with the constant term in each of the remaining
$\phi_{i,M}(1)$.

The first way produces $\left( -\frac{m_{11}^2}{2} + \frac{3m_{11}}{2}
\right) \prod_i \frac{1}{i^{m_{i1}}m_{i1}!}$ as the coefficient of $q^{-1}$.
Summing this over all $M \in \mathcal{M}_{part}(r)$, gives the
generating function for the contribution to the coefficient
of $q^{-1}$ as

$$ -\frac{1}{2} \left( \sum_{m\ge 0} \frac{m^2 t^m}{m!} \right)
\frac{e^{-t}}{1-t}
+ \frac{3}{2} \left( \sum_{m\ge 0} \frac{m t^m}{m!} \right)
\frac{e^{-t}}{1-t} $$

\noindent by Lemma \ref{1/1-t}, using $f_1(m)=m^2$ in the first
case and $f_1(m)=m$ in the second case. From the result in
Lemma \ref{te} the generating function for the
coefficient of $q^{-1}$ arising in this way is

$$ \left( -\frac{1}{2} (t^2+t) e^t + \frac{3}{2}te^t \right)
\frac{e^{-t}}{1-t} = \left(\frac{-t^2}{2}+t \right)\frac{1}{1-t}. $$

The second way an $M \in \mathcal{M}_{part}(r)$ can produce a $q^{-1}$
term, produces $-m_{21} \prod_i \frac{1}{i^{m_{i1}}m_{i1}!}$ as the coefficient
of $q^{-1}$. Summing this over all $M \in \mathcal{M}_{part}(r)$, gives
the generating function for the contribution to the coefficient
of $q^{-1}$ as

$$ \left( - \sum_{m\ge 0} \frac{m t^{2m}}{2^m m!} \right)
\frac{e^{\frac{-t^2}{2}}}{1-t} $$

\noindent by Lemma \ref{1/1-t}, using $f_2(m)=-m$. By Lemma \ref{te}
with $b=1$ and $k=2$, it can be shown that $\sum_{m \ge 0}
\frac{m t^{2m}}{2^m m!} = \frac{t^2}{2} e^{\frac{t^2}{2}}$, so the
generating function for the coefficient of $q^{-1}$ arising in
this way is

$$ \left( \frac{-t^2}{2}e^{\frac{t^2}{2}} \right)
\frac{e^{\frac{-t^2}{2}}}{1-t} = \left(\frac{-t^2}{2}\right)
\frac{1}{1-t}.$$

Summing together these two functions gives us the generating
function for the coefficient of $q^{-1}$ due to $M \in
\mathcal{M}_{part}(r)$. This is $\frac{-t^2+t}{1-t}$.

Now we need to look at those $M = (m_{ij}) \notin
\mathcal{M}_{part}(r)$ that also contribute to the $q^{-1}$
term.
For $M$ to
contribute to the $q^{-1}$ term there must exist an $i$ such that the
leading term of $\phi_{i,M}(1)$ is not a constant and this leading
term is $q^{-1}$. Thus $i\sum_j m_{ij} - i\sum_j jm_{ij} = -1$.
Rewriting the equation we see that we need $\sum_j (1-j)m_{ij} =
\frac{-1}{i}$. Since the left hand side is an integer, this equality
can only hold for $i=1$. Thus we need
$\phi_{i,M}(1)$ for $i \ge 2$ to have constant leading term and
$\sum_j (1-j)m_{1j}=-1$. By Lemma
\ref{kindapartiff} (since $\phi_{i,M}(1) = \phi^+_{i,M}(1)$ for $i\ge 2$)
the first condition
implies that $m_{ij} = 0$
for $i\ge 2$ and $j\ge 2$. The second condition implies that
$m_{12}=1$ and $m_{1j}=0$ for $j\ge 3$. Hence if $(m_{ij}) \notin
\mathcal{M}_{part}(r)$ and contributes to the $q^{-1}$ term then

\begin{equation} \label{m12=1b}
(m_{ij}) = \left\{
\begin{array}{l}
m_{i1} = \textrm{anything } (i \ge 1) \\
m_{12} = 1 \\
m_{1j} = 0 \textrm{ } (j \ge 3) \\
m_{ij} = 0 \textrm{ } (i \ge 2, j \ge 2). \\
\end{array}
\right.
\end{equation}

\noindent Note that there are no such $(m_{ij})$ when $r=1$.
If we take such an $(m_{ij}) \in \mathcal{M}(r)$ and form a new
$M'=(m_{ij}')$ with $m_{ij}' = m_{ij}$ for all $(i,j) \ne (1,2)$ and
$m_{12}' = 0$ then $M' \in \mathcal{M}_{part}(r-2)$. So we will think
of each $(m_{ij})$ as in (\ref{m12=1b}) as a partition of $r-2$ with $m_{12}=1$
appended to it.

We can see from Lemma \ref{phiihatgen}, that for each $M=(m_{ij})$ as in
(\ref{m12=1b}), for $i \ge 2$, $\phi_{i,M}(1)$ has constant term
$\frac{1}{i^{m_{i1}} m_{i1}!}$ and for $i=1$ the coefficient of
$q^{-1}$ is $\frac{1}{m_{11}!}$. Hence each of these $(m_{ij})$
produces $\prod_i \frac{1}{i^{m_{i1}} m_{i1}!}$ as the coefficient of
$q^{-1}$. Summing this over all $M \notin \mathcal{M}_{part}(r)$ that
contribute to the $q^{-1}$ term is then the same as summing over all
$M' \in \mathcal{M}_{part}(r-2)$. By Lemma \ref{1/1-t}, the generating
function for the sum of $\prod_i \frac{1}{i^{m_{i1}} m_{i1}!}$ over
all partitions of $r$ is $\frac{1}{1-t}$. The generating function for
the sum of $\prod_i \frac{1}{i^{m_{i1}} m_{i1}!}$ over all partitions
of $r-2$ has no terms of degree less than $2$ and so is
$\frac{t^2}{1-t}$. Hence the generating function for the coefficient
of $q^{-1}$ due to $M \notin \mathcal{M}_{part}(r)$ is
$\frac{t^2}{1-t}$.

Adding the generating functions for the coefficient of $q^{-1}$ due to
$M \in \mathcal{M}_{part}(r)$ and $M \notin \mathcal{M}_{part}(r)$ gives
us $\frac{t}{1-t} = t + t^2 + t^3 + \ldots$.
Hence for all $r\ge 1$ the $q^{-1}$ term in the expansion is $1$.

\bigskip
\noindent \textbf{Coefficient of $q^{-2}$}

We take the same approach in calculating the coefficient of
$q^{-2}$. We first sum over all $M \in
\mathcal{M}_{part}(r)$. Let $M \in \mathcal{M}_{part}(r)$. From the
expansion of the $\phi_{i,M}(1)$ in Corollary \ref{phiihatpart} we can
see that there are several ways to form a $q^{-2}$ term. Each of these
ways corresponds to a line of Table \ref{q^-2parttableM}.
For each line in Table \ref{q^-2parttableM}, the contribution from $\phi_{i,M}$
from $i \ge 5$ is 1.
If a line contains
exactly one non-$1$ entry then that entry is a certain function of
$m_{k1}$ for some $k$, say $f_k(m_{k1})$ and this line corresponds to
a summand $a_r$ of the coefficient of $q^{-2}$, where $a_r$ is as given in Lemma
\ref{1/1-t}. The sum $1 + \sum_{r=1}^\infty a_r t^r$ corresponding to
these values of $a_r$ is evaluated using Lemmas \ref{1/1-t} and Lemma \ref{te}
and recorded in the last entry for this line. There are two further
lines that contain two non-$1$ entries. These are certain functions
$f_k(m_{k1})$ with $k=1$ and $f_\ell(m_{\ell 1})$ with $\ell = 2$ and
correspond to a summand $a_r$ of the coefficient of
$q^{-2}$, where $a_r$ is as given in Lemma \ref{kl1/1-t}. The sum $1 +
\sum_{r=1}^\infty a_r t^r$ corresponding to these values of $a_r$ is
evaluated using Lemmas \ref{kl1/1-t} and Lemma \ref{te} and recorded in the
last entry for this line.

\begin{table}
\begin{center}
\begin{tabular}{c c c c | c}
$\phi_{1,M}(1)$ & $\phi_{2,M}(1)$ & $\phi_{3,M}(1)$
& $\phi_{4,M}(1)$ & Generating Function \\
\hline
$1$ & $1$ & $1$ & $-m_{41}q^{-2}$ &
$\left(-\frac{t^4}{4}\right)\frac{1}{1-t}$ \\
$1$ & $1$ & $-m_{31}q^{-2}$ & $1$ &
$\left(-\frac{t^3}{3}\right)\frac{1}{1-t}$ \\
$1$ & $\frac{3m_{21}}{2}q^{-2}$ & $1$ & $1$ &
$\left(\frac{3t^2}{4}\right)\frac{1}{1-t}$ \\
$1$ & $-\frac{m_{21}^2}{2}q^{-2}$ & $1$ & $1$ &
$\left(-\frac{t^4}{8}-\frac{t^2}{4}\right)\frac{1}{1-t}$ \\

$-\frac{m_{11}^2}{2}q^{-1}$ & $-m_{21}q^{-1}$ & $1$ &
$1$ &  $\left(\frac{t^4}{4} + \frac{t^3}{4} \right) \frac{1}{1-t}$ \\

$\frac{3m_{11}}{2}q^{-1}$ & $-m_{21}q^{-1}$ & $1$ &
$1$ & $\left(- \frac{3t^3}{4} \right) \frac{1}{1-t}$ \\

$-\frac{7m_{11}}{12}q^{-2}$ & $1$ & $1$ & $1$ &
$\left(-\frac{7t}{12}\right)\frac{1}{1-t}$ \\
$\frac{11m_{11}^2}{8}q^{-2}$ & $1$ & $1$ & $1$ &
$\left(\frac{11(t^2+t)}{8}\right)\frac{1}{1-t}$ \\
$-\frac{11m_{11}^3}{12}q^{-2}$ & $1$ & $1$ & $1$ &
$\left(\frac{-11(t^3+3t^2+t)}{12}\right)\frac{1}{1-t}$ \\
$\frac{m_{11}^4}{8}q^{-2}$ & $1$ & $1$ & $1$ &
$\left(\frac{t^4+6t^3+7t^2+t}{8}\right)\frac{1}{1-t}$ \\
\hline
& & & Total & $\left( -t^3 \right)\frac{1}{1-t}$ \\
\hline
\end{tabular}
\end{center}
\caption{Producing the generating function for the coefficient
of $q^{-2}$ in the expansion of $\sum_M \prod_i \phi_{i,M}(1)$
due to $M \in \mathcal{M}_{part}(r)$.}
\label{q^-2parttableM}
\end{table}

\noindent
The sum of all the generating functions is the generating function
for the coefficient of $q^{-2}$ due to $M \in
\mathcal{M}_{part}(r)$. It is $\frac{-t^3}{1-t}$.

We now need to look at the $M \notin \mathcal{M}_{part}(r)$ that can
produce a $q^{-2}$ term. For any $M \notin \mathcal{M}_{part}(r)$ the
leading term of $\phi^+_{i,M}(1)$ is the same as $\varphi^+_{i,M}(1)$ for
all $i$, so by the same reasoning as in the proof of
Theorem \ref{1-q^-2}, there are only four types of $M \notin
\mathcal{M}_{part}(r)$ that can produce a $q^{-2}$ term. They are
given below.

\begin{equation}\label{nonpart12}
(m_{ij}) = \left\{
\begin{array}{l}
m_{i1} = \textrm{anything } (i \ge 0) \\
m_{22} = 1 \\
m_{2j} = 0 \textrm{ } (j \ge 3) \\
m_{ij} = 0 \textrm{ } (i \ne 2, j \ge 2) \\
\end{array}
\right.
\end{equation}

\begin{equation}\label{nonpart32}
(m_{ij}) = \left\{
\begin{array}{l}
m_{i1} = \textrm{anything } (i \ge 0) \\
m_{13} = 1 \\
m_{1j} = 0 \textrm{ } (j \ne 1,3) \\
m_{ij} = 0 \textrm{ } (i \ge 2, j \ge 2) \\
\end{array}
\right.
\end{equation}

\begin{equation}\label{nonpart42}
(m_{ij}) = \left\{
\begin{array}{l}
m_{i1} = \textrm{anything } (i \ge 0) \\
m_{12} = 2 \\
m_{1j} = 0 \textrm{ } (j \ge 3) \\
m_{ij} = 0 \textrm{ } (i \ge 2, j \ge 2) \\
\end{array}
\right.
\end{equation}

\noindent and

\begin{equation}\label{nonpart22}
(m_{ij}) = \left\{
\begin{array}{l}
m_{i1} = \textrm{anything } (i \ge 0) \\
m_{12} = 1 \\
m_{1j} = 0 \textrm{ } (j \ge 3) \\
m_{ij} = 0 \textrm{ } (i \ge 2, j \ge 2). \\
\end{array}
\right.
\end{equation}

\noindent Note that there are no such $(m_{ij})$ for $r<4$.
If we take such an $(m_{ij}) \in \mathcal{M}(r)$ and form a new
$M'=(m_{ij}')$ with $m_{ij}' = m_{ij}$ for all $(i,j) \ne (2,2)$ and
$m_{22}' = 0$ then $M' \in \mathcal{M}_{part}(r-4)$. Thus we think
of each $(m_{ij})$ as in (\ref{nonpart12}) as a partition of $r-4$ with
$m_{22}=1$ appended to it.

For $(m_{ij})$ as in (\ref{nonpart12}), we have $m_2 = m_{21}+1$
and $m_i = m_{i1}$ for $i\ne 2$. By Lemma \ref{phiigen} we produce
$\phi_{i,M}(1)$ as follows.

$$\phi_{1,M}(1) = \frac{1}{ m_{11}!} (1 + O(q^{-1}))$$

$$\phi_{2,M}(1) = \frac{q^{-2}}{2^{m_{21}} m_{21}!}
(\frac{1}{2}) (1 + O(q^{-1}))$$

\noindent
and for $i\ge 3$

$$\phi_{i,M}(1) = \frac{1}{i^{m_{i1}} m_{i1}!} (1 + O(q^{-1})).$$

\noindent
The `$\frac{1}{2}$' inside the $\phi_{2,M}(1)$ occurs because $2^{m_2} =
2^{m_{21}}2^1$.

There is only one way to make up a $q^{-2}$ term and the first line in
Table \ref{q^-2nonparttableM} corresponds to forming a $q^{-2}$
term in this way.

\noindent Note that there are no such $(m_{ij})$ as in
(\ref{nonpart32}) for $r<3$.
If we take such an $(m_{ij}) \in \mathcal{M}(r)$ and form a new
$M'=(m_{ij}')$ with $m_{ij}' = m_{ij}$ for all $(i,j) \ne (1,3)$ and
$m_{13}' = 0$ then $M' \in \mathcal{M}_{part}(r-3)$. So we   think
of each $(m_{ij})$ as in (\ref{nonpart12}) as a partition of $r-3$ with
$m_{13}=1$ appended to it.

Also note that there are no such $(m_{ij})$ as in (\ref{nonpart42}) for $r<4$.
If we take such an $(m_{ij}) \in \mathcal{M}(r)$ and form a new
$M'=(m_{ij}')$ with $m_{ij}' = m_{ij}$ for all $(i,j) \ne (1,2)$ and
$m_{12}' = 0$ then $M' \in \mathcal{M}_{part}(r-4)$. So we   think
of each $(m_{ij})$ as in (\ref{nonpart42}) as a partition of $r-4$ with
$m_{12}=2$ appended to it.

For $(m_{ij})$ as in (\ref{nonpart32}), we have $m_1 = m_{11} + 1$
and $m_i = m_{i1}$ for $i \ne 1$, and by Lemma \ref{phiigen} we
produce $\phi_{i,M}(1)$ as follows.

$$ \phi_{1,M}(1) = \frac{q^{-2}}{m_{11}!} (1 + O(q^{-1}))$$

\noindent
and for $i\ge 2$

$$\phi_{i,M}(1) = \frac{1}{i^{m_{i1}} m_{i1}!} (1 + O(q^{-1})).$$

\noindent
There is only one way to make up a $q^{-2}$ term and the second line
in Table \ref{q^-2nonparttableM} corresponds to forming a
$q^{-2}$ term in this way.

For $(m_{ij})$ as in (\ref{nonpart42}), we have $m_1= m_{11} + 2$
and $m_i = m_{i1}$ for $i \ne 1$, and by Lemma \ref{phiigen} we
produce $\phi_{i,M}(1)$ as follows.

$$ \phi_{1,M}(1) = \frac{q^{-2}}{m_{11}! 2!} (1 + O(q^{-1}))$$

\noindent
and for $i\ge 2$

$$\phi_{i,M}(1) = \frac{1}{i^{m_{i1}} m_{i1}!} (1 + O(q^{-1})).$$

\noindent
There is again only one way to make a $q^{-2}$ term here and the third
line in Table \ref{q^-2nonparttableM} corresponds to forming a
$q^{-2}$ term in this way.
Note that $\prod_j m_{1j}! = m_{1j}!2!$, so we can take the half out as a
factor.

\noindent Note that there are no such $(m_{ij})$ for $r<2$.
If we take such an $(m_{ij}) \in \mathcal{M}(r)$ and form a new
$M'=(m_{ij}')$ with $m_{ij}' = m_{ij}$ for all $(i,j) \ne (1,2)$ and
$m_{12}' = 0$ then $M' \in \mathcal{M}_{part}(r-2)$. So we   think
of each $(m_{ij})$ as in (\ref{nonpart22}) as a partition of $r-2$ with
$m_{12}=1$ appended to it.

For $(m_{ij})$ as in (\ref{nonpart22}), we have $m_1 = m_{11} + 1$
and $m_i = m_{i1}$ for $i \ne 1$, and by Lemma \ref{phiigen} we
produce $\phi_{i,M}(1)$ as follows.

$$\phi_{1,M}(1) = \frac{q^{-1}}{ m_{11}!}
\left(1 + \left(-\frac{m_{11}^2}{2} + \frac{m_{11}}{2} +1 \right) q^{-1} +
O(q^{-2}) \right)$$

$$\phi_{2,M}(1) = \frac{1}{2^{m_{21}} m_{21}!} \left(1 -m_{21}q^{-1} +
O(q^{-2}) \right)$$

\noindent
and for $i\ge 3$

$$\phi_{i,M}(1) = \frac{1}{i^{m_{i1}} m_{i1}!} \left(1 + O(q^{-1}) \right).$$

We have that $\phi_{1,M}(1)$ has leading term involving $q^{-1}$
so we can choose the $q^{-1}$ term from both $\phi_{1,M}(1)$ and $\phi_{2,M}(1)$
or the $q^{-2}$ term from $\phi_{1,M}(1)$ and the constant term from
$\phi_{2,M}(1)$. Clearly we must choose the constant term from each of
$\phi_{i,M}(1)$ for $i \ge 3$.
The former choice corresponds to line four of Table
\ref{q^-2nonparttableM} while the latter choice corresponds to lines five, six
and seven of the table.

The first column
of Table \ref{q^-2nonparttableM} indicates the type of array $M$ used to produce the
$\phi_{i,M}(1)$. If a line contains
exactly one non-$1$ entry then that entry is a certain function of
$m_{k1}$ for some $k$, say $f_k(m_{k1})$ and this line corresponds to
a summand $a_r$ of the coefficient of $q^{-2}$ as given in Lemma
\ref{1/1-t}. The sum $1 + \sum_{r=1}^\infty a_r t^r$ corresponding to
these values of $a_r$ is evaluated using Lemmas \ref{1/1-t} and Lemma \ref{te}
and recorded in the last entry for this line. There are lines in which
all the entries are constant and we produce the generating function by
Corollary \ref{sumparts=1} in these cases.
Each generating function has been multiplied by $t^b$ for the positive
integer $b$ such that the arrays $(m_{ij})$ in question correspond to
partitions of $r-b$.

\begin{table}
\begin{center}
\begin{tabular}{c | c c c | c}
$M$ as in: & $\phi_{1,M}(1)$ & $\phi_{2,M}(1)$ & $\phi_{3+,M}(1)$ & Generating Function \\
\hline
(\ref{nonpart12}) & $1$ & $\frac{1}{2}$ & $1$ &
$\left(\frac{t^4}{2}\right) \frac{1}{1-t}$ \\
(\ref{nonpart32}) & $1$ & $1$ & $1$ & $\left({t^3}\right) \frac{1}{1-t}$ \\
(\ref{nonpart42}) & $\frac{1}{2}$ & $1$ & $1$ &
$\left(\frac{t^4}{2}\right)\frac{1}{1-t}$ \\
(\ref{nonpart22}) & $1$ & $-m_{21}$ & $1$ &
$\left(\frac{-t^4}{2}\right)\frac{1}{1-t}$ \\
(\ref{nonpart22}) & $-\frac{m_{11}^2}{2}$ & $1$ & $1$ &
$\left(-\frac{t^3}{2} - \frac{t^4}{2}\right) \frac{1}{1-t}$  \\

(\ref{nonpart22}) & $\frac{m_{11}}{2}$ & $1$ & $1$ &
$\left( \frac{t^3}{2}\right) \frac{1}{1-t}$  \\

(\ref{nonpart22}) & $1$ & $1$ & $1$ &
$\left( t^2 \right) \frac{1}{1-t}$  \\
\hline
& & & Total & $\left(t^3 + t^2 \right) \frac{1}{1-t}$ \\
\hline
\end{tabular}
\end{center}
\caption{Producing the generating function for the coefficient
of $q^{-2}$ in the expansion of $\sum_M \prod_i \phi_{i,M}(1)$
due to $M \in \mathcal{M}(r)$ corresponding to nonpartitions.}
\label{q^-2nonparttableM}
\end{table}

Adding all the generating functions together for those types of
$M \notin \mathcal{M}_{part}(r)$ gives $\frac{t^3+t^2}{1-t}$.
So the generating function
for the coefficient of $q^{-2}$ is $\frac{t^2}{1-t} = t^2 + t^3 + \ldots$.
Hence

\begin{displaymath}
\sum_{M \in\mathcal{M}(r)} \prod_{i=1}^r \phi_{i,M}(1)
= \left\{
\begin{array} {l l}
1 + q^{-1} + O(q^{-3}) & \textrm{if } r=1 \\
1 + q^{-1} + q^{-2} + O(q^{-3}) & \textrm{if } r\ge 2 \\
\end{array}
\right.
\end{displaymath}

Now that we have the expansions for all three parts we can multiply them
together as in Equation \ref{eqcrMn}.

When $r=1$ we have

$$ \lim_{n\to\infty} \omega_r(n) = 1 - 2q^{-1} + O(q^{-3}), $$

$$ \frac{1-q^{-5}}{(1-q^{-1})(1-q^{-2})} = 1 + q^{-1} + 2q^{-2} + O(q^{-3}), $$

$$ \sum_{M \in\mathcal{M}(r)} \prod_{i=1}^r \phi_{i,M}(1)
= 1 + q^{-1} + O(q^{-3}) $$

\noindent
and hence

$$ \lim_{n\to\infty} c_{\M,r}(n) = 1 - q^{-2} + O(q^{-3}). $$

\noindent
When $r\ge 2$ we have

$$ \lim_{n\to\infty} \omega_r(n) = 1 - 2q^{-1} - q^{-2} + O(q^{-3}), $$

$$ \frac{1-q^{-5}}{(1-q^{-1})(1-q^{-2})} = 1 + q^{-1} + 2q^{-2} + O(q^{-3}), $$

$$ \sum_{M \in\mathcal{M}(r)} \prod_{i=1}^r \phi_{i,M}(1)
= 1 + q^{-1} + q^{-2} + O(q^{-3}) $$

\noindent
and hence

$$ \lim_{n\to\infty} c_{\M,r}(n) = 1 - q^{-2} + O(q^{-3}). $$

\noindent
So for all $r \ge 1$ we have the result.
\end{proof}

It was again a basic assumption of the section that $r$, the dimension
of the invariant subspace, was greater than or equal to $1$.
However, we could have relaxed that assumption, allowed $r$ to be zero
and worked through the section in the same manner.

If we had allowed $r$ to equal zero then
Theorem \ref{CMr=CMblah} would state that $C_{\M,r}(t) = C_{M}(t)$
when $r=0$, since the only matrix in $\mathcal{M}(0)$ is the matrix
containing only zeroes.
Also, in the proof of Theorem \ref{1-q^-2b}, the answer for the
limit of $c_{\M,0}(n)$ would be $1 - q^{-3} + O(q^{-4})$.

These are exactly the answers that should arise, because when $r=0$
the invariant subspace has dimension $0$ and hence the matrix group
$\GL(V)_U$ is equal to $\GL(V)$. By Theorem \ref{limpropM(V)}
the limiting proportion of cyclic matrices in $\GL(V)$ is indeed
$1 - q^{-3} + O(q^{-4})$.

\subsection{Explicit Generating Functions For Small $r$} \label{4.4}

By Theorem \ref{C_r=Cblah} and Theorem \ref{CMr=CMblah}
we can give the generating functions $C_{\GL,r}(t)$ and $C_{\M,r}(t)$, for any $r$,
in terms of Wall's generating functions $C_{\GL}(t)$ and $C_\M(t)$ respectively,
and functions produced by certain procedures.
Using Theorem \ref{limcrn} and Theorem \ref{limcrMn}
we obtain algorithmically
the exact limiting proportion of cyclic matrices in $\GL(V)_U$
and $\M(V)_U$ respectively, for any $r$.
For small $r$ it is possible to calculate
explicitly the generating functions
as well as their associated limiting proportions.

For $r=1$, we derive from Theorem \ref{C_r=Cblah} that

$$ \displaystyle C_{\GL,1}(t) = \displaystyle C_{\GL}(t) \times \frac{(1-q^{-1})t}{1 - q^{-1} + tq^{-2}}.$$

This result was obtained by Jason Fulman \cite[Theorem 14]{Ful2}
as was the result below determining
$\lim_{n\to\infty} c_{\GL,1}(n)$ \cite[Corollary 2]{Ful2}.
However Fulman
derived them in the context of proportions inside a maximal
parabolic subgroup of the larger general linear group $\GL(n+1,q)$.

For matrix algebras with $r=1$ we find

\begin{displaymath}
\begin{array}{rcl}
C_{\M,1}(t) & = & \displaystyle C_\M(t) \times q \times
\frac{tq^{-1}}{1-q^{-1}+tq^{-2}} \\

& = & \displaystyle C_\M(t) \times \frac{t}{1-q^{-1}+tq^{-2}}. \\
\end{array}
\end{displaymath}

By Lemma \ref{1-t}, multiplying the generating functions by $(1-t)$
and evaluating at $t=1$ will give us the limit of the coefficients.
By recalling that $(1-t)C_{\GL}(t)$ evaluated at $t=1$ is
$\frac{1-q^{-5}}{1+q^{-3}}$, we find that $\lim_{n\to\infty} c_{\GL,1}(n)$ equals

$$ \frac{1-q^{-5}}{1+q^{-3}}
\left( \frac{1-q^{-1}}{1 - q^{-1} + q^{-2}} \right) = 1 - q^{-2} -
2q^{-3} + q^{-5} + 3q^{-6} + O(q^{-7}). $$

Then, recalling that $(1-t) C_\M(t)$ evaluated at $t=1$ is
$\frac{1-q^{-5}}{(1-q^{-1})(1-q^{-2})}$ and that $\lim_{n\to\infty}
\omega_1(n) = (1-q^{-1}) \prod_{i=1}^\infty (1-q^{-i})$ we find that
$\lim_{n\to\infty} c_{\M,1}(n)$ equals

\begin{displaymath}
\begin{array}{cl}
\displaystyle
& \displaystyle
\frac{1-q^{-5}}{(1-q^{-1})(1-q^{-2})} \times
\frac{1}{1-q^{-1}+q^{-2}} \times
(1-q^{-1})
\prod_{i=1}^\infty (1-q^{-i}) \\

= & 1 - q^{-2} - 2q^{-3} - q^{-4} + 2q^{-6} + O(q^{-7}). \\

\end{array}
\end{displaymath}

When $r=2$, we find that $\mathcal{M}(2)$ has three elements. They are

\begin{displaymath}
\begin{array} {c c c c}

\left(
\begin{tabular}{cc}
2 & 0 \\
0 & 0 \\
\end{tabular}
\right),

&

\left(
\begin{tabular}{cc}
0 & 1 \\
0 & 0 \\
\end{tabular}
\right)

&

\textrm{and}

&

\left(
\begin{tabular}{cc}
0 & 0 \\
1 & 0 \\
\end{tabular}
\right)

\end{array}
\end{displaymath}

\noindent
which correspond to two distinct monic degree $1$ irreducibles of multiplicity $1$,
one monic degree $1$ irreducible of multiplicity $2$, and one monic degree $2$
irreducible of multiplicity $1$, respectively.

For each element $M$ in $\mathcal{M}(2)$ we calculate $\prod_i
\phi^+_{i,M}(t)$ and then sum over $M$ and multiply
by $C_{\GL}(t)$ to give $C_{\GL,2}(t)$.
Similarly for each element $M$ of $\mathcal{M}(2)$ we calculate $\prod_i
\phi_{i,M}(t)$ and then sum over $M$ and multiply
by $C_\M(t)$ to give $C_{\M,2}(t)$.
This yields that $C_{\GL,2}(t)$ is equal to

$$ C_{\GL}(t) \left(
\frac{t^2(1-q^{-1})(1-2q^{-1})}{2(1-q^{-1}+tq^{-2})^2} +
\frac{t^2(q^{-1}-q^{-2})}{1-q^{-1}+tq^{-2}} +
\frac{t^2(1-q^{-1})}{2(1-q^{-2}+t^2q^{-4})} \right) $$

and $C_{\M,2}(t)$ equal to

$$ C_\M(t) \left(
\frac{t^2(1-q^{-1})}{2(1-q^{-1}+tq^{-2})^2} +
\frac{t^2 q^{-1}}{1-q^{-1}+tq^{-2}} +
\frac{t^2(1-q^{-1})}{2(1-q^{-2}+t^2q^{-4})} \right). $$

Applying Lemma \ref{1-t}, we find that the limiting proportions satisfy
$c_{\GL,2}(\infty)$ equal to

\begin{displaymath}
\begin{array}{c l}

& \displaystyle
\frac{1-q^{-5}}{1+q^{-3}} \left(
\frac{(1-q^{-1})(1-2q^{-1})}{2(1-q^{-1}+q^{-2})^2}
+ \frac{(q^{-1}-q^{-2})}{1-q^{-1}+q^{-2}}
+ \frac{(1-q^{-1})}{2(1-q^{-2}+q^{-4})}
\right)\\

= & 1 - q^{-2} - 3q^{-3} + q^{-4} + 3q^{-5} + 4q^{-6} + O(q^{-7}).\\

\end{array}
\end{displaymath}

\noindent
and $\lim_{n\to\infty} \frac{c_{\M,2}(n)}{\omega_2(n)}$ is equal to

\begin{displaymath}
\begin{array}{l}

\displaystyle
\frac{1-q^{-5}}{(1-q^{-1})(1-q^{-2})}
\left(
\frac{(1-q^{-1})}{2(1-q^{-1}+q^{-2})^2}
+ \frac{q^{-1}}{1-q^{-1}+q^{-2}}
+ \frac{(1-q^{-1})}{2(1-q^{-2}+q^{-4})}
\right)\\

= 1+2q^{-1}+4q^{-2}+3q^{-3} + 3q^{-4} + q^{-5} + 2q^{-6} + O(q^{-7})\\
\end{array}
\end{displaymath}

\noindent
giving

$$ \displaystyle \lim_{n\to\infty} c_{\M,2}(n) =
1-q^{-2}-4q^{-3}-q^{-4}+4q^{-5}+5q^{-6}+O(q^{-7}).$$

\begin{remark} \label{remark4.4}
Equation (\ref{eg4.4}) and Equation (\ref{eg4.4b})
summarise the results of $\lim_{n\to\infty} c_{\GL,r}(n)$
and $\lim_{n\to\infty} c_{\M,r}(n)$ respectively for $r=1,2$.
They include similar calculations performed in
Mathematica \cite{mathematica} for other small values of $r$.
\end{remark}

$$ c_{\GL,r}(\infty) = \lim_{n\to\infty} c_{\GL,r}(n) $$
\begin{equation} \label{eg4.4}
= \left\{
\begin{array}{l l}
1 - q^{-2} - 2q^{-3} \hspace{0.97cm} + \hspace{0.18cm} q^{-5}
+ 3q^{-6} + \hspace{0.18cm} q^{-7} + O(q^{-8}) & \textrm{for } r=1, \\
1 - q^{-2} - 3q^{-3} + q^{-4} + 3q^{-5} + 4q^{-6}
-2q^{-7} + O(q^{-8}) & \textrm{for } r=2, \\
1 - q^{-2} - 3q^{-3} + q^{-4} + 4q^{-5} + 4q^{-6}
-5q^{-7} + O(q^{-8}) & \textrm{for } r=3, \\
1 - q^{-2} - 3q^{-3} + q^{-4} + 4q^{-5} + 4q^{-6}
-6q^{-7} + O(q^{-8}) & \textrm{for } r=4, \\
1 - q^{-2} - 3q^{-3} + q^{-4} + 4q^{-5} + 4q^{-6}
-6q^{-7} + O(q^{-8}) & \textrm{for } r=5, \\
1 - q^{-2} - 3q^{-3} + q^{-4} + 4q^{-5} + 4q^{-6}
-6q^{-7} + O(q^{-8}) & \textrm{for } r=6, \\
1 - q^{-2} - 3q^{-3} + q^{-4} + 4q^{-5} + 4q^{-6}
-6q^{-7} + O(q^{-8}) & \textrm{for } r=7. \\
\end{array}
\right.
\end{equation}

$$c_{\M,r}(\infty) = \lim_{n\to\infty} c_{\M,r}(n) $$
\begin{equation} \label{eg4.4b}
= \left\{
\begin{array}{l l}
1-q^{-2}-2q^{-3}-\hspace{0.17cm}q^{-4}\hspace{1.15cm}
+\hspace{0.18cm}2q^{-6}+\hspace{0.18cm}3q^{-7}+O(q^{-8})
& \textrm{for }r=1,\\
1-q^{-2}-4q^{-3}-\hspace{0.17cm}q^{-4}+4q^{-5}+\hspace{0.18cm}5q^{-6}+\hspace{0.18cm}4q^{-7}+O(q^{-8})
& \textrm{for }r=2,\\
1-q^{-2}-4q^{-3}-3q^{-4}+4q^{-5}+11q^{-6}+\hspace{0.18cm}8q^{-7}+O(q^{-8})
& \textrm{for }r=3,\\
1-q^{-2}-4q^{-3}-3q^{-4}+2q^{-5}+11q^{-6}+14q^{-7}+O(q^{-8})
& \textrm{for }r=4,\\
1-q^{-2}-4q^{-3}-3q^{-4}+2q^{-5}+\hspace{0.18cm}9q^{-6}+14q^{-7}+O(q^{-8})
& \textrm{for }r=5,\\
1-q^{-2}-4q^{-3}-3q^{-4}+2q^{-5}+\hspace{0.18cm}9q^{-6}+12q^{-7}+O(q^{-8})
& \textrm{for }r=6,\\
1-q^{-2}-4q^{-3}-3q^{-4}+2q^{-5}+\hspace{0.18cm}9q^{-6}+12q^{-7}+O(q^{-8})
& \textrm{for }r=7.\\
\end{array}
\right.
\end{equation}

Our results in Equation \ref{eg4.4} agree with Theorem \ref{1-q^-2} which states
that $\lim_{n\to\infty} c_{\GL,r}(n)$ is $1-q^{-2}+O(q^{-3})$ for all $r$.
Similarly the results in Equation \ref{eg4.4b} agree with
Theorem \ref{1-q^-2b} which states that the limit of $c_{\M,r}(n)$
is $1-q^{-2}+O(q^{-3})$ for all $r$.

Although Theorem \ref{1-q^-2} and Theorem \ref{1-q^-2b} give
the limiting proportion for all $r$
this does not imply that $\lim_{n\to\infty} c_{\GL,r}(n)$
or $\lim_{n\to\infty} c_{\M,r}(n)$
are independent of $r$.
In the expansion of $\lim_{n\to\infty} c_{\GL,r}(n)$, the $q^{-2}$ term is independent of $r$
but the later terms do depend on $r$ as Equation \ref{eg4.4} clearly shows.

We believe that for large enough $r$
the coefficient of $q^{-3}$ will eventually become constant. In fact we believe that this is the case for
all terms. Stated formally, we believe that for all
$d \ge 0$ there exists $r_d$ such that the coefficient of $q^{-d}$
in the expansion of $\lim_{n\to\infty} c_{\GL,r}(n)$ as a power series in $q^{-1}$
is the same for all $r \ge r_d$.
By Theorem \ref{1-q^-2} this statement is true for $d = 0,1,2$ but remains to be proved
for all terms in the expansion.

\subsection{A Family of Noncyclic Matrices} \label{4.5}

By Theorem \ref{1-q^-2ch1}, the proportion of cyclic matrices in $\GL(V)_U$, as the
dimension of $V$ tends to infinity, approaches $1-q^{-2} + O(q^{-3})$. Hence there
exists a set $S(V)_U$ of noncyclic matrices in $\GL(V)_U$ with proportion
$\frac{|S(V)_U|}{|\GL(V)_U|}= q^{-2} + O(q^{-3})$.
We construct such a set in this section.

We first choose a basis for $U$ and extend it to be a basis for $V$. Let
$u_1, \ldots, u_r$ be a basis for $U$ and let $u_1, \ldots, u_r, u_{r+1},
\ldots, u_n$ be a basis for $V$. Let $U_0 = \langle u_2, \ldots, u_r \rangle$
and let $V_0 = \langle u_1, \ldots, u_{n-1} \rangle$. For a fixed
$\lambda \in \mathbb{F}_q$ let

\begin{equation} \label{T1}
\mathcal{T}_\lambda = \left\{
\begin{array}{c | c}
T \in \GL(V)_U &

\begin{array}{l l l}
(t-\lambda)^3 \nmid c_T(t),\\
\exists w_1\in u_1 + U_0\textrm{ such that } w_1T =\lambda w_1 \textrm{ and }\\
\exists w_2\in u_n + V_0\textrm{ such that } w_2T =\lambda w_2\\
\end{array}\\

\end{array}
\right\}.
\end{equation}

Each of the matrices in $\mathcal{T}_\lambda$ is noncyclic since the $(t-\lambda)$-primary
component is the direct sum of two cyclic summands.
We prove some lemmas which will aid us in proving the upcoming
theorem which ultimately exhibits the set of noncyclic matrices in $\GL(V)_U$ with proportion
$q^{-2} + O(q^{-3})$ that we are looking for.

\begin{lemma} \label{sizeT1}
For $\lambda \in \mathbb{F}_q^*$,

$$| \mathcal{T}_\lambda | = q^{n^2 + r^2 - rn - 3}+O(q^{n^2 + r^2 - rn - 4}).$$
\end{lemma}

\begin{proof}
We count the number of
triples $(w_1, w_2, T)$, where $T$ is a matrix in
$\mathcal{T}_\lambda$, and $w_1$ and $w_2$ are $\lambda$-eigenvectors for $T$ satisfying
the criteria for $\mathcal{T}_\lambda$ in Equation (\ref{T1}).
Since $w_1 \in u_1 + U_0 \subseteq U$
and $w_2 \notin U$, the $\lambda$-eigenspace of $T$ is $W=\langle w_1,w_2 \rangle$,
$W \cap U = \langle w_1 \rangle$ and $W \cap (u_1 + U_0) = \{ w_1 \}$.

Suppose $T \in \mathcal{T}_\lambda$ and let
$W$ be the $\lambda$-eigenspace of $T$. Then $W = \langle w_1,
w_2 \rangle$ for some $\lambda$-eigenvectors $w_1 \in u_1 + U_0$
and $w_2 \in u_n + V_0$.
As explained in the previous paragraph,
$w_1$ is determined uniquely by $T$.

Now let $w_2'$ be an alternative choice for the second basis vector,
that is, suppose that $w_2'T = \lambda w_2'$ with $w_2' \in u_n + V_0$ and $w_2' \ne w_2$.
Then $0 \ne w_2 - w_2' \in W \cap V_0$, and as $w_2 \in W \setminus V_0$ and
$\dim (W) = 2$, it follows that $W \cap V_0 = \langle w_2 - w_2' \rangle.$
However, $w_1 \in W \cap U \subseteq W \cap V_0$.
Hence $w_2 - w_2' = aw_1$ for some $a \in \F_q$.
Moreover for each $a \in \F_q$, $w_2'=aw_1+w_2$ lies both in $W$ and in $u_n + V_0$.
Hence there are $q$ choices for the second basis vector for $W$ corresponding to
the $q$ choices for $a$.

We will continue the count of triples $(w_1, w_2, T)$ but take this little
break to note that $|\mathcal{T}_\lambda |$ is equal to the number of triples
$(w_1, w_2, T)$ divided by $q$ because each $T$ has $q$ possibilities for the pair
$(w_1,w_2)$.

First we choose vectors $w_1 \in u_1 + V_0$ and $w_2 \in u_n + V_0$. Note that
$w_1, u_2, \ldots, u_{n-1}, w_2$ is also a basis for $V$. For each $T$ occurring in a
triple $(w_1, w_2, T)$ with these chosen vectors, the matrix representing $T$ with
respect to the basis $w_1, u_2, \ldots, u_{n-1}, w_2$ must be of the form

$$
\left(
\begin{array}{c c c c | c c c c}
\lambda & 0 & \ldots & 0 & 0 & \ldots & \ldots & 0 \\
* & & &  &            0 & \ldots & \ldots & 0 \\
\vdots & & A & &           0 & \ldots & \ldots & 0 \\
* & & & &             0 & \ldots & \ldots & 0 \\
\hline
* & \ldots & \ldots & * & & & & * \\
* & \ldots & \ldots & * & & B & & \vdots \\
* & \ldots & \ldots & * & & & & * \\
0 & \ldots & \ldots & 0 & 0 & \ldots & 0 & \lambda \\
\end{array}
\right)
$$

\noindent
where $*$ represents any value from the field and $A$ and $B$ are
$(r-1)\times(r-1)$ and $(n-r-1)\times(n-r-1)$ invertible matrices respectively,
neither of which has $t$ nor $t-\lambda$ as a factor of their characteristic
polynomial.

It is not difficult to show that there are $N(r) \ge q^{(r-1)^2} - O(q^{(r-1)^2-1})$ choices for the
matrix $A$ and $N(n-r) \ge q^{(n-r-1)^2} - O(q^{(n-r-1)^2-1})$ choices for the
matrix $B$ (a detailed proof is given in \cite[Lemma 2.1.5]{myPhD}).
Then there are $n-2+r(n-r-1)$ entries $*$ and there are $q^{n-2+r(n-r-1)}$
choices for these $*$ entries. Hence there are $N(r)N(n-r) q^{n-2+r(n-r-1)} \ge
q^{n^2+r^2-nr-n-r} + O(q^{n^2+r^2-nr-n-r-1})$ such matrices $T$ for a
given $w_1$ and $w_2$.

Then to finish our count of the number of triples $(w_1, w_2, T)$ we need to
multiply by the number of such $w_1$ and $w_2$ which is $q^{r-1}$ and $q^{n-1}$
respectively. Hence the number of triples $(w_1, w_2, T)$ is
$q^{n^2 + r^2 - rn - 2} + O(q^{n^2 + r^2 - rn - 3})$ and dividing
by $q$ gives
$|\mathcal{T}_\lambda | = q^{n^2 + r^2 - rn - 3} + O(q^{n^2 + r^2 - rn - 4})$.
\end{proof}

\begin{lemma} \label{sizeT1cap}
For $\lambda, \gamma \in \mathbb{F}_q^*$ with $\lambda \ne \gamma$,

$$ | \mathcal{T}_\lambda \cap \mathcal{T}_\gamma | \le q^{n^2 + r^2 - nr - 6}.$$
\end{lemma}

\begin{proof}
Let $T \in \mathcal{T}_\lambda \cap \mathcal{T}_\gamma$ for
$\lambda \ne \gamma$. Since $T \in \mathcal{T}_\lambda$
there exists $w_1 \in u_1 + U_0$ such that $w_1 T = \lambda w_1$ and
there exists $w_2 \in u_n + V_0$ such that $w_2 T = \lambda w_2$.
Similarly since $T \in \mathcal{T}_\gamma$
there exists $x_1 \in u_1 + U_0$ such that $x_1 T = \gamma x_1$ and
there exists $x_2 \in u_n + V_0$ such that $x_2 T = \gamma x_2$.
Since $\langle w_1, w_2 \rangle, \langle x_1, x_2 \rangle$ lie in distinct
eigenspaces for $T$,
$\langle w_1, w_2 \rangle \cap \langle x_1, x_2 \rangle = 0$ so
$\{ w_1, w_2, x_1, x_2 \}$ is linearly independent.

Taking a similar approach to that of Lemma \ref{sizeT1}, we count the number of
tuples of the form $(w_1, w_2, x_1, x_2, T)$ and relate that back to the size of
$| \mathcal{T}_\lambda \cap \mathcal{T}_\gamma |$.

From the proof of Lemma \ref{sizeT1}, we know that for a given
$T \in \mathcal{T}_\lambda \cap \mathcal{T}_\gamma$,
$w_1$ is uniquely determined and there are $q$ choices for $w_2$.
Similarly, $x_1$ is uniquely determined and there are $q$ choices for $x_2$. So
$| \mathcal{T}_\lambda \cap \mathcal{T}_\gamma |$ equals the number of
tuples $(w_1, w_2, x_1, x_2, T)$ divided by $q^2$.

There are $q^{r-1}$ choices for $w_1 \in u_1 + U_0$. For a given $w_1$, since
$x_1 \in u_1 + U_0 = w_1 + U_0$, there are $q^{r-1}-1$ choices for $x_1$ and the action
of $T$ on $\langle w_1, x_1 \rangle$ is determined uniquely. In particular,
$u_0 = x_1 - w_1 \in U$ and $u_0 T = \gamma x_1 - \lambda w_1 = \gamma u_0 + (\gamma - \lambda) w_1.$
Similarly there are $q^{n-1}$ choices for $w_2 \in u_n + V_0$, and for a given
$w_2$, there are $q^{n-1}-1$ choices for $x_2$ and the action of $T$ on
$\langle w_1, w_2, x_1, x_2 \rangle$ is uniquely determined.
In particular, $v_0 = x_2 - w_2 \in V_0$ and
$v_0T = \gamma x_2 - \lambda w_2 = \gamma v_0 + (\gamma - \lambda) w_2.$

Suppose $(w_1, x_1, w_2, x_2)$ are as above and $T$ occurs with them in a tuple.
Since $w_1, u_0, v_0, v_2$ are linearly independent we can extend them to a
basis $w_1, u_0, y_3, \ldots, y_{n-2}, v_0, w_2$ for $V$ such that
$w_1, u_0, y_3, \ldots, y_r$ is a basis for $U$. With respect to this basis
$T$ must be of the form

$$
\left(
\begin{array}{c c c c c | c c c c c}
\lambda & 0 & \ldots & \ldots & 0 &      0 & & \ldots & & 0 \\
\gamma - \lambda & \gamma & 0 & \ldots & 0 & 0 & & \ldots & & 0 \\
* & * & & & &                             0 &  & \ldots & & 0 \\
\vdots & \vdots & & A & &                   0 &  & \ldots & & 0 \\
* & * & & & &                             0 &  & \ldots & & 0 \\
\hline
* & & \ldots & & * &     & & & * & * \\
* & &  \ldots & & * &   & B & & \vdots & \vdots \\
* & & \ldots  & & * &   & & & * & * \\
0 & & \ldots & & 0 &     0 & \ldots & 0 & \gamma & \gamma-\lambda \\
0 & & \ldots & & 0 &     0 & \ldots & 0 & 0 & \lambda \\
\end{array}
\right)
$$

\noindent
where $*$ represents any value from the field and $A$ and $B$ are
$(r-2)\times(r-2)$ and $(n-r-2)\times(n-r-2)$ invertible matrices respectively.
We this time allow $A$ and $B$ to have $t-\lambda$ and $t-\gamma$ as factors
for simplicity despite it being an over count.

There are less than $q^{(r-2)^2}$ such matrices $A$,
there are less than $q^{(n-r-2)^2}$ such matrices $B$ and there
are $q^{2n-8+r(n-r-2)}$ possible values for the $*$ entries. When we multiply
these together along with the $q^{r-1} (q^{r-1} - 1) q^{n-1} (q^{n-1} - 1)$
choices for $(w_1, w_2, x_1, x_2)$,
we get the number of tuples $(w_1, w_2, x_1, x_2, T)$ to be less than
$q^{n^2+r^2-nr-4}$. Hence, after dividing by $q^2$, we get

$$ | \mathcal{T}_\lambda \cap \mathcal{T}_\gamma | \le q^{n^2 + r^2 - nr - 6}.$$
\end{proof}

We now state and prove Theorem \ref{noncyc} which exhibits a
family of noncyclic matrices in $\GL(V)_U$ with
proportion $q^{-2} + O(q^{-3})$.

\begin{theorem} \label{noncyc}
The union of $\mathcal{T}_\lambda$ for $\lambda \in \F_q^*$ forms a set of
noncyclic matrices in $\GL(V)_{U}$ with limiting proportion $q^{-2} + O(q^{-3})$
as $n$ tends to infinity.
\end{theorem}

\begin{proof}
We cannot simply take the sum of $| \mathcal{T}_\lambda |$ over all
$\lambda$ to obtain $| \cup \mathcal{T}_\lambda |$ since some matrices belong
to more than one $\mathcal{T}_\lambda$ but we
can subtract off the number of those matrices which belong to two or more
$\mathcal{T}_\lambda$ to give

$$ | \cup_{\lambda \in \mathbb{F}_q^*} \mathcal{T}_\lambda | \ge
(q-1) | \mathcal{T}_\lambda | - {q-1 \choose 2}
| \mathcal{T}_\lambda \cap \mathcal{T}_\gamma | $$

\noindent
by the Principle of Inclusion and Exclusion. By Lemma \ref{sizeT1} and Lemma \ref{sizeT1cap},

$$ | \cup_{\lambda \in \mathbb{F}_q^*} \mathcal{T}_\lambda | \ge
q^{n^2+r^2-nr-2} + O(q^{n^2+r^2-nr-3}). $$

We then simply divide this by $|\GL(V)_U| = q^{n^2 + r^2 - nr} + O(q^{n^2+r^2-nr-1})$ to give

$$\frac{| \cup_{\lambda \in \mathbb{F}_q^*} \mathcal{T}_\lambda |}
{|\GL(V)_U|} = q^{-2} + O(q^{-3}).$$
\end{proof}

\appendix

\end{document}